\documentclass[reqno]{amsart}

\usepackage{amsmath}
\usepackage{amsfonts}
\usepackage{amssymb,enumerate}
\usepackage{amsthm}
\usepackage[all]{xy}
\usepackage{rotating}
\usepackage{hyperref}
\usepackage[all]{mathtools}

\theoremstyle{plain}
\newtheorem{lem}{Lemma}[section]
\newtheorem{cor}[lem]{Corollary}

\newtheorem{thm}[lem]{Theorem}

\theoremstyle{definition}
\newtheorem{defn}[lem]{Definition}
\newtheorem{ex}[lem]{Example}
\newtheorem{question}[lem]{Question}

\newtheorem{disc}[lem]{Remark}
\newtheorem{anal}[lem]{Discussion}

\newtheorem{notn}[lem]{Notation}
\newtheorem{fact}[lem]{Fact}
\newtheorem{para}[lem]{}
\newtheorem{exer}[lem]{Exercise}

\newtheorem{notation}[lem]{Notation}

\newtheorem*{Convention}{Convention}
\newtheorem*{assumption}{Assumption}

\newcommand{\pd}{\operatorname{pd}}

\newcommand{\id}{\operatorname{id}}

\newcommand{\depth}{\operatorname{depth}}	
\newcommand{\rank}{\operatorname{rank}}

\newcommand{\len}{\operatorname{len}}

\newcommand{\grade}{\operatorname{grade}}

\newcommand{\HH}{\operatorname{H}}
\newcommand{\Hom}{\operatorname{Hom}}	
\newcommand{\coker}{\operatorname{Coker}}
\newcommand{\spec}{\operatorname{Spec}}
\newcommand{\s}{\mathfrak{S}}

\newcommand{\im}{\operatorname{Im}}
\newcommand{\shift}{\mathsf{\Sigma}}

\newcommand{\Ker}{\operatorname{Ker}}

\newcommand{\ideal}[1]{\mathfrak{#1}}
\newcommand{\m}{\ideal{m}}
\newcommand{\n}{\ideal{n}}

\newcommand{\fm}{\ideal{m}}

\newcommand{\fa}{\ideal{a}}

\newcommand{\comp}[1]{\widehat{#1}}
\newcommand{\ol}{\overline}
\newcommand{\wti}{\widetilde}

\newcommand{\ass}{\operatorname{Ass}}

\newcommand{\supp}{\operatorname{Supp}}

\newcommand{\bbz}{\mathbb{Z}}
\newcommand{\bbn}{\mathbb{N}}

\newcommand{\from}{\leftarrow}
\newcommand{\xra}{\xrightarrow}
\newcommand{\xla}{\xleftarrow}

\newcommand{\into}{\hookrightarrow}

\newcommand{\vf}{\varphi}

\newcommand{\y}{\mathbf{y}}

\newcommand{\x}{\mathbf{x}}

\newcommand{\Mod}{\operatorname{Mod}}

\renewcommand{\geq}{\geqslant}
\renewcommand{\leq}{\leqslant}
\renewcommand{\ker}{\Ker}

\newcommand{\Ext}[4][R]{\operatorname{Ext}_{#1}^{#2}(#3,#4)}	
	
\newcommand{\Lotimes}[3][R]{#2\otimes^{\mathbf{L}}_{#1}#3}
\newcommand{\Otimes}[3][R]{#2\otimes_{#1}#3}
\renewcommand{\Hom}[3][R]{\operatorname{Hom}_{#1}(#2,#3)}	
\newcommand{\Tor}[4][R]{\operatorname{Tor}^{#1}_{#2}(#3,#4)}

\newcommand{\ssm}{\smallsetminus}

\newcommand{\und}[1]{#1^{\natural}}
\newcommand{\mult}[2]{\mu^{#1,#2}}
\newcommand{\HomA}[2]{\operatorname{Hom}_{A}(#1,#2)}

\newcommand{\od}{\operatorname{Mod}}
\newcommand{\gl}{\operatorname{\underline{GL}}}
\newcommand{\ggl}{\operatorname{GL}}
\newcommand{\ta}{\operatorname{\textsf{T}}}

\newcommand{\en}{\operatorname{End}}

\newcommand{\au}{\operatorname{Aut}}
\newcommand{\yext}[4][R]{\operatorname{YExt}_{#1}^{#2}(#3,#4)}	

\newcommand{\tamod}{\ta^{\od^U(W)}}
\newcommand{\taglm}{\ta^{\ggl(W)_0\cdot M}}

\newcommand{\bba}{\mathbb{A}}

\newcommand{\signum}[1]{\operatorname{sgn}(#1)}

\numberwithin{equation}{lem}

\begin{document}

\bibliographystyle{amsplain}

\author{Kristen A. Beck}

\address{Kristen A. Beck, Department of Mathematics,
The University of Arizona,
617 N. Santa Rita Ave.,
P.O. Box 210089,
Tucson, AZ 85721,
USA}

\email{kbeck@math.arizona.edu}

\urladdr{http://math.arizona.edu/\~{}kbeck/}

\author{Sean Sather-Wagstaff}

\address{Sean Sather-Wagstaff, Department of Mathematics,
NDSU Dept \# 2750,
PO Box 6050,
Fargo, ND 58108-6050
USA}

\email{sean.sather-wagstaff@ndsu.edu}

\urladdr{http://www.ndsu.edu/pubweb/\~{}ssatherw/}

\title[DG commutative algebra]
{A somewhat gentle introduction to  differential graded commutative algebra}

\date{\today}

\thanks{
Sean Sather-Wagstaff was supported in part by a grant from the NSA}

\dedicatory{Dedicated with much respect to Tony Geramita}



\begin{abstract}
Differential graded (DG) commutative algebra provides powerful techniques for
proving theorems about modules over commutative rings. 
These notes are a somewhat colloquial introduction to these techniques.
In order to provide some motivation
for commutative algebraists who are wondering about the benefits of
learning and using these techniques,
we present them in the context of a recent result of Nasseh and Sather-Wagstaff.
These  notes were used for the course ``Differential Graded Commutative Algebra''
that was part of the
Workshop on Connections Between Algebra and Geometry
held at the University of Regina, May 29--June 1, 2012.
\end{abstract}

\maketitle

\setcounter{tocdepth}{1}
\tableofcontents

\section{Introduction} \label{sec0}

\begin{Convention}
The term ``ring'' is short for ``commutative noetherian ring with identity'',
and ``module'' is short for ``unital module''.
Let $R$ be a ring.
\end{Convention}

These are notes for the course ``Differential Graded Commutative Algebra''
that was part of the
Workshop on Connections Between Algebra and Geometry
held at the University of Regina, May 29--June 1, 2012.
They represent our attempt to provide a small amount of (1) motivation
for commutative algebraists who are wondering about the benefits of
learning and using Differential Graded (DG) techniques, and (2)
actual DG techniques. 

\subsection*{DG Algebra} 
DG commutative algebra provides a useful framework for proving
theorems about rings and modules, the statements of which have no reference to
the DG universe. 
For instance, a standard theorem 
says the following: 

\begin{thm}[\protect{\cite[Corollary 1]{foxby:mirufbc}}]
\label{thm120521a}
Let $(R,\m)\to (S,\n)$ be a flat local ring homomorphism, that is, a ring homomorphism
making $S$ into a flat $R$-module such that $\m S\subseteq \n$.
Then $S$ is Gorenstein if and only if $R$ and $S/\m S$ are Gorenstein.
Moreover, there is an equality of Bass series
$I_S(t)=I_R(t)I_{S/\m S}(t)$.
\end{thm}

(See Definition~\ref{defn120521b} for the term ``Bass series''.)
Of course, the flat hypothesis is very important here. 
On the other hand, the use of DG algebras allows for a slight 
(or vast, depending on your perspective) improvement of this:

\begin{thm}[\protect{\cite[Theorem A]{avramov:bsolrhoffd}}]
\label{thm120521b}
Let $(R,\m)\to (S,\n)$ be a  local ring homomorphism of finite flat dimension, that is, a 
local ring homomorphism
such that $S$ has a bounded resolution by flat $R$-module.
Then there is a formal Laurent series $I_{\vf}(t)$ with non-negative integer
coefficients such that
$I_S(t)=I_R(t)I_{\vf}(t)$.
In particular, if $S$ is Gorenstein, then so is $R$.
\end{thm}

In this result, the series $I_{\vf}(t)$ is the \emph{Bass series} of $\vf$.
It is the Bass series of the ``homotopy closed fibre'' of $\vf$
(instead of the usual closed fibre $S/\m S$ of $\vf$ that is used in 
Theorem~\ref{thm120521a}) which is the commutative DG algebra
$\Lotimes S{R/\m}$. In particular, when $S$ is flat over $R$,
this is the usual closed fibre $S/\m S\cong \Otimes S{R/\m}$,
so one recovers Theorem~\ref{thm120521a} as a corollary of
Theorem~\ref{thm120521b}.

Furthermore, DG algebra comes equipped with constructions 
that can be used to replace your given ring with one that is nicer in some sense.
To see how this works,
consider the following strategy for using  completions.

To prove a theorem about a given local ring 
$R$, first show that the assumptions ascend
to the completion $\comp R$,  prove the result for the complete ring $\comp R$,
and  show how the  conclusion for $\comp R$ implies the
desired conclusion for $R$. This technique is useful since 
frequently $\comp R$ is nicer then $R$. For instance, $\comp R$
is a homomorphic image of a power series ring over a field or a complete
discrete valuation ring, so
it is universally catenary (whatever that means)
and it has a dualizing complex (whatever that is),
while the original ring $R$ may not have either of these properties.

When $R$ is Cohen-Macaulay and local, a similar strategy sometimes allows one to mod out by a 
maximal $R$-regular sequence $\x$ to  assume that $R$ is artinian. The regular sequence
assumption is like the flat condition for $\comp R$ in that it (sometimes) allows for the
transfer of hypotheses and conclusions between $R$ and the quotient
$\ol R:=R/(\x)$. The artinian
hypothesis is particularly nice, for instance, when $R$ contains a field because then 
$\ol R$ is a finite dimensional algebra over a field.

The DG universe contains a  construction $\wti R$ that is
similar $\ol R$, with an advantage and a disadvantage.
The advantage is that it is more flexible than $\ol R$ because it does
not require the ring to be Cohen-Macaulay, and it produces a finite dimensional
algebra over a field regardless of whether or not $R$ contains a field.
The disadvantage is that $\wti R$ is a  DG algebra
instead of just an algebra, so it is graded commutative (almost, but not quite, commutative)
and there is a bit more data
to  track  when working with $\wti R$. 
However, the advantages outweigh the disadvantages in that $\wti R$
allows us to prove results for arbitrary local rings that can only be proved
(as we understand things today) in special cases using $\ol R$.
One such  result is the following:

\begin{thm}[\protect{\cite[Theorem A]{nasseh:lrfsdc}}]
\label{intthm120521a}
A local ring has only finitely many semidualizing
modules up to isomorphism.
\end{thm}

Even if you don't know what a semidualizing module is,
you can see the point. Without DG techniques, we only know how to
prove this result for Cohen-Macaulay rings that contain a field;
see Theorem~\ref{intthm120522b}. With DG techniques, you get the
unqualified result, which answers a question of Vasconcelos~\cite{vasconcelos:dtmc}.

\subsection*{What These Notes Are} 
Essentially, these notes contain a sketch of the 
proof of Theorem~\ref{intthm120521a};
see~\ref{proof120523b}, \ref{proof120523a}, and~\ref{proof120523c} below.
Along the way, we provide a big-picture view of some of the
tools and techniques in DG algebra (and other areas) needed to 
get a basic understanding of this proof.
Also, since our motivation comes from the study of semidualizing modules,
we  provide a bit of motivation for the study of those gadgets in Appendix~\ref{sec120522c}.
In particular, we do not assume that the reader is familiar
with the semidualizing world.

Since these notes are based on a course, they contain many exercises;
sketches of solutions are contained in Appendix~\ref{sec130212a}.
They also contain a number of examples and facts that are presented without proof.
A diligent reader may also wish to consider many of these as exercises.

\subsection*{What These Notes Are Not} 
These notes do not contain a great number of details about the 
tools and techniques in DG algebra. There are already excellent sources
available for this, particularly, the 
seminal works~\cite{avramov:ifr, avramov:dgha, avramov:tlg}. 
The interested reader is encouraged to dig into these sources
for proofs and constructions not given here. 
Our goal is to give some idea of what the tools look like
and how they get used to solve problems.
(To help  readers in their digging, we provide many
references for properties that we use.)

\subsection*{Notation}
When it is convenient, we use notation from~\cite{bruns:cmr,matsumura:crt}.
Here we specify our conventions for some notions that have several notations:

$\pd_R(M)$: projective dimension of an $R$-module $M$

$\id_R(M)$: injective dimension of an $R$-module $M$

$\len_R(M)$: length of an $R$-module $M$

$S_n$: the symmetric group on $\{1,\ldots,n\}$.

$\operatorname{sgn}(\iota)$: the signum of an element $\iota\in S_n$.

\subsection*{Acknowledgements}
We are grateful to the Department of Mathematics and Statistics at the University of Regina
for its hospitality during the Workshop on Connections Between Algebra and Geometry.
We are also grateful to the organizers and participants for providing such a stimulating environment,
and to Saeed Nasseh for carefully reading parts of this manuscript.

\section{Semidualizing Modules}
\label{sec120521b}

This section contains background material on semidualizing modules.
It also contains a special case of Theorem~\ref{intthm120521a};
see Theorem~\ref{intthm120522b}.
Further survey material can be found in~\cite{sather:sdm, sather:bnsc} and Appendix~\ref{sec120522c}.

\begin{defn}\label{defn120521a}
A finitely generated $R$-module $C$ is \emph{semidualizing}
if the natural homothety map $\chi^R_C\colon R\to\Hom CC$ given by $r\mapsto[c\mapsto rc]$ is an isomorphism
and $\Ext iCC=0$ for all $i\geq 1$. 
A \emph{dualizing} $R$-module is a semidualizing $R$-module
such that $\id_R(C)<\infty$.
The set of isomorphism classes of semidualizing $R$-modules
is denoted $\s_0(R)$.
\end{defn}

\begin{disc}\label{disc120522c}
The symbol $\s$ is an S, as in \verb+\mathfrak{S}+.
\end{disc}

\begin{ex}\label{ex130331}
The free $R$-module $R^1$ is semidualizing.
\end{ex}

\begin{fact}\label{disc120521a}
The ring $R$ has a dualizing module if and only if it is Cohen-Macaulay
and a homomorphic image of a Gorenstein ring; when these conditions
are satisfied, a dualizing $R$-module is the same as a ``canonical'' $R$-module.
See Foxby~\cite[Theorem 4.1]{foxby:gmarm}, Reiten~\cite[(3) Theorem]{reiten:ctsgm}, and Sharp~\cite[2.1 Theorem (i)]{sharp:fgmfidccmr}.
\end{fact}

\begin{disc}\label{disc120521b}
To the best of our knowledge, semidualizing modules were first introduced
by Foxby~\cite{foxby:gmarm}. They have been rediscovered
independently by several authors who seem to all use different terminology
for them. A few examples of this, presented chronologically, are:
\begin{center}
\begin{tabular}{lll}
Author(s)& Terminology & Context\\ \hline
Foxby~\cite{foxby:gmarm}& ``PG-module of rank 1''
& commutative algebra\\
Vasconcelos~\cite{vasconcelos:dtmc}& ``spherical module''
& commutative algebra\\
Golod~\cite{golod:gdagpi}& ``suitable module'' \footnote{Apparently, 
another translation of the Russian term
Golod used is ``comfortable'' module.}
& commutative algebra\\
Wakamatsu~\cite{wakamatsu:mtse}& ``generalized tilting module''
& representation theory\\
Christensen~\cite{christensen:scatac}& ``semidualizing module''
& commutative algebra\\
Mantese and Reiten~\cite{mantese:wtm}& ``Wakamatsu tilting module''
& representation theory
\end{tabular}
\end{center}
\end{disc}

The following facts are quite useful in practice.

\begin{fact}\label{fact120521a}
Assume that $(R,\m)$ is local, and let $C$ be a semidualizing $R$-module.
If $R$ is Gorenstein, then $C\cong R$.
The converse holds if $C$ is a dualizing $R$-module.
See~\cite[(8.6) Corollary]{christensen:scatac}.

If $\pd_R(C)<\infty$,
then $C\cong R$ by~\cite[Fact 1.14]{sather:bnsc}.
Here is a sketch of the proof. The isomorphism 
$\Hom CC\cong R$ implies that $\supp_R(C)=\spec(R)$ and
$\ass_R(C)=\ass_R(R)$. In particular, an element $x\in \m$ is
$C$-regular if and only if it is $R$-regular. 
If $x$ is $R$-regular, it follows that $C/xC$ is semidualizing over $R/xR$.
By induction on
$\depth(R)$, we conclude that $\depth_R(C)=\depth(R)$.
The Auslander-Buchsbaum formula implies that $C$ is projective,
so it is free since $R$ is local. Finally, the isomorphism $\Hom CC\cong R$ implies that
$C$ is free of rank 1, that is, $C\cong R$.
\end{fact}

\begin{fact}\label{fact120522a}
Let $\vf\colon R\to S$ be a ring homomorphism of finite flat dimension.
(For example, this is so if $\vf$  is flat or surjective with kernel generated by
an $R$-regular sequence.)
If $C$ is a semidualizing $R$-module, then $\Otimes{S}{C}$ is 
a semidualizing $S$-module. The converse holds when
$\vf$ is faithfully flat or local. 
The functor $\Otimes S-$ induces a well-defined function
$\s_0(R)\to\s_0(S)$ which is injective when $\vf$ is local.
See~\cite[Theorems 4.5 and 4.9]{frankild:rrhffd}.
\end{fact}

\begin{exer}\label{exer130218a}
Verify the conclusions of Fact~\ref{fact120522a} when $\vf$ is flat.
That is, let $\vf\colon R\to S$ be a flat ring homomorphism.
Prove that if $C$ is a semidualizing $R$-module, then $\Otimes{S}{C}$ is 
a semidualizing $S$-module. Prove that the converse holds when
$\vf$ is faithfully flat, e.g., when $\vf$ is local. 
\end{exer}

The next lemma is  for use in Theorem~\ref{intthm120522b}, which is a special case of
Theorem~\ref{intthm120521a}. See Remark~\ref{disc130331a} and Question~\ref{q130331a}
for further perspective.

\begin{lem}\label{lem120522a}
Assume that $R$ is local and artinian. Then there is an integer
$\rho$ depending only on $R$ such that 
$\len_R(C)\leq \rho$ for every semidualizing $R$-module $C$. 
\end{lem}

\begin{proof}
Let $k$ denote the residue field of $R$.
We show that the integer $\rho=\len_R(R)\mu^0_R$
satisfies the conclusion where $\mu^0_R=\rank_k(\Hom kR)$.
(This is the 0th Bass number
of $R$; see Definition~\ref{defn120521b}.)
Let $C$ be a semidualizing $R$-module.
Set $\beta=\rank_k(\Otimes kC)$
and $\mu=\rank_k(\Hom kC)$.
Since $R$ is artinian and $C$ is finitely generated, it follows that $\mu\geq 1$.
Also, the fact that $R$ is local implies that
there is an $R$-module epimorphism $R^\beta\to C$,
so we have $\len_R(C)\leq\len_R(R)\beta$.
Thus, it remains to show that $\beta\leq\mu^0_R$.

The following sequence of isomorphisms uses  adjointness
and tensor cancellation:
\begin{align*}
k^{\mu^0_R}
&\cong\Hom{k}{R}\\
&\cong\Hom{k}{\Hom{C}{C}}\\
&\cong\Hom{\Otimes Ck}{C}\\
&\cong\Hom{\Otimes[k]{k}{(\Otimes Ck)}}{C}\\
&\cong\Hom[k]{\Otimes Ck}{\Hom kC}\\
&\cong\Hom[k]{k^\beta}{k^\mu}\\
&\cong k^{\beta\mu}.
\end{align*}
Since $\mu\geq 1$, it follows that
$\beta\leq\beta\mu=\mu^0_R$,
as desired.
\end{proof}

\begin{disc}
\label{disc130331a}
Assume that $R$ is local and Cohen-Macaulay.
If $D$ is dualizing for $R$, then there is an equality $e_R(D)=e(R)$
of Hilbert-Samuel multiplicites with respect to the maximal ideal of $R$.
See, e.g., \cite[Proposition 3.2.12.e.i and Corollary 4.7.8]{bruns:cmr}. 
It is unknown whether the same equality holds
for an arbitrary \emph{semi}dualizing $R$-module. Using the
``additivity formula'' for multiplicities, this boils down to the following.
Some progress is contained in~\cite{cooper:mesm}.
\end{disc}

\begin{question}
\label{q130331a}
Assume that $R$ is local and artinian. 
For every semidualizing $R$-module $C$, must one have
$\len_R(C)=\len(R)$? 
\end{question}

While we are in the mood for questions, here is a big one.
In every explicit calculation of $\s_0(R)$, the answer is ``yes'';
see~\cite{sather:divisor,sather:lbnsc}.

\begin{question}
Assume that $R$ is local. Must $|\s_0(R)|$ be $2^n$ for some $n\in\bbn$?
\end{question}

Next, we sketch the proof of 
Theorem~\ref{intthm120521a} when $R$ is Cohen-Macaulay and contains a field. 
This sketch serves to guide the proof of the result in general.

\begin{thm}[\protect{\cite[Theorem 1]{christensen:cmafsdm}}]
\label{intthm120522b}
Assume that $(R,\m,k)$ is Cohen-Macaulay local and contains a field.
Then $|\s_0(R)|<\infty$.
\end{thm}

\begin{proof}
Case 1: $R$ is artinian, and $k$ is algebraically closed. 
In this case, Cohen's structure theorem implies that $R$ is a 
finite dimensional $k$-algebra.
Since $k$ is algebraically closed,
a result of Happel~\cite[proof of first proposition in section 3]{happel:sm} 
says that for each $n\in\bbn$ the following set is finite.
$$
T_n=\{\text{isomorphism classes of $R$-modules $N$}\mid
\text{$\Ext 1NN=0$ and $\len_R(N)=n$}\}$$
Lemma~\ref{lem120522a} implies that there is a $\rho\in\bbn$ such that
$\s_0(R)$ is contained in the finite set $\bigcup_{n=1}^\rho T_n$, so $\s_0(R)$ is finite.

Case 2: $k$ is algebraically closed.
Let $\x=x_1,\ldots,x_n\in\m$ be a maximal $R$-regular sequence.
Since $R$ is Cohen-Macaulay, the quotient $R'=R/(\x)$ is artinian.
Also, $R'$ has the same residue field as $R$, so Case 1 implies that
$\s_0(R')$ is finite. 
Since $R$ is local, Fact~\ref{fact120522a} provides an injection
$\s_0(R)\into\s_0(R')$, so $\s_0(R)$ is finite as well.

Case 3: the general case.
A result of Grothendieck~\cite[Th\'{e}or\`{e}m 19.8.2(ii)]{grothendieck:ega4-1}
provides a flat local ring homomorphism
$R\to \ol R$ such that $\ol R/\m \ol R$ is algebraically closed.
In particular, since $R$ and $\ol R/\m \ol R$ are Cohen-Macaulay,
it follows that $\ol R$ is Cohen-Macaulay.
The fact that $R$ contains a field implies that $\ol R$ also
contains a field. Hence, Case 2 shows that $\s_0(\ol R)$ is finite.
Since $R$ is local, Fact~\ref{fact120522a} provides an injective function
$\s_0(R)\into\s_0(R')$, so $\s_0(R)$ is finite as well.
\end{proof}

\begin{disc}\label{disc120522b}
Happel's result uses some  deep ideas from algebraic geometry
and representation theory. 
The essential point comes from a theorem of Voigt~\cite{voigt:idteag}
(see also Gabriel~\cite[1.2 Corollary]{gabriel:frto}). We'll need a 
souped-up version of this result for the full proof of Theorem~\ref{intthm120521a}.
This is the point of Section~\ref{sec120522f}.
\end{disc}

\begin{disc}\label{disc120522d}
The proof of Theorem~\ref{intthm120522b} uses the extra assumptions (extra compared to Theorem~\ref{intthm120521a})
in  crucial places. The Cohen-Macaulay assumption is used in the
reduction to the artinian case.
And the fact that $R$ contains a field is used in order to invoke Happel's result.
In order to remove these assumptions for the proof of 
Theorem~\ref{intthm120521a}, we find an 
algebra $U$ that is finite dimensional over an algebraically closed  field
such that there is an injective function $\s_0(R)\into\s(U)$. The trick is that $U$ is a DG algebra,
and $\s(U)$ is a set of equivalence classes of semidualizing
DG $U$-modules.
So, we need to understand the following:
\begin{enumerate}[(a)]
\item \label{disc120522d1}
What are DG algebras, and how is $U$ constructed?
\item \label{disc120522d2}
What are semidualizing DG modules, and how is the map $\s_0(R)\into\s(U)$ constructed?
\item \label{disc120522d4}
Why is $\s(U)$ finite?
\end{enumerate}
This is the point of the rest of the notes. See Sections~\ref{sec120522d},
\ref{sec120522d'}, and~\ref{sec120522f}.
\end{disc}

\section{Hom Complexes}
\label{sec120522h}

This section has several purposes. 
First, we set some notation and terminology.
Second, we make sure that the reader is familiar with some notions that we 
need later in the notes. 
One of the main points of this section is Fact~\ref{exer120522e}.

\subsection*{Complexes}
The following gadgets form the foundation for homological algebra, and we shall use them extensively.

\begin{defn}\label{defn120522b}
An \emph{$R$-complex}\footnote{Readers more comfortable with notations like
$X_\bullet$ or $X_*$ for complexes should feel free to decorate their complexes
as they see fit.} is a sequence of $R$-module homomorphisms
$$X=\cdots\xra{\partial^X_{i+1}}X_i\xra{\partial^X_{i}}X_{i-1}\xra{\partial^X_{i-1}}
\cdots$$
such that $\partial^X_{i}\partial^X_{i+1}=0$ for all $i$.
For each $x\in X_i$, the \emph{degree} of $x$ is $|x|:=i$.
The \emph{$i$th homology module} of $X$ is 
$\HH_i(X):=\ker(\partial^X_{i})/\im(\partial^X_{i+1})$.
A \emph{cycle} in $X_i$ is an element of $\ker(\partial^X_{i})$.
\end{defn}

We use the following notation for augmented resolutions in several places below.

\begin{ex}
\label{ex120522a}
Let $M$ be an $R$-module.
We consider $M$ as an $R$-complex ``concentrated in degree 0'':
$$M=\quad 0\to M\to 0.$$
Given an augmented projective resolution 
$$P^+=\cdots\xra{\partial^P_2}P_1\xra{\partial^P_1}P_0\xra\tau M\to 0$$
the truncated resolution 
$$P=\cdots\xra{\partial^P_2}P_1\xra{\partial^P_1}P_0\to 0$$
is an $R$-complex such that
$\HH_0(P)\cong M$ and $\HH_i(P)=0$ for all $i\neq 0$.
Similarly, given an augmented injective resolution 
$$^+I=\quad 0\to M\xra{\epsilon}I_0\xra{\partial^I_0}I_{-1}\xra{\partial^I_{-1}}
\cdots $$
the truncated resolution 
$$I=\quad 0\to I_0\xra{\partial^I_0}I_{-1}\xra{\partial^I_{-1}}
\cdots $$
is an $R$-complex such that
$\HH_0(I)\cong M$ and $\HH_i(I)=0$ for all $i\neq 0$.
\end{ex}

\subsection*{The Hom Complex}
The next constructions are used extensively in these notes.
For instance, the chain maps are the morphisms in the category of $R$-complexes.

\begin{defn}\label{defn120522c}
Let $X$ and $Y$ be $R$-complexes.
The \emph{Hom complex} $\Hom XY$ is defined as follows.
For each integer $n$, set
$\Hom XY_n:=\prod_{p\in\bbz}\Hom{X_p}{Y_{p+n}}$
and
$\partial^{\Hom XY}_n(\{f_p\}):=\{\partial^Y_{p+n}f_p-(-1)^nf_{p-1}\partial^X_p\}$.
A \emph{chain map} $X\to Y$ is a cycle in $\Hom XY_0$, i.e., an element of $\Ker(\partial^{\Hom XY}_0)$.
An element in $\Hom XY_0$ is
\emph{null-homotopic} if it is in $\im(\partial^{\Hom XY}_1)$.
An \emph{isomorphism} $X\xra\cong Y$ is a chain map $X\to Y$
with a two-sided inverse.
We sometimes write $f$ in place of $\{f_p\}$.
\end{defn}

\begin{exer}
\label{exer120522b}
Let $X$ and $Y$ be $R$-complexes.
\begin{enumerate}[(a)]
\item \label{exer120522b1}
Prove that $\Hom XY$ is an $R$-complex. 
\item \label{exer120522b2}
Prove that a chain map 
$X\to Y$ is a sequence of $R$-module homomorphisms
$\{f_p\colon X_p\to Y_p\}$ making the following diagram commute:
$$\xymatrix{
\cdots\ar[r]^-{\partial^X_{i+1}}
&X_i\ar[r]^-{\partial^X_{i}}\ar[d]_-{f_i}
&X_{i-1}\ar[r]^-{\partial^X_{i-1}}\ar[d]_-{f_{i-1}}
&\cdots
\\
\cdots\ar[r]^-{\partial^Y_{i+1}}
&Y_i\ar[r]^-{\partial^Y_{i}}
&Y_{i-1}\ar[r]^-{\partial^Y_{i-1}}
&\cdots.
}$$
\item \label{exer120522b3}
Prove that if $\{f_p\}\in\Hom XY_0$ is null-homotopic, then it is a chain map.
\item \label{exer120522b4}
Prove that a sequence $\{f_p\}\in\Hom XY_0$ is null-homotopic
if and only if there is a sequence 
$\{s_p\colon X_p\to Y_{p+1}\}$ of $R$-module homomorphisms
such that
$f_p=\partial^Y_{p+1}s_p+s_{p-1}\partial^X_p$ for all $p\in\bbz$.
\end{enumerate}
\end{exer}

The following exercises contain  useful properties of these constructions.

\begin{exer}[``Hom cancellation'']
\label{exer120522i}
Let $X$ be an $R$-complex.
Prove that the map $\tau\colon\Hom RX\to X$ given by
$\tau_n(\{f_p\})=f_n(1)$ is an isomorphism of $R$-complexes.
\end{exer}

\begin{exer}
\label{exer130314b}
Let $X$ be an $R$-complex, and let $M$ be an $R$-module.
\begin{enumerate}[(a)]
\item
\label{exer130314b1}
Prove that $\Hom MX$ is isomorphic to the following complex:
$$\cdots\xra{(\partial^X_{n+1})_*}(X_{n})_*\xra{(\partial^X_{n})_*}(X_{n-1})_*\xra{(\partial^X_{n-1})_*}\cdots$$
where $(-)_*=\Hom M-$ and $(\partial^X_{n})_*(f)=\partial^X_{n}f$.
\item
\label{exer130314b2}
Prove that $\Hom XM$ is isomorphic to  the following complex:
\begin{gather*}
\cdots\xra{(\partial^X_{n})^*}X_{n}^*\xra{(\partial^X_{n+1})^*}X_{n+1}^*\xra{(\partial^X_{n+2})^*}\cdots.
\end{gather*}
where $(-)^*=\Hom -M$ and $(\partial^X_{n})^*(f)=f\partial^X_{n}$.
[Hint: Mind the signs.]
\end{enumerate}
\end{exer}

\begin{exer}
\label{exer120522f}
Let $f\colon X\to Y$ be a chain map.
\begin{enumerate}[(a)]
\item \label{exer120522f1}
Prove that for each $i\in\bbz$, the chain map $f$ induces
a well-defined
$R$-module homomorphism $\HH_i(f)\colon\HH_i(X)\to\HH_i(Y)$
given by $\HH_i(f)(\ol x):=\ol{f_i(x)}$.
\item \label{exer120522f2}
Prove that if $f$ is null-homotopic, then $\HH_i(f)=0$ for all $i\in\bbz$.
\end{enumerate}
\end{exer}

The following concept is central for homological algebra;
see, e.g., Remark~\ref{disc120522e}.

\begin{defn}\label{defn120522d}
A chain map $f\colon X\to Y$ is a \emph{quasiisomorphism}
if for all $i\in\bbz$ the induced map $\HH_i(f)\colon\HH_i(X)\to\HH_i(Y)$
is an isomorphism. We use the symbol $\simeq$  to identify
quasiisomorphisms.
\end{defn}

\begin{exer}
\label{exer120522h}
Prove that each isomorphism  of $R$-complexes is a quasiisomorphism.
\end{exer}

\begin{exer}
\label{exer120522g}
Let $M$ be an $R$-module with augmented projective resolution $P^+$
and augmented injective resolution $^+I$; see the notation from
Example~\ref{ex120522a}. Prove that $\tau$ and $\epsilon$
induce quasiisomorphisms $P\xra\simeq M\xra\simeq I$.
\end{exer}

\begin{disc}\label{disc120522e}
Let $M$ and $N$ be $R$-modules.
The fact that $\Ext iMN$ can be computed using a projective resolution $P$
of $M$ or an injective resolution $I$ of $N$ is called the ``balance'' property
for Ext. It can be proved by showing that there are quasiisomorphisms
$\Hom PN\xra\simeq\Hom PI\xla\simeq\Hom MI$. See Fact~\ref{fact130312a}.
\end{disc}

\subsection*{Hom and Chain Maps (Functoriality)}
Given that the chain maps are the morphisms in the category of $R$-complexes, 
the next construction and the subsequent exercise  indicate that $\Hom Z-$ and $\Hom -Z$ are functors.

\begin{defn}\label{defn130312a}
Given a chain map $f\colon  X\to Y$ and an $R$-complex $Z$,
we define $\Hom Zf\colon\Hom ZX\to\Hom ZY$
as follows: each sequence $\{g_p\}\in\Hom ZX_n$ is mapped to $\{f_{p+n}g_p\}\in\Hom ZY_n$.
Similarly, define the map 
$\Hom fZ\colon\Hom YZ\to\Hom XZ$
by the formula $\{g_p\}\mapsto \{g_pf_p\}$.
\end{defn}

\begin{disc}
We do not use a sign-change in this definition
because $|f|=0$.
\end{disc}

\begin{exer}\label{exer130312a}
Given a chain map $f\colon  X\to Y$ and an $R$-complex $Z$,
Prove that  $\Hom Zf$ 
and  $\Hom fZ$ are  chain maps.
\end{exer}

\begin{fact}\label{fact130312a}
Let $f\colon  X\xra\simeq Y$ be a quasiisomorphism, and let $Z$ be an $R$-complex. 
In general, the chain map $\Hom Zf\colon\Hom ZX\to\Hom ZY$ is not a 
quasiisomorphism. However, if $Z$ is a complex of projective
$R$-modules such that $Z_i=0$ for $i\ll 0$, then $\Otimes Zf$ is  a 
quasiisomorphism.
Similarly, the chain map $\Hom fZ\colon\Hom YZ\to\Hom XZ$ is not a 
quasiisomorphism. However, if $Z$ is a complex of injective
$R$-modules such that $Z_i=0$ for $i\gg 0$, then $\Hom fZ$ is  a 
quasiisomorphism.
\end{fact}

\subsection*{Homotheties and Semidualizing Modules}
We next explain how the Hom complex relates to the semidualizing modules from Section~\ref{sec120521b}.

\begin{exer}
\label{exer120522c}
Let $X$ be an $R$-complex, and let $r\in R$.
For each $p\in\bbz$, let $\mu^{X,r}_p\colon X_p\to X_p$
be given by $x\mapsto rx$.
(Such a map is a ``homothety''.
When it is convenient, we denote this map as $X\xra rX$.)

Prove that $\mu^{X,r}:=\{\mu^{X,r}_p\}\in\Hom XX_0$ is a chain map.
Prove that for all $i\in\bbz$ the induced map 
$\HH_i(\mu^{X,r})\colon\HH_i(X)\to\HH_i(Y)$
is multiplication by $r$.
\end{exer}

\begin{exer}
\label{exer120522d}
Let $X$ be an $R$-complex. 
We use the notation from Exercise~\ref{exer120522c}.
Define $\chi^X_0\colon R\to\Hom XX$ by the formula
$\chi^X_0(r):=\mu^{X,r}\in\Hom XX_0$.
Prove that this determines a chain map $\chi^X\colon R\to\Hom XX$.
The chain map $\chi^X$ is the ``homothety morphism'' for $X$.
\end{exer}

\begin{fact}
\label{exer120522e}
Let $M$ be a finitely generated $R$-module.
We use the notation from Exercise~\ref{exer120522d}.
The following conditions are equivalent:
\begin{enumerate}[(i)]
\item \label{exer120522e1}
$M$ is a semidualizing $R$-module.
\item \label{exer120522e2}
For each projective resolution $P$ of $M$, the chain map
$\chi^P\colon R\to\Hom PP$ is a quasiisomorphism.
\item \label{exer120522e3}
For some projective resolution $P$ of $M$, the  chain map
$\chi^P\colon R\to\Hom PP$ is a quasiisomorphism.
\item \label{exer120522e4}
For each injective resolution $I$ of $M$, the  chain map
$\chi^I\colon R\to\Hom II$ is a quasiisomorphism.
\item \label{exer120522e5}
For some injective resolution $I$ of $M$, the  chain map
$\chi^I\colon R\to\Hom II$ is a quasiisomorphism.
\end{enumerate}
In some sense, the point is that the homologies of the complexes
$\Hom PP$ and $\Hom II$ are exactly the modules $\Ext iMM$ by Fact~\ref{fact130312a}.
\end{fact}

\section{Tensor Products and the Koszul Complex}
\label{sec120522g}

Tensor products for complexes are as fundamental for complexes as they are for modules.
In this section, we use them to construct the Koszul complex; see Definition~\ref{defn120522h}.
In Section~\ref{sec120522d'}, we use them for base change; see, e.g., Exercise~\ref{disc111223a}.

\subsection*{Tensor Product of Complexes}
As with the Hom complex, the sign convention in the next construction guarantees that
it is an $R$-complex; see Exercise~\ref{exer120522k} and Remark~\ref{disc130311}.
Note that Remark~\ref{disc130331b} describes a notational simplification.

\begin{defn}\label{defn120522g}
Fix $R$-complexes $X$ and $Y$.
The \emph{tensor product complex} $\Otimes XY$ is defined as follows.
For each index $n$, set
$(\Otimes XY)_n:=\bigoplus_{p\in\bbz}\Otimes{X_p}{Y_{n-p}}$
and let $\partial^{\Otimes XY}_n$ be given on generators by the formula
$\partial^{\Otimes XY}_n(\ldots,0,x_p\otimes y_{n-p},0,\ldots):=
(\ldots,0,\partial^X_p(x_p)\otimes y_{n-p},(-1)^px_p\otimes \partial^Y_{n-p}(y_{n-p}),0,\ldots)$.
\end{defn}

\begin{exer}
\label{exer120522k}
Let $X$, $Y$, and $Z$ be $R$-complexes.
\begin{enumerate}[(a)]
\item
\label{exer120522k1}
Prove that $\Otimes XY$ is an $R$-complex. 
\item
\label{exer120522k2}
Prove that there is a ``tensor cancellation'' isomorphism
$\Otimes RX\cong X$.
\item
\label{exer120522k3}
Prove that there is a ``commutativity'' isomorphism
$\Otimes XY\cong\Otimes YX$.
(Hint: Mind the signs. This isomorphism
is given by $x\otimes y\mapsto(-1)^{|x||y|}y\otimes x$.)
\item
\label{exer120522k4}
Verify the ``associativity'' isomorphism
$\Otimes{X}{(\Otimes YZ)}\cong\Otimes{(\Otimes XY)}{Z}$.
\end{enumerate}
\end{exer}

\begin{disc}
\label{disc130311}
There is a rule of thumb for sign conventions like the one in the hint for Exercise~\ref{exer120522k}:
whenever two factors $u$ and $v$ are commuted in an expression, you multiply by $(-1)^{|u||v|}$.
This can already be seen in $\partial^{\Hom XY}$ and $\partial^{\Otimes XY}$. This graded commutativity is one of the 
keys to DG algebra. See Section~\ref{sec120522d}.
\end{disc}

\begin{disc}
\label{disc130331b}
After working with the tensor product of complexes for a few moments, one realizes that the sequence notation 
$(\ldots,0,x_p\otimes y_{n-p},0,\ldots)$ is unnecessarily cumbersome. 
We use the sequence notation in a few of the solutions in Appendix~\ref{sec130212a}, but not for many of them.
Similarly, from now on, instead of writing $(\ldots,0,x_p\otimes y_{n-p},0,\ldots)$, we write the simpler
$x_p\otimes y_{n-p}$. As we note in~\ref{para130311e}, one needs to be somewhat careful with this notation,
as elements $u\otimes v$ and $x\otimes y$ only live in the same summand when $|u|=|x|$ and $|v|=|y|$.
\end{disc}

\begin{fact}\label{disc120522f}
Given $R$-complexes $X^1,\ldots, X^n$ an induction argument
using the associativity isomorphism from Exercise~\ref{exer120522k}
shows that the $n$-fold tensor product
$X^1\otimes_R\cdots\otimes_R X^n$ is well-defined (up to isomorphism).
\end{fact}

\subsection*{Tensor Products and Chain Maps (Functoriality)}
As for Hom, 
the next construction and  exercise  indicate that $\Otimes Z-$ and $\Otimes -Z$ are functors.

\begin{defn}\label{defn120523a}
Given a chain map $f\colon  X\to Y$ and an $R$-complex $Z$,
we define $\Otimes Zf\colon\Otimes ZX\to\Otimes ZY$
by the formula $z\otimes y\mapsto z\otimes f(y)$.
Define the map 
$\Otimes fZ\colon\Otimes XZ\to\Otimes YZ$
by the formula $x\otimes z\mapsto f(x)\otimes z$.
\end{defn}

\begin{disc}
We do not use a sign-change in this definition
because $|f|=0$.
\end{disc}

\begin{exer}\label{exer120523a}
Given a chain map $f\colon  X\to Y$ and an $R$-complex $Z$,
the maps $\Otimes Zf\colon\Otimes ZX\to\Otimes ZY$
and $\Otimes fZ\colon\Otimes XZ\to\Otimes YZ$
are chain maps.
\end{exer}

\begin{fact}\label{fact120523a}
Let $f\colon  X\xra\simeq Y$ be a quasiisomorphism, and let $Z$ be an $R$-complex. 
In general, the chain map $\Otimes Zf\colon\Otimes ZX\to\Otimes ZY$ is not a 
quasiisomorphism. However, if $Z$ is a complex of projective
$R$-modules such that $Z_i=0$ for $i\ll 0$, then $\Otimes Zf$ is  a 
quasiisomorphism.
\end{fact}

\subsection*{The Koszul Complex}
Here begins our discussion of the prototypical DG algebra.

\begin{defn}\label{defn120522h}
Let $\x=x_1,\ldots,x_n\in R$.
For $i=1,\ldots,n$ set
$$K^R(x_i)=\quad 0\to R\xra{x_i} R\to 0.$$
Using Remark~\ref{disc120522f}, we set
$$K^R(\x)=K^R(x_1,\ldots,x_n)=K^R(x_1)\otimes_R\cdots\otimes_R K^R(x_n).$$
\end{defn}

\begin{exer}
\label{exer120522j}
Let $x,y,z\in R$.
Write out explicit formulas, using matrices for the differentials,
for $K^R(x,y)$ and $K^R(x,y,z)$.
\end{exer}

\begin{exer}
\label{exer120522l}
Let $\x=x_1,\ldots,x_n\in R$.
Prove that $K^R(\x)_i\cong R^{\binom{n}{i}}$ for all $i\in\bbz$.
(Here we use the convention $\binom{n}{i}=0$ for all $i<0$ and $i>n$.)
\end{exer}

\begin{exer}
\label{exer120522u}
Let $\x=x_1,\ldots,x_n\in R$. Let $\sigma\in S_n$,
and set $\x'=x_{\sigma(1)},\ldots,x_{\sigma(n)}$.
Prove that $K^R(\x)\cong K^R(\x')$.
\end{exer}

Given a generating sequence $\x$ for the maximal ideal of a local ring $R$, one concludes from the next lemma that each homology module $\HH_i(K^R(\x))$
has finite length.
This is crucial for the proof of Theorem~\ref{intthm120521a}.

\begin{lem}
\label{exer120522t}
Let $\x=x_1,\ldots,x_n\in R$, and consider the ideal $\fa=(\x)R$.
Then $\fa\HH_i(K^R(\x))=0$ for all $i\in\bbz$.
\end{lem}

\begin{proof}[Sketch of proof] 
It suffices to show that for $j=1,\ldots, n$ we have 
$x_j\HH_i(K^R(\x))=0$ for all $i\in\bbz$.
By symmetry (Exercise~\ref{exer120522u})
it suffices to show that $x_1\HH_i(K^R(\x))=0$ for all $i\in\bbz$ for $j=1,\ldots, n$.
The following diagram shows that the chain map $K^R(x_1)\xra{x_1}K^R(x_1)$ is null-homotopic.
$$\xymatrix{
0\ar[r]
&R\ar[r]^{x_1}\ar[d]^<<<{x_1}\ar[ld]
&R\ar[r]\ar[d]^<<<{x_1}\ar[ld]^1
&0\ar[ld]
\\
0\ar[r]
&R\ar[r]_{x_1}
&R\ar[r]
&0
}$$
It is routine to show that this implies that the induced map
$K^R(\x)\xra{x_1}K^R(\x)$ is null-homotopic.
The desired conclusion now follows from Exercise~\ref{exer120522f}.
\end{proof}

The next construction allows us to push our complexes around.

\begin{defn}\label{defn120522k}
Let $X$ be an $R$-complex, and let $n\in\bbz$.
The \emph{$n$th suspension} (or \emph{shift}) of $X$ 
is the complex $\shift^nX$ such that
$(\shift^nX)_i:=X_{i-n}$ and
$\partial^{\shift^nX}_i=(-1)^n\partial^X_{i-n}$.
We set $\shift X:=\shift^1X$.
\end{defn}

The next fact is in general quite useful, though we do not exploit it here.

\begin{fact}
\label{exer120522n}
Let $\x=x_1,\ldots,x_n\in R$.
The Koszul complex $K^R(\x)_i$ is ``self-dual'', that is,
that there is an isomorphism of $R$-complexes
$\Hom{K^R(\x)}{R}\cong\shift^nK^R(\x)$.
(This fact is not trivial.)
\end{fact}

\begin{exer}
\label{exer130314a}
Verify the isomorphism from Fact~\ref{exer120522n} for $n=1,2,3$.
\end{exer}

The following result gives the first indication of the utility of the Koszul complex.
We use it explicitly in the proof of Theorem~\ref{intthm120521a}.

\begin{lem}\label{lem120529a}
Let $\x=x_1,\ldots,x_n\in R$.
If $\x$ is $R$-regular, then $K^R(\x)$ is a free resolution of $R/(\x)$ over $R$.
\end{lem}

\begin{proof}
Argue by induction on $n$. 

Base case: $n=1$. Assume that $x_1$ is $R$-regular. Since $K^R(x_1)$
has the form $0\to R\xra{x_1}R\to 0$, the fact that $x_1$ is $R$-regular implies that
$\HH_i(K^R(x_1))=0$ for all $i\neq 0$. Since each module in $K^R(x_1)$ is free,
it follows that $K^R(x_1)$ is a free resolution of $R/(x_1)$.

Inductive step: Assume that $n\geq 2$ and that the result holds for regular sequences of
length $n-1$. Assume that $\x$ is $R$-regular (of length $n$). 
Thus, the sequence $\x'=x_1,\ldots,x_{n-1}$ is $R$-regular, and $x_n$ is $R/(\x')$-regular.
The first condition implies that $K':=K^R(\x')$ is a free resolution of $R/(\x')$ over $R$.
By definition, we have $K:=K^R(\x)=K'\otimes_RK^R(x_n)$. 
Further, by definition of $K'\otimes_RK^R(x_n)$, we have
$$K\cong\cdots 
\xra{\left(\begin{smallmatrix}\partial^{K'}_{i+1}&(-1)^{i} x_n\\ 0&\partial^{K'}_{i}\end{smallmatrix}\right)}
\begin{matrix}K'_i\\ \bigoplus \\ K'_{i-1}\end{matrix}
\xra{\left(\begin{smallmatrix}\partial^{K'}_{i}&(-1)^{i-1} x_n\\ 0&\partial^{K'}_{i-1}\end{smallmatrix}\right)}
\begin{matrix}K'_{i-1}\\ \bigoplus \\ K'_{i-2}\end{matrix}
\xra{\left(\begin{smallmatrix}\partial^{K'}_{i-1}&(-1)^{i-2} x_n\\ 0&\partial^{K'}_{i-2}\end{smallmatrix}\right)}
\cdots.$$
Using this, there is a short exact sequence of $R$-complexes and chain maps\footnote{Readers
who are familiar with the mapping cone description of $K$ should not be surprised by this argument.}
$$\xymatrix@R=3mm{
0\ar[r]
&K'\ar[r]
&K\ar[r]
&K''\ar[r]
&0
\\
&\vdots\ar[dd]_{\partial^{K'}_{i+1}}
&\vdots\ar[dd]_{\partial^{K}_{i+1}}
&\vdots\ar[dd]_{\partial^{K'}_{i}}
&
\\ \\
0\ar[r]
&K'_i\ar[r]\ar[dd]_{\partial^{K'}_{i}}
&K_i\ar[r]\ar[dd]_{\partial^{K}_{i}}
&K'_{i-1}\ar[r]\ar[dd]_{\partial^{K'}_{i-1}}
&0
\\ \\
0\ar[r]
&K'_{i-1}\ar[r]\ar[dd]_{\partial^{K'}_{i-1}}
&K_{i-1}\ar[r]\ar[dd]_{\partial^{K}_{i-1}}
&K'_{i-2}\ar[r]\ar[dd]_{\partial^{K'}_{i-2}}
&0
\\
\\
&\vdots
&\vdots
&\vdots
&
}$$
where
$K''$ is obtained by shifting $K'$.\footnote{Note that $K''$ is technically not equal to the complex $\shift K'$ from
Definition~\ref{defn120522k}, since there is no sign on the differential. On the other hand the complexes
$K''$ and $\shift K'$ are isomorphic.}
Furthermore, it can be shown that the long exact sequence in homology
has the form
$$\cdots\HH_i(K')\xra{(-1)^{i} x_n}\HH_i(K')\to\HH_i(K)\to\HH_i(K')\xra{(-1)^{i-1} x_n}\HH_i(K')\to\cdots.$$
Since $K'$ is a free resolution of $R/(\x')$ over $R$, we have $\HH_i(K')=0$ for all $i\neq 0$,
and $\HH_0(K')\cong R/(\x')$.
As $x_n$ is $R/(\x')$-regular, an analysis of the long exact sequence shows that
$\HH_i(K)=0$ for all $i\neq 0$,
and $\HH_0(K)\cong R/(\x)$.
It follows that $K$ is a free resolution of $R/(\x)$, as desired.
\end{proof}

\subsection*{Alternate Description of the Koszul Complex}

The following  description of $K^R(\x)$  says that $K^R(\x)$ is given by the ``exterior algebra'' on $R^n$;
see Fact~\ref{exer120522m}.

\begin{defn}\label{defn120522j}
Let $\x=x_1,\ldots,x_n\in R$.
Fix a basis $e_1,\ldots,e_n\in R^n$.
For $i>1$, set $\bigwedge^iR^n:=R^{\binom{n}{i}}$
with basis given by the set of formal symbols
$e_{j_1}\wedge\cdots\wedge e_{j_i}$ such that $1\leq j_1<\cdots<j_i\leq n$.
This extends to all $i\in\bbz$ as follows:
$\bigwedge^1R^n=R^{n}$ with basis $e_1,\ldots,e_n$
and $\bigwedge^0R^n=R^{1}$ with basis $1$;
for $i<0$, set $\bigwedge^iR^n=R^{\binom{n}{i}}=0$.

Define $\wti K^R(\x)$ as follows.
For all $i\in\bbz$ set
$\wti K^R(\x)_i=\bigwedge^iR^n$,
and let $\partial^{\wti K^R(\x)}_i$ be given on basis vectors by the formula
$$\partial^{\wti K^R(\x)}_i(e_{j_1}\wedge\cdots\wedge e_{j_i})=
\sum_{p=1}^i(-1)^{p+1}x_{j_p}e_{j_1}\wedge\cdots\wedge \widehat{e_{j_p}}\wedge\cdots\wedge e_{j_i}$$ 
where the notation $\widehat{e_{j_p}}$ indicates that
$e_{j_p}$ has been removed from the list.
In the case $i=1$, the  formula reads as
$\partial^{\wti K^R(\x)}_1(e_{j})=x_j$.
\end{defn}

\begin{disc}\label{disc120522k}
Our definition of $\bigwedge^iR^n$ is \emph{ad hoc}.
A better way to think about it (in some respects) is in terms of a universal mapping property
for alternating multilinear maps.
A basis-free construction can be given in terms of a certain quotient of
the $i$-fold tensor product $R^n\otimes_R\cdots\otimes_RR^n$.
\end{disc}

\begin{exer}
\label{exer120522o}
Let $\x=x_1,\ldots,x_n\in R$.
Write out explicit formulas, using matrices for the differentials,
for $\wti K^R(\x)$ in the cases $n=1,2,3$.
\end{exer}

\begin{fact}
\label{exer120522m}
Let $\x=x_1,\ldots,x_n\in R$.
There is an isomorphism of $R$-complexes
$K^R(\x)\cong \wti K^R(\x)$.
(This fact is not trivial. For perspective on this, compare the solutions to Exercises~\ref{exer120522j}
and~\ref{exer120522o} in~\ref{para130312b} and~\ref{para130312e}.)
In light of this fact,  we do not distinguish between $K^R(\x)$ and $\wti K^R(\x)$ for the remainder of these notes.
\end{fact}

\begin{disc}\label{disc120522z}
A third description of $K^R(\x)$ involves the mapping cone.
Even though it is extremely useful,
we do not discuss it  in detail here.
\end{disc}

\subsection*{Algebra Structure on the Koszul Complex}
In our estimation, the Koszul complex is one of the most important
constructions in commutative algebra. 
When the sequence $\x$ is $R$-regular, it is an $R$-free resolution
of $R/(\x)$, by Lemma~\ref{lem120529a}. In general, it detects depth and has all scads of other magical properties.
For us, one its most important features is its
algebra structure, which we describe next.

\begin{defn}\label{defn120522l}
Let $n\in\bbn$ and let $e_1,\ldots,e_n\in R^n$ be a basis.
In $\bigwedge^2R^n$, define 
$$
e_{j_2}\wedge e_{j_1}
:=
\begin{cases}
-e_{j_1}\wedge e_{j_2}&
\text{whenever $1\leq j_1<j_2\leq n$} \\
0 &\text{whenever $1\leq j_1=j_2\leq n$.} 
\end{cases}
$$
Extending this bilinearly, we define $\alpha\wedge\beta$ for all
$\alpha,\beta\in\bigwedge^1R^n=R^n$:
write $\alpha=\sum_p\alpha_pe_p$
and $\beta=\sum_q\beta_qe_q$, and define
$$\alpha\wedge\beta
=\left(\sum_p\alpha_pe_p\right)\wedge\left(\sum_q\beta_qe_q\right)
=\sum_{p,q}\alpha_p\beta_qe_p\wedge e_q
=\sum_{p<q}(\alpha_p\beta_q-\alpha_q\beta_p)e_p\wedge e_q.
$$
This extends  to a multiplication
$\bigwedge^1R^n\times\bigwedge^tR^n\to\bigwedge^{1+t}R^n$
using the following formula, assuming that $1\leq i\leq n$ and $1\leq j_1<\cdots<j_t\leq n$:
$$
e_{i}\wedge (e_{j_1}\wedge \cdots\wedge e_{j_t})
:=
\begin{cases}
0 &\text{if $i= j_p$ for some $p$} \\
e_{i}\wedge e_{j_1}\wedge \cdots\wedge e_{j_t}&
\text{if $i<j_1$} \\
(-1)^pe_{j_1}\wedge \cdots \wedge e_{j_p}\wedge e_{i}\wedge\cdots\wedge e_{j_t}&
\text{if $j_p<i<j_{p+1}$} \\
(-1)^te_{j_1}\wedge \cdots\wedge e_{j_t}\wedge e_{i}&
\text{if $j_t<i$.} 
\end{cases}
$$
This extends (by induction on $s$) to a multiplication
$\bigwedge^sR^n\times\bigwedge^tR^n\to\bigwedge^{s+t}R^n$
using the following formula when $i_1<\ldots<i_s$ and $1\leq j_1<\cdots<j_t\leq n$:
$$(e_{i_1}\wedge \cdots\wedge e_{i_s})\wedge (e_{j_1}\wedge \cdots\wedge e_{j_t}):=
e_{i_1}\wedge [(e_{i_2}\wedge \cdots\wedge e_{i_s})\wedge (e_{j_1}\wedge \cdots\wedge e_{j_t})].
$$
This multiplication is denoted as $(\alpha,\beta)\mapsto\alpha\wedge\beta$.
When $s=0$, since $\bigwedge^0R^n=R$,  the usual scalar multiplication 
$R\times\bigwedge^tR^n\to\bigwedge^{t}R^n$ describes multiplication
$\bigwedge^0R^n\times\bigwedge^tR^n\to\bigwedge^{t}R^n$,
and similarly when $t=0$.
This further extends to a well-defined multiplication on 
$\bigwedge R^n:=\bigoplus_i\bigwedge^iR^n$.
\end{defn}

\begin{disc}
According to Definition~\ref{defn120522l}, for $0\in\bigwedge^sR^n$ and $\beta\in\bigwedge^tR^n$,
we have $0\wedge\beta=0=\beta\wedge 0$.
\end{disc}

\begin{ex}
\label{ex130316a}
We compute a few products in $\bigwedge R^4$:
\begin{align*}
(e_1\wedge e_2)\wedge(e_3\wedge e_4)
&=e_1\wedge e_2\wedge e_3\wedge e_4\\
(e_1\wedge e_2)\wedge(e_2\wedge e_3)
&=0 \\
(e_1\wedge e_3)\wedge(e_2\wedge e_4)
&=e_1\wedge [e_3\wedge (e_2\wedge e_4)] \\
&=e_1\wedge [-e_2\wedge e_3\wedge e_4] \\
&=-e_1\wedge e_2\wedge e_3\wedge e_4.
\end{align*}
\end{ex}

\begin{exer}
\label{exer120522p}
Write out multiplication tables (for basis vectors only) for
$\bigwedge R^n$ with $n=1,2,3$.
\end{exer}

Definition~\ref{defn120522l} suggests the next notation, which facilitates many computations.
The subsequent exercises   examplify this, culminating in the important Exercise~\ref{exer120522s}

\begin{defn}
\label{disc130316a}
Let $n\in\bbn$, let $e_1,\ldots,e_n\in R^n$ be a basis,
and let $j_1,\ldots,j_t\in\{1,\ldots,n\}$.
Since multiplication of basis vectors in $\bigwedge R^n$ is defined inductively,
the following element (also defined inductively)
$$
e_{j_1}\wedge\cdots\wedge e_{j_t}:=
\begin{cases}
0&\text{if $j_p=j_q$ for some $p\neq q$} \\
e_{j_1}\wedge(e_{j_2}\wedge\cdots\wedge e_{j_t})
&\text{if $j_p\neq j_q$ for all $p\neq q$.}
\end{cases}
$$
is well-defined.
\end{defn}

\begin{exer}
\label{exer120522q}
Let $n\in\bbn$, and let $e_1,\ldots,e_n\in R^n$ be a basis.
Let $j_1,\ldots,j_t\in\{1,\ldots,n\}$ such that $j_p\neq j_q$ for all $p\neq q$, and let $\iota\in S_n$  such that $\iota$ fixes all elements of $\{1,\ldots,n\}\smallsetminus\{j_1,\ldots,j_t\}$.
Prove that
$$
e_{\iota(j_1)}\wedge\cdots\wedge e_{\iota(j_t)}
=\operatorname{sgn}(\iota)e_{j_1}\wedge\cdots\wedge e_{j_t}
$$
\end{exer}

\begin{exer}
\label{exer120522q'}
Let $n\in\bbn$, and let $e_1,\ldots,e_n\in R^n$ be a basis.
Prove that for all $i_1,\ldots,i_{s},j_1,\ldots,j_t\in\{1,\ldots,n\}$ we have
$$(e_{i_1}\wedge \cdots\wedge e_{i_s})\wedge (e_{j_1}\wedge \cdots\wedge e_{j_t})=
e_{i_1}\wedge \cdots\wedge e_{i_s}\wedge e_{j_1}\wedge \cdots\wedge e_{j_t}$$
and that this is $0$ if there is a repetition in the list $i_1,\ldots,i_{s},j_1,\ldots,j_t$.
(The points here are the order on the subscripts matter in Definition~\ref{defn120522l} and do not matter in Definition~\ref{disc130316a},
so one needs to make sure that the signs that occur from Definition~\ref{defn120522l} agree with those from Definition~\ref{disc130316a}.)
\end{exer}

\begin{exer}
\label{exer120522r}
Let $n\in\bbn$ and let $e_1,\ldots,e_n\in R^n$ be a basis.
Prove that the multiplication from Definition~\ref{defn120522l} makes 
$\bigwedge R^n$ into a graded commutative $R$-algebra.
That is:
\begin{enumerate}[(a)]
\item \label{exer120522r1}
multiplication in $\bigwedge R^n$
is associative, distributive, and unital; 
\item \label{exer120522r2}
for
elements $\alpha\in \bigwedge^s R^n$ and $\beta\in \bigwedge^t R^n$,
we have $\alpha\wedge\beta=(-1)^{st}\beta\wedge\alpha$;
\item \label{exer120522r3}
for $\alpha\in \bigwedge^s R^n$, if $s$ is odd, then $\alpha\wedge\alpha=0$;
and
\item \label{exer120522r4}
the composition $R\xra\cong\bigwedge^0 R^n\xra\subseteq\bigwedge R^n$
is a ring homomorphism, the image of which is contained in the center of $\bigwedge R^n$.
\end{enumerate}
Hint: 
The distributive law holds essentially by definition.
For the other properties in~\eqref{exer120522r1}
and~\eqref{exer120522r2}, prove the desired formula for basis vectors, then 
verify it for general elements using distributivity.
\end{exer}

\begin{exer}
\label{exer130318a}
Let $n\in\bbn$ and let $e_1,\ldots,e_n\in R^n$ be a basis.
Let $\x=x_1,\ldots,x_n\in R$,
and let $j_1,\ldots,j_t\in\{1,\ldots,n\}$.
Prove that the element
$e_{j_1}\wedge\cdots\wedge e_{j_t}$
from Definition~\ref{disc130316a}
satisfies the following formula:
$$\partial^{\wti K^R(\x)}_i(e_{j_1}\wedge\cdots\wedge e_{j_t})=
\sum_{s=1}^t(-1)^{s+1}x_{j_s}e_{j_1}\wedge\cdots\wedge \widehat{e_{j_s}}\wedge\cdots\wedge e_{j_t}.$$ 
(Note that, if $j_1<\cdots<j_t$, then this is the definition of $\partial^{\wti K^R(\x)}_i(e_{j_1}\wedge\cdots\wedge e_{j_t})$.
However, we are not assuming that $j_1<\cdots<j_t$.)
\end{exer}

\begin{exer}
\label{exer120522s}
Let $n\in\bbn$ and let $e_1,\ldots,e_n\in R^n$ be a basis.
Let $\x=x_1,\ldots,x_n\in R$.
Prove that the multiplication from Definition~\ref{defn120522l} satisfies the 
``Leibniz rule'':
for
elements $\alpha\in \bigwedge^s R^n$ and $\beta\in \bigwedge^t R^n$,
we have 
$$\partial^{\wti K^R(\x)}_{s+t}(\alpha\wedge\beta)=
\partial^{\wti K^R(\x)}_{s}(\alpha)\wedge\beta
+(-1)^s\alpha\wedge\partial^{\wti K^R(\x)}_{t}(\beta).$$
Hint: 
Prove the  formula for basis vectors and 
verify it for general elements using distributivity and linearity.
\end{exer}

\section{DG Algebras and  DG Modules I}
\label{sec120522d}

This section introduces the main tools for the proof of Theorem~\ref{intthm120521a}.
This proof is outlined in~\ref{proof120523b}.
From the point of view of this proof, the first important example of a DG algebra  is the Koszul complex; see Example~\ref{ex130331a}.
However, the proof showcases another important example, namely, the DG algebra resolution of Definition~\ref{defn130331}. 

\subsection*{DG Algebras}
The first change of perspective required for the proof of Theorem~\ref{intthm120521a}
is the change from rings to DG algebras.

\begin{defn}
\label{DGK}
A \emph{commutative differential graded algebra over $R$} (\emph{DG $R$-algebra} for short)
is an $R$-complex $A$ equipped with a binary
operation $A\times A\to A$, written as $(a,b)\mapsto ab$ and called
the \emph{product} on $A$, 
satisfying the following properties:\footnote{We
assume that readers at this level are familiar with associative laws and the
like. However, given that the DG universe is riddled with sign conventions, we
explicitly state these laws for the sake of clarity.}
\begin{description}
\item[associative] for all $a,b,c\in A$ we have $(ab)c=a(bc)$;
\item[distributive] for all $a,b,c\in A$ such that 
$|a|=|b|$ we have $(a+b)c=ac+bc$ and $c(a+b)=ca+cb$;
\item[unital] there is an element $1_A\in A_0$ such that for all $a\in A$ we have $1_Aa=a$;
\item[graded commutative] for all $a,b\in A$ we have 
$ba = (-1)^{|a||b|}ab\in A_{|a|+|b|}$, and $a^2=0$ when
$|a|$ is odd; 
\item[positively graded] $A_i=0$ for $i<0$; and
\item[Leibniz Rule] 
for all $a,b\in A$ we have 
$\partial^A_{|a|+|b|}(ab)=\partial^A_{|a|}(a)b+(-1)^{|a|}a\partial^A_{|b|}(b)$.
\end{description}
Given a DG $R$-algebra $A$, the \emph{underlying algebra} is the
graded commutative  $R$-algebra
$\und{A}=\bigoplus_{i=0}^\infty A_i$.
When $R$ is a field and $\rank_R(\bigoplus_{i\geq 0}A_i)<\infty$, we say that $A$ is \emph{finite-dimensional} over $R$.
\end{defn}

It should be helpful for the reader to keep the next two examples in mind for the remainder of these notes.

\begin{ex}\label{ex120523a}
The ring $R$, considered as a complex concentrated in degree 0, is a DG $R$-algebra
such that $\und R=R$.
\end{ex}

\begin{ex}\label{ex130331a}
Given a sequence $\x=x_1,\cdots,x_n\in R$,
the Koszul complex $K=K^R(\x)$ is a DG $R$-algebra
such that $\und{K}=\bigwedge R^n$;
see Exercises~\ref{exer120522r} and~\ref{exer120522s}.
In particular, if $n=1$, then $\und K\cong R[X]/(X^2)$.
\end{ex}

The following exercise is a routine interpretation of the Leibniz Rule.
It is also an important foreshadowing of the final part of the  proof of Theorem~\ref{intthm120521a}
which is given in Section~\ref{sec120522f}. See also Exercise~\ref{fact110216a'}.

\begin{exer}\label{fact110216a}
Let $A$ be a DG $R$-algebra.
Prove that there is a well-defined
chain map $\mu^A\colon \Otimes AA\to A$ given by $\mu^A(a\otimes b)=ab$,
and that $A_0$ is an $R$-algebra.
\end{exer}

The next notion will allow us to transfer information from one DG algebra to another
as in the arguments for $R\to\comp R$ and $R\to\overline R$ described in Section~\ref{sec0}.

\begin{defn}
\label{DGK'}
A \emph{morphism} of DG $R$-algebras is a chain map
$f\colon A\to B$ between DG $R$-algebras  respecting products and multiplicative identities:
$f(aa')=f(a)f(a')$ and $f(1_A)=1_B$.
A morphism of DG $R$-algebras that is also a quasiisomorphism is a
\emph{quasiisomorphism of DG $R$-algebras}.
\end{defn}

Part~\eqref{ex120523a'''1} of the next exercise contains the first morphism of DG algebras that we  use 
in the proof of Theorem~\ref{intthm120521a}.

\begin{exer}\label{ex120523a'''}
Let $A$ be a DG $R$-algebra.
\begin{enumerate}[(a)]
\item\label{ex120523a'''1}
Prove that the map $R\to A$ given by $r\mapsto r\cdot 1_A$ is a morphism of DG $R$-algebras.
As a special case, 
given a sequence $\x=x_1,\cdots,x_n\in R$,
the natural map $R\to K=K^R(\x)$ given by $r\mapsto r\cdot 1_K$ 
is a morphism of DG $R$-algebras.
\item\label{ex120523a'''2}
Prove that the natural inclusion map $A_0\to A$ is a morphism of DG $R$-algebras.
\item\label{ex120523a'''3}
Prove that the natural map $K\to R/(\x)$ is a morphism of DG $R$-algebras
that is a quasiisomorphism if $\x$ is $R$-regular; see Exercise~\ref{exer120522g} and Lemma~\ref{lem120529a}.
\end{enumerate}
\end{exer}

Part~\eqref{ex120523a''2} of the next exercise is needed for the subsequent definition.

\begin{exer}\label{ex120523a''}
Let $A$ be a DG $R$-algebra.
\begin{enumerate}[(a)]
\item
\label{ex120523a''1}
Prove that the condition $A_{-1}=0$ implies that $A_0$ surjects onto $\HH_0(A)$
and that $\HH_0(A)$ is an $A_0$-algebra.  
\item
\label{ex120523a''2}
Prove that 
the $R$-module $A_i$ is an $A_0$-module, and $\HH_i(A)$ is an $\HH_0(A)$-module
for each $i$.
\end{enumerate}
\end{exer}

\begin{defn}\label{defn110216a}
Let $A$ be a DG $R$-algebra. We say that $A$ is \emph{noetherian}
if $\HH_0(A)$ is noetherian and  $\HH_i(A)$ is  finitely generated over $\HH_0(A)$ for all $i\geq 0$.
\end{defn}

\begin{exer}
\label{exer130312b}
Given a sequence $\x=x_1,\cdots,x_n\in R$,
prove that the Koszul complex $K^R(\x)$ is a noetherian DG $R$-algebra.
Moreover, prove that any DG $R$-algebra $A$ such that each $A_i$ is finitely generated
over $R$ is noetherian.
\end{exer}

\subsection*{DG Modules}
In the passage from rings to DG algebras, modules and complexes change to DG modules,
which we describe next.

\begin{defn}
\label{defn130313a}
Let $A$ be a DG $R$-algebra, and let $i$ be an integer. A \emph{differential graded module over} $A$
(\emph{DG $A$-module} for short) is an $R$-complex $M$ equipped with a
binary operation $A\times M\to M$, written as $(a,m)\mapsto am$ and called
the \emph{scalar multiplication} of $A$ on $M$, 
satisfying the following properties:
\begin{description}
\item[associative] for all $a,b\in A$ and $m\in M$ we have $(ab)m=a(bm)$;
\item[distributive] for all $a,b\in A$ and $m,n\in M$ such that 
$|a|=|b|$ and $|m|=|n|$, we have $(a+b)m=am+bm$ and $a(m+n)=am+an$;
\item[unital] for all $m\in M$ we have $1_Am=m$;
\item[graded] for all $a\in A$ and $m\in M$ we have 
$am\in M_{|a|+|m|}$; 
\item[Leibniz Rule] 
for all $a\in A$ and $m\in M$ we have 
$\partial^A_{|a|+|m|}(am)=\partial^A_{|a|}(a)m+(-1)^{|a|}a\partial^M_{|m|}(m)$.
\end{description}
The \emph{underlying $\und{A}$-module} associated to $M$ is the
$\und{A}$-module
$\und{M}=\bigoplus_{j=-\infty}^\infty M_j$.

The $i$th \emph{suspension} of a DG $A$-module $M$
is the DG $A$-module $\shift^iM$ defined by $(\shift^iM)_n :=
M_{n-i}$ and $\partial^{\shift^iM}_n := (-1)^i\partial^M_{n-i}$. 
The scalar multiplication
on $\shift^iM$ is defined by the formula
$a\ast m:=(-1)^{i|a|}a m$.
The notation $\shift M$ is short for $\shift^1M$.
\end{defn}

The next  exercise contains examples that  should be helpful to keep in mind.

\begin{exer}\label{ex110218a'}
\begin{enumerate}[(a)]
\item
\label{ex110218a'1}
Prove that DG $R$-module is just an $R$-complex.
\item
\label{ex110218a'2}
Given a DG $R$-algebra $A$, prove that the complex $A$ is a DG $A$-module
where the scalar multiplication is just the internal multiplication on $A$.
\item
\label{ex110218a'3}
Given a morphism $A\to B$ of DG $R$-algebras, prove that every DG $B$-module 
is a DG $A$-module by restriction of scalars.
As a special case, 
given a sequence $\x=x_1,\cdots,x_n\in R$,
every $R/(\x)$-complex is a DG $K^R(\x)$-module; see Exercise~\ref{ex120523a'''}.
\end{enumerate}
\end{exer}

The operation $X\mapsto\Otimes AX$ described next is ``base change'', which is crucial for our
passage between DG algebras in the proof of Theorem~\ref{intthm120521a}.

\begin{exer}\label{ex120526b}
Let $\x=x_1,\cdots,x_n\in R$, and set $K=K^R(\x)$. Given an $R$-module 
$M$, prove that the complex $\Otimes{K}M$ is a 
DG $K$-module via the multiplication
$a(b\otimes m):=(ab)\otimes m$.
More generally, given an $R$-complex $X$ and a DG $R$-algebra $A$,
prove that
the complex $\Otimes{A}X$ is a 
DG $A$-module via the multiplication
$a(b\otimes x):=(ab)\otimes x$.
\end{exer}

\begin{exer}\label{exer130331a}
Let $A$ be a DG $R$-algebra,  let $M$ be a DG $A$-module, and let $i\in\bbz$.
Prove that $\shift^iM$ is a DG $A$-module.
\end{exer}

The next exercise further foreshadows important aspects of Section~\ref{sec120522f}.

\begin{exer}\label{fact110216a'}
Let $A$ be a DG $R$-algebra, and let $M$ be a DG $A$-module.
Prove that there is a well-defined
chain map $\mu^M\colon \Otimes AM\to M$ given by $\mu^M(a\otimes m)=am$.
\end{exer}

We consider the following example throughout these notes. 
It is simple but demonstrates our constructions.
And even it has some non-trivial surprises.

\begin{ex}\label{ex120526a}
We consider the trivial Koszul complex $U=K^R(0)$:
$$U=\quad 0\to Re\xra 0 R1\to 0.$$
The notation indicates that we are using the basis $e\in U_1$ and $1=1_U\in U_0$.

Exercise~\ref{ex120526b} shows that $R$ is a DG $U$-module.
Another example is the following, again with specified basis in each degree:
$$G=\cdots \xra 1 Re_3\xra 0 R1_2\xra 1Re_1\xra 0R1_0\to 0.$$
The notation for the bases is chosen to help remember the DG $U$-module
structure:
\begin{align*}
1\cdot 1_{2n}&=1_{2n}
&1\cdot e_{2n+1}&=e_{2n+1}
\\
e\cdot 1_{2n}&=e_{2n+1}
&e\cdot e_{2n+1}&=0.
\end{align*}
One checks directly that $G$ satisfies the axioms to be a DG $U$-module.
It is worth noting that $\HH_0(G)\cong R$ and $\HH_i(G)=0$ for all $i\neq 0$.\footnote{For perspective, $G$ is modeled on the  free resolution 
$\cdots \xra{e}R[e]/(e^2)\xra{e}R[e]/(e^2)\to 0$ of $R$ over $R[e]/(e^2)$.}
\end{ex}

We continue with Example~\ref{ex120526a}, but working over 
a field $F$ instead of $R$.

\begin{ex}\label{ex120524e}
We consider the trivial Koszul complex $U=K^F(0)$:
$$U=\quad 0\to Fe\xra 0 F1\to 0.$$
Consider the graded vector space $W=\bigoplus_{i\in\bbz}W_i$, where $W_0=F\eta_0\cong F$ with basis $\eta_0$
and $W_i=0$ for $i\neq 0$:
$$W=\quad 0\bigoplus F\nu_0\bigoplus 0.$$
We are interested in identifying all the possible DG $U$-module structures on $W$,
that is, all possible differentials 
$$0\to F\nu_0\to 0$$
and rules for scalar multiplication making this into a DG $U$-module.
See Section~\ref{sec120522f} for more about this.

The given vector space $W$ has exactly one DG $U$-module structure.
To see this, first note that we have no choice for the differential
since it maps $W_i\to W_{i-1}$ and at least one of these modules is 0; hence 
$\partial_i=0$ for all $i$.
Also, we have no choice for the scalar multiplication: 
multiplication by $1=1_U$ must be the identity,
and multiplication by $e$ maps $W_i\to W_{i+1}$
and at least one of these modules is 0.
(This example is trivial, but it will be helpful later.)

Similarly, we consider the graded vector space
$$W'=\quad 0\bigoplus F\eta_1\bigoplus F\eta_0\bigoplus 0.$$
This vector space allows for one possibly non-trivial differential
$$\partial'_1\in\Hom[F]{F\eta_1}{F\eta_0}\cong F.$$
So, in order to make $W'$ into an $R$-complex, we need to choose an element $x_1\in F$:
$$ (W',x_1)=\quad 0\to F\eta_1\xra{x_1} F\eta_0\to 0.$$
To be explicit, this means that $\partial'_1(\eta_1)=x_1\eta_0$, and hence $\partial'_1(r\eta_1)=x_1r\eta_0$ for all $r\in F$.
Since $W'$ is concentrated in degrees 0 and 1, this is  an $R$-complex.

For the scalar multiplication of $U$ on the complex $(W',x_1)$, again multiplication by 1 must be the identity,
but multiplication by $e$ has one nontrivial option
$$\mu'_0\in\Hom[F]{F\eta_0}{F\eta_1}\cong F$$
which we identify with an element $x_0\in F$.
To be explicit, this means that $e\eta_0=x_0\eta_1$, and hence $er\eta_0=x_0r\eta_1$ for all $r\in F$.

For the Leibniz Rule to be satisfied, we must have
\begin{align*}
\partial'_{i+1}(e\cdot \eta_i)
&=\partial^U_1(e)\cdot \eta_i+(-1)^{|e|}e\cdot \partial'_i(\eta_i)
\intertext{for $i=0,1$. We begin with $i=0$:}
\partial'_{1}(e\cdot \eta_0)
&=\partial^U_1(e)\cdot \eta_0+(-1)^{|e|}e\cdot \partial'_0(\eta_0)\\
\partial'_{1}(x_0\eta_1)&=0\cdot \eta_0-e\cdot 0\\
x_0\partial'_{1}(\eta_1)&=0\\
x_0x_1\eta_0&=0
\intertext{so we have $x_0x_1=0$, that is, either $x_0=0$ or $x_1=0$.
For $i=1$, we have}
\partial'_{2}(e\cdot \eta_1)&=\partial^U_1(e)\cdot \eta_1+(-1)^{|e|}e\cdot \partial'_1(\eta_1)
\\
0&=0\cdot \eta_1-e\cdot (x_1\eta_0)
\\
0&=-x_1e\cdot \eta_0
\\
0&=-x_1x_0 \eta_1
\end{align*}
so we again conclude that $x_0=0$ or $x_1=0$.
One can check the axioms from Definition~\ref{defn130313a} to see that either of these choices
gives rise to a DG $U$-module structure on $W'$.
In other words, the DG $U$-module structures on $W'$ are parameterized by the following algebraic subset of $F^2=\bba^2_F$
$$\{(x_0,x_1)\in\bba^2_F\mid x_0x_1=0\}=V(x_0)\cup V(x_1)$$
which is the union of the two coordinate axes in $\bba^2_F$.
This is one of the fundamental points of Section~\ref{sec120522f}, that DG module structures on a fixed finite-dimensional graded vector
space are parametrized by algebraic varieties.
\end{ex}

Homologically finite DG modules, defined next, take the place of finitely generated modules in our passage to the DG universe.

\begin{defn}
\label{defn110218b}
Let $A$ be a DG $R$-algebra. A DG $A$-module
$M$ is \emph{bounded below} if $M_n=0$ for all $n\ll 0$;
and it is \emph{homologically finite} if each $\HH_0(A)$-module $\HH_n(M)$ is finitely generated
and $\HH_n(M)=0$ for $|n|\gg 0$.
\end{defn}

\begin{ex}\label{ex120526c}
In Exercise~\ref{ex120526b}, the
DG $K$-module $R/(\x)$ is bounded below and homologically finite.
In Example~\ref{ex120526a}, the
DG $U$-modules $R$ and $G$ are bounded below and homologically finite.
In Example~\ref{ex120524e}, the
DG $U$-module structures on $W$ and $W'$ are bounded below and homologically finite.
\end{ex}

\subsection*{Morphisms of DG Modules}
In the passage from modules and complexes to DG modules, homomorphisms and chain maps are replaced with morphisms.

\begin{defn}\label{defn120523b}
A \emph{morphism} of DG $A$-modules is a chain map
$f\colon M\to N$ between DG $A$-modules that respects scalar multiplication:
$f(am)=af(m)$.
Isomorphisms in the category of DG $A$-modules are identified by the
symbol $\cong$. 
A \emph{quasiisomorphism} of DG $A$-modules is a morphism $M\to N$ such that each induced map
$\HH_i(M)\to\HH_i(N)$ is an isomorphism, i.e., a morphism of DG $A$-modules that is
a quasiisomorphism of $R$-complexes; these are identified by the symbol $\simeq$.
\end{defn}

\begin{disc}\label{ex110218a}
A morphism 
of DG $R$-modules is simply a chain map, and  a quasiisomorphism 
of DG $R$-modules is simply a quasiisomorphism in the sense of
Definition~\ref{defn120522d}.
Given a DG $R$-algebra $A$, a morphism of DG $A$-modules is an isomorphism if and only if 
it is injective and surjective.
\end{disc}

The next exercise indicates that base change is a functor; see Exercise~\ref{ex120526b}.

\begin{exer}
\label{exer130313a}
Let $\x=x_1,\cdots,x_n\in R$, and set $K=K^R(\x)$. 
\begin{enumerate}[(a)]
\item
\label{exer130313a1}
Given an $R$-module homomorphism
$f\colon M\to N$, prove that the chain map $\Otimes{K}f\colon \Otimes{K}M\to \Otimes{K}N$ 
is a morphism of
DG $K$-modules.
More generally, given a chain map of $R$-complexes 
$g\colon X\to Y$ and a DG $R$-algebra $A$, prove that
the chain map $\Otimes{A}g\colon \Otimes{A}X\to \Otimes{A}Y$ 
is a morphism of
DG $A$-modules. 
\item
\label{exer130313a2}
Give an example showing that if $g$ is a quasiisomorphism, then $\Otimes Ag$ need not be a 
quasiisomorphism. (Note that if $A_i$ is is $R$-projective for each $i$
(e.g., if $A=K$),
then $g$ being a quasiisomorphism implies that $\Otimes Ag$ is a 
quasiisomorphism by Fact~\ref{fact120523a}.)
\item
\label{exer130313a3}
Prove that the natural map $K\to R/(\x)$ is a morphism of DG $K$-modules.
More generally, prove that every morphism $A\to B$ of DG $R$-algebras is a morphism
of DG $A$-modules, where $B$ is a DG $A$-module via restriction of scalars.
\end{enumerate}
\end{exer}

Next, we use our running example to provide some morphisms of DG modules.

\begin{ex}\label{ex120526d}
We continue with the notation of Example~\ref{ex120526a}.

Let $f\colon G\to\shift R$ be a morphism of DG $U$-modules:
$$
\xymatrix{
G=\!\!\!\!\!\!\ar[d]_-f&\cdots \ar[r]^-1 &Re_3\ar[r]^-0 &R1_2\ar[r]^-1\ar[d]
&Re_1\ar[r]^-0\ar[d]_-{f_1}&R1_0\ar[r]\ar[d]&0
\\
\shift R&&&0\ar[r]&R\ar[r]&0.}
$$
Commutativity of the first square shows that $f=0$. 
One can also see this from the following computation:
$$f_1(e_1)=f_1(e\cdot 1_0)
=ef_0(1_0)=0.$$
That is, the only morphism of DG $U$-modules $G\to\shift R$ is the zero-morphism.
The same conclusion holds for any morphism $G\to\shift^{2n+1} R$ with $n\in\bbz$.

On the other hand, for each $n\in\bbn$, every element $r\in R$ determines a
morphism $g^{r,n}\colon G\to\shift^{2n} R$, via multiplication by $r$.
For instance in the case $n=1$:
$$
\xymatrix{
G=\!\!\!\!\!\!\ar[d]_-{g^{r,1}}&\cdots \ar[r]^-1 &Re_3\ar[r]^-0 \ar[d]&R1_2\ar[r]^-1\ar[d]_-{g_2^{r,1}}^{=r\cdot}
&Re_1\ar[r]^-0\ar[d]&R1_0\ar[r]&0
\\
\shift^2 R=&&0\ar[r]&R\ar[r]&0.}
$$
Each square commutes, and the linearity condition is from the next computations:
\begin{gather*}
g_2^{r,1}(1\cdot 1_2)=g_2^{r,1}(1_2)=r=1\cdot r=1\cdot g_2^{r,1}(1_2)
\\
g_2^{r,1}(e\cdot e_1)=g_2^{r,1}(0)=0=e\cdot 0=e\cdot g_1^{r,1}(e_1)\\
g_3^{r,1}(e\cdot 1_2)=g_3^{r,1}(e_3)=0=e\cdot r=e\cdot g_2^{r,1}(1_2).
\end{gather*}
Further, the isomorphism $\Hom RR\cong R$ shows that each morphism $G\to\shift^{2n} R$
is of the form $g^{r,n}$.
Also, one checks readily that the  map
$G\xra{g^{u,n}}\shift^{2n}R$ is a quasiisomorphism for each unit $u$ of $R$.
\end{ex}

\subsection*{Truncations of DG Modules}
The next operation allows us to replace a given DG module with a ``shorter'' one;
see Exercise~\ref{disc110302a}\eqref{disc110302a2}.

\begin{defn}\label{truncations}
Let $A$ be a DG $R$-algebra, and let $M$ be a DG $A$-module.
The \emph{supremum} of $M$ is
$$\sup(M):=\sup\{i\in\bbz\mid\HH_i(M)\neq 0\}.$$
Given an integer $n$, the \emph{$n$th soft left truncation of $M$} is the complex
$$
\tau(M)_{(\leq n)}:=\quad 0\to M_n/\im(\partial^M_{n+1})\to M_{n-1}\to M_{n-2}\to \cdots
$$
with differential induced by $\partial^M$.
\end{defn}

\begin{ex}\label{ex120526h}
We continue with the notation of Example~\ref{ex120526a}.
For each $n\geq 1$, set $i=\lfloor n/2\rfloor$ and prove that
$$\tau(G)_{(\leq n)}=\quad
0\to R1_{2i}\xra 1Re_{2i-1}
\xra 0 \cdots \xra 1 Re_3\xra 0 R1_2\xra 1Re_1\xra 0R1_0\to 0.$$
\end{ex}

\begin{disc}\label{disc120530a}
Let $P$ be a projective resolution of an $R$-module $M$,
with $P^+$ denoting the augmented resolution as in Example~\ref{ex120522a}.
Then one has
$\tau(P)_{(\leq 0)}\cong M$.
In particular, $P$ is a projective resolution of $\tau(P)_{(\leq 0)}$.
\end{disc}

\begin{exer}\label{disc110302a}
Let $A$ be a DG $R$-algebra,  let $M$ be a DG $A$-module,
and let $n\in\bbz$.
\begin{enumerate}[(a)]
\item
\label{disc110302a1}
Prove that the truncation $\tau(M)_{(\leq n)}$ is a DG $A$-module with the induced
scalar multiplication,
and the natural chain map
$M\to \tau(M)_{(\leq n)}$ is a morphism  of DG $A$-modules.
\item
\label{disc110302a2}
Prove that the morphism
from part~\eqref{disc110302a1} is a quasiisomorphism if and only if $n\geq \sup(M)$.
\end{enumerate}
\end{exer}

\subsection*{DG Algebra Resolutions}
The following fact provides the final construction needed to
give an initial sketch of the proof of Theorem~\ref{intthm120521a}.

\begin{fact}\label{fact120523b}
Let $Q\to R$ be a ring epimorphism. Then there is a
quasiisomorphism $A\xra\simeq R$ of DG $Q$-algebras
such that each $A_i$ is finitely generated and projective over $R$
and $A_i=0$ for $i>\pd_Q(R)$.
See, e.g., \cite[Proposition 2.2.8]{avramov:ifr}.
\end{fact}

\begin{defn}
\label{defn130331}
In Fact~\ref{fact120523b}, the quasiisomorphism $A\xra\simeq R$ is 
a \emph{DG algebra resolution} of $R$ over $Q$. 
\end{defn}

\begin{disc}\label{disc130625a}
When $\y\in Q$ is a $Q$-regular sequence, the Koszul complex
$K^Q(\y)$ is a DG algebra resolution of $Q/(\y)$ over $Q$ by Lemma~\ref{lem120529a} and Example~\ref{ex130331a}.
Section~\ref{sec120522a} contains other famous examples.
\end{disc}

\begin{ex}\label{ex130625a}
Let $Q$ be a  ring, and consider an ideal $I\subsetneq Q$.
Assume that the quotient $R:=Q/I$ has $\pd_Q(R)\leq 1$.
Then every projective resolution of $R$ over $Q$ of the form $A=(0\to A_1\to Q\to 0)$
has the structure of a DG algebra resolution.
\end{ex}

We conclude this section with the beginning of the proof of Theorem~\ref{intthm120521a}.
The rest of the proof is contained in~\ref{proof120523a} and~\ref{proof120523c}.

\begin{para}[\textbf{First part of the  Proof of Theorem~\ref{intthm120521a}}]
\label{proof120523b}
There is a flat
local ring homomorphism $R\to R'$ such that
$R'$ is complete with algebraically closed residue field, as in the proof of Theorem~\ref{intthm120522b}.
Since there is a 1-1 function $\s_0(R)\into\s_0(R')$ by Fact~\ref{fact120522a}, we can
replace $R$ with $R'$ and
assume without loss of generality that $R$ is complete with algebraically
closed residue field.

Since $R$ is complete and local, Cohen's structure theorem
provides a ring epimorphism $\tau\colon (Q,\n,k)\to (R,\m,k)$
where $Q$ is a complete regular local ring
such that $\m$ and $\n$ have the same minimal number of generators.
Let $\y=y_1,\ldots,y_n\in\n$ be a minimal generating sequence for $\n$,
and set $\x=x_1,\ldots,x_n\in\m$ where $x_i:=\tau(y_i)$.
It follows that we have $K^R(\x)\cong\Otimes[Q]{K^Q(\y)}{R}$.
Since $Q$ is regular and $\y$ is a minimal generating sequence for $\n$,
the Koszul complex $K^Q(\y)$ is a minimal $Q$-free resolution of $k$
by Lemma~\ref{lem120529a}.

Fact~\ref{fact120523b} provides
a DG algebra resolution $A\xra\simeq R$ of $R$ over $Q$.
Note that $\pd_Q(R)<\infty$ since $Q$ is regular.
We consider the following diagram of morphisms of DG $Q$-algebras:
\begin{equation}\label{eq120523a}
R\to K^R(\x)\cong\Otimes[Q]{K^Q(\y)}{R}\xla\simeq\Otimes[Q]{K^Q(\y)}{A}\xra\simeq \Otimes[Q]{k}{A}=:U.
\end{equation}
The first map is from Exercise~\ref{ex120523a'''}.
The isomorphism is from the previous paragraph. 
The first quasiisomorphism comes from an application of
$\Otimes[Q]{K^Q(\y)}-$ to the quasiisomorphism $R\xla\simeq A$,
using Fact~\ref{fact120523a}.
The second quasiisomorphism comes from an application of
$\Otimes[Q]-A$ to the quasiisomorphism $K^Q(\y)\xra\simeq k$.
Note that $\Otimes[Q]{k}{A}$ is a finite dimensional DG $k$-algebra
because of the assumptions on $A$.

We show in~\ref{proof120523a} below how this provides a diagram 
\begin{equation}\label{eq120523b}
\s_0(R)\into\s(R)\xra{\equiv} \s(K^R(\x))
\xla\equiv\s(\Otimes[Q]{K^Q(\y)}{A})\xra\equiv \s(U)
\end{equation}
where $\equiv$ identifies bijections of sets.
We then show in~\ref{proof120523c} 
that $\s(U)$ is finite, and it follows that
$\s_0(R)$ is finite, as desired.
\end{para}

\section{Examples of Algebra Resolutions}
\label{sec120522a}

\newcommand{\Pf}{\operatorname{Pf}}
\newcommand{\dell}{\partial}

Remark~\ref{disc130625a} and Example~\ref{ex130625a} 
provide constructions of DG algebra resolutions, in particular, for rings of projective dimension at most 1. 
The point of this section is to extend this to rings of projective dimension 2 and  rings of projective dimension 3 determined by
Gorenstein ideals.

\begin{defn}
Let $I$ be an ideal of a local ring $(R,\fm)$.  The \emph{grade} of $I$ in $R$, denoted $\grade_R(I)$, is defined to be the length of the longest regular sequence of $R$ contained in $I$.  Equivalently, we have
\[
\grade_R(I):=\min\left\{i\mid\operatorname{Ext}_R^i(R/I,R)\neq 0\right\}
\]
and it follows that $\grade_R(I)\leq\pd_R(R/I)$.
We say that $I$ is \emph{perfect} if $\grade_R(I)=\pd_R(R/I)<\infty$.  In this case, $\operatorname{Ext}_R^i(R/I,R)$ is non-vanishing precisely when $i=\pd_R(R/I)$.  If, in addition, this single non-vanishing cohomology module is isomorphic to $R/I$, then $I$ is said to be \emph{Gorenstein}.
\end{defn}

\begin{notn}
Let $A$ be a matrix\footnote{While much of the work in this section can be done basis-free, 
the formulations are somewhat more transparent  when bases are specified and matrices are used to represent homomorphisms.}
over $R$ and $J,K\subset\bbn$.
The submatrix of $A$ obtained by deleting columns indexed by $J$ and rows indexed by $K$ is denoted $A^J_K$.  For simplicity, we abbreviate $A^{\emptyset}_{\{i\}}$  as $A_i$, and so on. Let $I_n(A)$ be the ideal of $R$ generated by the ``$n\times n$ minors'' of $A$,
that is, the determinants of the $n\times n$ matrices of the form $A^J_K$.
\end{notn}

\subsection*{Resolutions of length two}\label{sec:length2}

The following result, known as the Hilbert-Burch Theorem, provides a characterization of perfect ideals of grade two.  It was first proven by Hilbert in 1890 in the case that $R$ is a polynomial ring \cite{hilbert:taf}; the more general statement was proven by Burch in 1968 \cite[Theorem 5]{burch:ifhdlr}.

\begin{thm}[\protect{\protect{\cite{burch:ifhdlr,hilbert:taf}}}]\label{thm:codim2}
Let $I$ be an ideal of the local ring $(R,\m)$.
\begin{enumerate}[\rm(a)]
\item \label{item130625a}
If $\pd_R(R/I)=2$, then 
\begin{enumerate}[\rm(1)]
\item 
there is a non-zerodivisor $a\in R$ such that
$R/I$ has a projective resolution of the form
$0\to R^n\xra{A} R^{n+1}\xra{B} R\to 0$
where
$B$ is the $1\times (n+1)$ matrix with $i$th column given by  $(-1)^{i-1}a\det(A_i)$,
\item
one has $I=aI_n(A)$, and
\item
the ideal $I_n(A)$ is perfect of grade 2.
\end{enumerate}
\item \label{item130625b}
Conversely, if $A$ is an $(n+1)\times n$ matrix over $R$ such that $\grade(I_n(A))\geq 2$,
then 
$R/I_n(A)$ has a projective resolution of the form
$0\to R^n\xra{A} R^{n+1}\xra{B} R\to 0$ where
$B$ is the $1\times (n+1)$ matrix with $i$th column given by  $(-1)^{i-1}\det(A_i)$.
\end{enumerate}
\end{thm}

\begin{exer}\label{exer:grade2perfect}
Using Theorem \ref{thm:codim2}, construct a grade two perfect ideal $I$ over the ring $R=k[\![x,y]\!]$.
\end{exer}

Herzog~\cite{herzog:kadla} showed that Hilbert-Burch resolutions 
can be endowed with DG algebra structures.

\begin{thm}[\protect{\cite{herzog:kadla}}]\label{thm:codim2dg}
Given an $(n+1)\times n$ matrix $A$ over $R$, let $B$ be the $1\times (n+1)$ matrix with $i$th column  given by $(-1)^{i-1}a\det(A_i)$ for some non-zerodivisor $a\in R$.  Then the  $R$-complex
\[
0\to \bigoplus_{\ell=1}^n Rf_\ell\xra{A}\bigoplus_{\ell=1}^{n+1} Re_\ell\xra{B}R1\to 0
\]
has the structure of a DG $R$-algebra with the following multiplication relations:
\begin{enumerate}[\rm(1)]
\item $e_i^2=0=f_jf_k$ and $e_if_j=0=f_je_i$ for all $i,j,k$, and
\item $\displaystyle e_ie_j=-e_je_i=a\sum_{k=1}^n(-1)^{i+j+k}\det(A^k_{i,j})f_k$ for all $1\leq i< j\leq n+1$.
\footnote{Note that the sign in this expression differs from the one found in~\cite[Example 2.1.2]{avramov:ifr}.}
\end{enumerate}
\end{thm}

\begin{exer}\label{exer130405a}
Verify the Leibniz rule for the product defined in Theorem \ref{thm:codim2dg}. 
\end{exer}

\begin{exer}\label{exer130405b}
Using the ideal $I$ from Exercise \ref{exer:grade2perfect}, construct a (minimal) free resolution for $R/I$, then specify the relations which give this resolution a DG $R$-algebra structure.
\end{exer}

\subsection*{Resolutions of length three}\label{sec:length3}

We now turn our attention to resolutions of length three. First, we recall the needed machinery.

\begin{defn}\label{def:altermap}
A square matrix $A$ over $R$ is \emph{alternating} if it is skew-symmetric and has all 0's on its diagonal. 
Let $A$ be an $n\times n$ alternating matrix over $R$. 
If $n$ is even, then there is an element $\Pf(A)\in R$ such that $\Pf(A)^2=\det(A)$.
If $n$ is odd then $\det(A)=0$, so we set $\Pf(A)=0$.
The element $\Pf(A)$ is called the \emph{Pfaffian} of $A$.
(See, e.g., \cite[Section 3.4]{bruns:cmr} for more details.)
We denote by $\Pf_{n-1}(A)$ the ideal of $R$ generated by the submaximal Pfaffians of $A$, that is,
\[
\Pf_{n-1}(A):=\left(\Pf(A^i_i)\mid 1\leq i\leq n\right)R.
\]
\end{defn}

\begin{ex}\label{ex130625c}
Let $x,y,z\in R$.
For the matrix $A=\left[\begin{smallmatrix*}[r]0&x\\-x&0\end{smallmatrix*}\right]$, we have
$\det(A)=x^2$, so $\Pf(A)=x$, and $\Pf_1(A)=0$.

For the matrix 
$B=\left[\begin{smallmatrix*}[r]0&x&y\\-x&0&\phantom{-}z\\-y&-z&0\end{smallmatrix*}\right]$
we have $\det(B)=0=\Pf(B)$ and
$$\Pf_2(B)=(\Pf(\left[\begin{smallmatrix*}[r]0&x\\-x&0\end{smallmatrix*}\right]),
\Pf(\left[\begin{smallmatrix*}[r]0&y\\-y&0\end{smallmatrix*}\right]),
\Pf(\left[\begin{smallmatrix*}[r]0&z\\-z&0\end{smallmatrix*}\right]))R=(x,y,z)R.$$
\end{ex}

Buchsbaum and Eisenbud~\cite{buchsbaum:asffr} study the structure of resolutions of length three. Specifically, they characterize such resolutions and exhibit their DG structure.

\begin{thm}[\protect{\cite[Theorem 2.1]{buchsbaum:asffr}}]\label{thm:codim3}
Let $I$ be an ideal of the local ring $(R,\m)$.
\begin{enumerate}[\rm(a)]
\item \label{item130625c}
If $I$ is Gorenstein and $\pd_R(R/I)=3$, then 
there is an odd integer $n\geq 3$ and an $n\times n$ matrix $A$ over $\m$
such that $I=\Pf_{n-1}(A)$ and $R/I$ has a minimal free resolution of the form
$0\to R\xra{B^T}R^n\xra{A} R^{n}\xra{B} R\to 0$
where
$B$ is the $1\times n$ matrix with $i$th column given by  $(-1)^{i}\Pf(A_i)$.
\item \label{item130625d}
Conversely, if $A$ is an $n\times n$ alternating matrix over $\m$ such that $\rank(A)=n-1$, then $n$ is odd and
$\grade(\Pf_{n-1}(A))\leq 3$; if $\grade(\Pf_{n-1}(A))= 3$
then 
$R/\Pf_{n-1}(A)$ has a minimal free resolution of the form
$$0\to R\xra{B^T}R^n\xra{A} R^{n}\xra{B} R\to 0$$
where
$B$ is the $1\times n$ matrix with $i$th column given by  $(-1)^{i-1}\Pf(A_i)$.
\end{enumerate}
\end{thm}

It follows from Theorem \ref{thm:codim3} that the minimal number of generators of a grade-3 Gorenstein ideal must be odd.  

\begin{ex}\label{ex130625b}
Let $R=k[\![x,y,z]\!]$ and consider the
alternating matrix 
$B=\left[\begin{smallmatrix*}[r]0&x&y\\-x&0&\phantom{-}z\\-y&-z&0\end{smallmatrix*}\right]$
from Example~\ref{ex130625c} with
$$I=\Pf_2(B)=(\Pf(\left[\begin{smallmatrix*}[r]0&x\\-x&0\end{smallmatrix*}\right]),
\Pf(\left[\begin{smallmatrix*}[r]0&y\\-y&0\end{smallmatrix*}\right]),
\Pf(\left[\begin{smallmatrix*}[r]0&z\\-z&0\end{smallmatrix*}\right]))R=(x,y,z)R.$$
Theorem \ref{thm:codim3} implies that 
a minimal free resolution of $R/I$ over $R$ is of the form
\[
0\to Rg\xra{\left[\begin{smallmatrix*}[r]z\\-y\\x\end{smallmatrix*}\right]}Rf_1\oplus Rf_2\oplus Rf_3\xra{\left[\begin{smallmatrix*}[r]0&x&y\\-x&0&\phantom{-}z\\-y&-z&0\end{smallmatrix*}\right]}Re_1\oplus Re_2\oplus Re_3\xra{\left[\begin{smallmatrix*}[c]z&-y&x\end{smallmatrix*}\right]}R1\to 0.
\]
(Compare this to the Koszul complex $K^R(x,-y,z)$.)
\end{ex}

The next two examples of Buchsbaum and Eisenbud \cite{buchsbaum:asffr} illustrate that there exists, for any odd $n\geq 3$, a grade-3 Gorenstein ideal which is $n$-generated.

\begin{ex}[\protect{\cite[Proposition 6.2]{buchsbaum:asffr}}]\label{ex:grade3Gor'}
Let $R=k[\![ x,y,z ]\!]$. For $n\geq 3$ odd, define $M_n$ to be the $n\times n$ alternating matrix whose entries \emph{above} the diagonal are 
\[
(M_n)_{ij}:=
\begin{cases}
x & \text{if $i$ is odd and $j=i+1$}\\
y & \text{if $i$ is even and $j=i+1$}\\
z & \text{if $j=n-i+1$}\\
0 & \text{otherwise.}
\end{cases}
\]
For instance, we have
\begin{align*}
M_3&=\left[\begin{matrix*}[r]0&x&z\\-x&0&\phantom{-}y\\-z&-y&0\end{matrix*}\right]
&\Pf_2(M_3)&=(x,y,z)R
\\
M_5&=\left[\begin{matrix*}[r]
0&x&0&0 &z\\
-x&\phantom{-}0&\phantom{-}y & \phantom{-}z &\phantom{-} 0 \\
0 & -y & 0 & x & 0 \\
0 & -z & -x & 0 & y \\
-z & 0 & 0 & -y & 0
\end{matrix*}\right]
&\Pf_4(M_5)&=(y^2,xz,xy+z^2,yz,x^2)R.
\end{align*}
Then $\Pf_{n-1}(M_n)$ is an $n$-generated Gorenstein ideal of grade 3 in $R$.
\end{ex}

\begin{ex}[\protect{\cite[Proposition 6.1]{buchsbaum:asffr}}]\label{ex:grade3Gor}
Let $M_n$  be a generic $(2n+1)\times (2n+1)$ alternating matrix over the ring $Q_n:=R\left[\!\left[x_{i,j}\mid 1\leq i<j\leq 2n+1\right]\!\right]$. That is:
\[
M_n:=\left[\begin{matrix*}[c]
0&x_{1,2}&x_{1,3}&\cdots &x_{1,2n+1}\\[0.05in]
-x_{1,2}&0&x_{2,3}&\cdots &x_{2,2n+1}\\[0.05in]
-x_{1,3}&-x_{2,3}&0&\\[0.05in]
\vdots &\vdots& &\ddots\\[0.05in]
-x_{1,2n+1}&-x_{2,2n+1}&&&0
\end{matrix*}\right]
\]
Then, for every $n\geq 1$, the ideal $\Pf_{2n}(M_n)$ is  Gorenstein  of grade 3 in $Q_n$.
\end{ex}

The next result provides an explicit description of the DG algebra structure on the resolution from Theorem~\ref{thm:codim3}.  
Note that our result is the more elemental version from~\cite[Example 2.1.3]{avramov:ifr}.
It is worth noting that Buchsbaum and Eisenbud show that every free resolution over $R$ of the form
$0\to F_3\to F_2\to F_1\to R\to 0$ has the structure of a DG $R$-algebra.

\begin{thm}[\protect{\cite[Theorem 4.1]{buchsbaum:asffr}}]\label{thm:codim3dg}
Let $A$ be an $n\times n$ alternating matrix over $R$, and let $B$ be the $1\times n$ matrix whose $i$th column is given by $(-1)^{i-1}\Pf(A_i^i)$.  Then the graded $R$-complex
\[
0\to Rg\xra{B^{\mathsf{T}}}\bigoplus_{\ell=1}^n Rf_\ell\xra{A}\bigoplus_{\ell=1}^n Re_\ell\xra{B}R1\to 0
\]
admits the structure of a DG $R$-algebra, with the following products:
\begin{enumerate}[\rm(1)]
\item $e_i^2=0$ and $e_if_j=f_je_i=\delta_{ij}g$ for all $1\leq i,j\leq n$, and
\item $\displaystyle e_ie_j=-e_je_i=\sum_{k=1}^n(-1)^{i+j+k}\rho_{ijk}\Pf(A_{ijk}^{ijk})f_k$ for all $1\leq i\neq j\leq n$, where $\rho_{ijk}=-1$ whenever $i<k<j$, and $\rho_{ijk}=1$ otherwise.
\end{enumerate}
\end{thm}

\begin{exer}\label{exer130405c}
Let $R=k[x,y,z]$ and consider the graded $R$-complex given by
\[
0\to Rg\xra{\left[\begin{smallmatrix*}[r]z\\-y\\x\end{smallmatrix*}\right]}Rf_1\oplus Rf_2\oplus Rf_3\xra{\left[\begin{smallmatrix*}[r]0&x&y\\-x&0&\phantom{-}z\\-y&-z&0\end{smallmatrix*}\right]}Re_1\oplus Re_2\oplus Re_3\xra{\left[\begin{smallmatrix*}[c]z&-y&x\end{smallmatrix*}\right]}R1\to 0
\]
Using Theorem \ref{thm:codim3dg}, write the product relations that give this complex the structure of a DG $R$-algebra.
\end{exer}

\subsection*{Longer resolutions}

In general, resolutions of length greater than 3 are not guaranteed to possess a DG algebra structure,
as the next example of Avramov shows.

\begin{ex}[\protect{\cite[Theorem 2.3.1]{avramov:ifr}}]\label{ex:codim4counterex}
Consider the local ring $R=k[\![ w,x,y,z]\!]$.  Then the minimal free resolutions over $R$ of 
the quotients
$R/(w^2,wx,xy,yz,z^2)R$
and
$R/(w^2,wx,xy,yz,z^2,wy^6,x^7,x^6z,y^7)R$ 
do not admit  DG $R$-algebra structures.
\end{ex}

On the other hand, resolutions of ideals $I$ with $\pd_R(R/I)\geq 4$ that are sufficiently nice do admit DG algebra structures.
For instance,
Kustin and Miller prove the following in~\cite[Theorem]{kustin:gacfct} and~\cite[4.3 Theorem]{kustin:asmrgrecf}.

\begin{ex}[\protect{\cite{kustin:gacfct,kustin:asmrgrecf}}]\label{ex:codim4Gor}
Let $I$ be a Gorenstein ideal of a local ring $R$.  If $\pd_R(R/I)=4$, then the minimal $R$-free resolution of $R/I$ has the structure of a DG $R$-algebra.
\end{ex}

\section{DG Algebras and  DG Modules II}
\label{sec120522d'}

In this section, we describe the notions needed to define semidualizing DG modules 
and to explain some of their base-change properties.
This includes a discussion of two types of Ext for DG modules.
The section concludes with another piece of the proof of 
Theorem~\ref{intthm120521a}; see~\ref{proof120523c}.

\begin{Convention}
Throughout this section, $A$ is a DG $R$-algebra,
and $L$, $M$, and $N$ are DG $A$-modules.
\end{Convention}

\subsection*{Hom  for DG Modules}
Given the fact that the semidualizing property for $R$-modules is defined in part by a Hom condition,
it should come as no surprise that we begin this section with the following.

\begin{defn}
\label{defn110223a}
Given an integer $i$, a \emph{DG $A$-module homomorphism of degree $n$} is an element
$f\in\Hom MN_n$ such that
$f_{i+j}(am)=(-1)^{ni}af_j(m)$
for all $a\in A_i$ and $m\in M_j$.
The graded submodule of $\Hom MN$
consisting of all DG $A$-module homomorphisms $M\to N$
is denoted $\Hom[A]MN$.
\end{defn}

Part~\eqref{fact110223a2} of the next exercise gives another hint of the semidualizing property for DG modules. 

\begin{exer}\label{fact110223a}
\begin{enumerate}[(a)]
\item
\label{fact110223a1}
Prove that  $\Hom[A]MN$ is a DG $A$-module 
via the action
$$(af)_j(m):=a (f_j(m))=(-1)^{|a||f|}f_{j+|a|}(am)
$$
and using the differential from $\Hom MN$.
\item
\label{fact110223a2}
Prove that for each $a\in A$ the multiplication map $\mult Ma\colon M\to M$
given by $m\mapsto am$
is a homomorphism of degree $|a|$.
\item
\label{fact110223a3}
Prove that $f\in\Hom[A]MN_0$ is a morphism if and only if it is a cycle, that is, if and only if $\partial^{\Hom[A]MN}_0(f)=0$.
\end{enumerate}
\end{exer}

\begin{ex}\label{ex120526e}
We continue with the notation of Example~\ref{ex120526a}.
From  computations like those in Example~\ref{ex120526d}, it follows
that $\Hom[U]GR$ has the  form
$$\Hom[U]GR
=\quad
0\to R\to 0\to R\to 0\to\cdots
$$
where the copies of $R$ are in even non-positive degrees.
Multiplication by $e$ is 0 on $\Hom[U]GR$, by degree considerations,
and multiplication by 1 is the identity.
\end{ex}

Next, we give an indication of the functoriality of $\Hom[A] N-$ and $\Hom[A] -N$. 

\begin{defn}\label{defn130330a}
Given a morphism  $f\colon  L\to M$ of DG $A$-modules,
we define the map $\Hom[A] Nf\colon\Hom[A] NL\to\Hom[A] NM$
as follows: each sequence $\{g_p\}\in\Hom[A] NL_n$ is mapped to $\{f_{p+n}g_p\}\in\Hom[A] NM_n$.
Similarly, define the map 
$\Hom[A] fN\colon\Hom[A] MN\to\Hom[A] LN$
by the formula $\{g_p\}\mapsto \{g_pf_p\}$.
\end{defn}

\begin{disc}
We do not use a sign-change in this definition
because $|f|=0$.
\end{disc}

\begin{exer}\label{exer130330a}
Given a morphism  $f\colon  L\to M$ of DG $A$-modules,
prove that the maps $\Hom[A] Nf$ and $\Hom[A] fN$ are well-defined morphisms of DG $A$-modules.
\end{exer}

\subsection*{Tensor Product for DG Modules}
As with modules and complexes, we use the tensor product to base change DG modules along a morphism
of DG algebras.

\begin{defn}
\label{defn110218c}
The \emph{tensor product} $\Otimes[A]MN$ is the quotient 
$(\Otimes MN)/U$ where $U$ is  generated over $R$ by the elements of the form
$\Otimes[]{(am)}{n}-(-1)^{|a| |m|}\Otimes[]{m}{(an)}$.
Given an element $\Otimes[]mn\in\Otimes MN$, we denote the image in $\Otimes[A]MN$ as $\Otimes[]mn$.
\end{defn}

\begin{exer}\label{fact110218b}
Prove that the tensor product $\Otimes[A]MN$ is a DG $A$-module
via the scalar multiplication
$$a(\Otimes[] mn):=\Otimes[]{(am)}{n}=(-1)^{|a| |m|}\Otimes[]{m}{(an)}.$$
\end{exer}

The next exercises describe base change and some canonical isomorphisms for DG modules. 

\begin{exer}\label{disc111223a}
Let $A\to B$ be a morphism of DG $R$-algebras.
Prove that $\Otimes[A]BM$ has the structure of a DG $B$-module
by the action $b(\Otimes[] {b'}m):=\Otimes[]{(bb')}{m}$.
Prove that this structure is compatible with the DG $A$-module structure
on $\Otimes[A]BM$ via restriction of scalars.
\end{exer}

\begin{exer}\label{fact120526a}
Verify the following isomorphisms of DG $A$-modules:
\begin{align*}
\Hom[A]AL&\cong L&&\text{Hom cancellation}
\\
\Otimes[A]AL&\cong L&&\text{tensor cancellation}
\\
\Otimes[A]LM&\cong\Otimes[A]ML&&\text{tensor commutativity}
\end{align*}
In particular, there are DG $A$-module isomorphisms $\Hom[A]AA\cong A\cong\Otimes[A]AA$.
\end{exer}

\begin{fact}\label{fact130330a}
There is a natural ``Hom tensor adjointness'' DG $A$-module isomorphism
$\Hom[A]{\Otimes[A] LM}{N}\cong\Hom[A]M{\Hom[A]LN}$.
\end{fact}

Next, we give an indication of the functoriality of $\Otimes[A] N-$ and $\Otimes[A] -N$. 

\begin{defn}\label{defn130330b}
Given a morphism $f\colon  L\to M$ of DG $A$-modules,
we define the map $\Otimes[A] Nf\colon\Otimes[A] NL\to\Otimes[A] NM$
by the formula $z\otimes y\mapsto z\otimes f(y)$.
Define the map 
$\Otimes[A] fN\colon\Otimes[A] LN\to\Otimes[A] MN$
by the formula $x\otimes z\mapsto f(x)\otimes z$.
\end{defn}

\begin{disc}
We do not use a sign-change in this definition
because $|f|=0$.
\end{disc}

\begin{exer}\label{exer130330b}
Given a morphism $f\colon  L\to M$ of DG $A$-modules,
prove that the maps $\Otimes[A] Nf$
and $\Otimes fN$
are well-defined morphisms of DG $A$-modules.
\end{exer}

\subsection*{Semifree Resolutions}
Given the fact that the semidualizing property includes an Ext-vanishing condition, it should come as
no surprise that we need a version of free resolutions in the DG setting.

\begin{defn}
A subset $E$ of $L$ is called a \emph{semibasis} if it
is a basis of the underlying $A^\natural$-module $L^ \natural $.
If $L$ is bounded below, then $L$ is called \emph{semi-free}
if it has a semibasis.\footnote{As is noted in~\cite{avramov:dgha}, when $L$ is not bounded below,
the definition of ``semi-free'' is  more technical.
However, our results do not require this level of generality, so we focus only on
this  case.}
A \emph{semi-free resolution} of a DG $A$-module $M$ is a quasiisomorphism
$F\xra{\simeq}M$ of DG $A$-modules such that $F$ is semi-free.
\end{defn}

The next exercises and example give some semi-free examples 
to keep in mind.

\begin{exer}\label{ex120524a}
Prove that a semi-free DG $R$-module is simply a 
bounded below complex of free $R$-modules.
Prove  that each free resolution $F$ of an $R$-module $M$
gives rise  to a semi-free resolution $F\xra\simeq M$; see Exercise~\ref{exer120522g}.
\end{exer}

\begin{exer}\label{ex120524b}
Prove that $M$ is exact (as an $R$-complex) if and only if
$0\xra\simeq M$ is a semi-free resolution. 
Prove that the DG $A$-module $\shift^nA$ is semi-free for each $n\in\bbz$, as is $\bigoplus_{n\geq n_0}\shift^nA^{\beta_n}$
for all $n_0\in\bbz$ and $\beta_n\in\bbn$.
\end{exer}

\begin{exer}\label{exer130330c}
Let $\x=x_1,\cdots,x_n\in R$, and set $K=K^R(\x)$. 
\begin{enumerate}[(a)]
\item\label{exer130330c1}
Given 
a bounded below complex $F$ of free $R$-modules, prove that
the complex $\Otimes{K}F$ is a semi-free DG $K$-module.
\item\label{exer130330c2}
If $F\xra\simeq M$ is a free resolution of an $R$-module $M$,
prove that $\Otimes{K}F\xra\simeq \Otimes{K}M$ is a 
semi-free resolution of the DG $K$-module $\Otimes{K}M$.
More generally, 
if $F\xra\simeq M$ is a semi-free resolution of a DG $R$-module $M$,
prove that $\Otimes{K}F\xra\simeq \Otimes{K}M$ is a 
semi-free resolution of the DG $K$-module $\Otimes{K}M$.
See Fact~\ref{fact120523a}.
\end{enumerate}
\end{exer}

\begin{ex}\label{ex120526f}
In the notation of Example~\ref{ex120526a},
the natural map $g^{1,0}\colon G\to R$ is a semi-free resolution of $R$ over $U$;
see Example~\ref{ex120526d}.
The following display indicates why $G$ is semi-free over $U$,
that is, why $\und G$ is free over $\und U$:
\begin{align*}
U=&\quad 0\to Re\xra 0 R1\to 0
\\
\und U=& \quad Re\bigoplus R1
\\
G=&\cdots \xra 1 Re_3\xra 0 R1_2\xra 1Re_1\xra 0R1_0\to 0
\\
\und G=&\cdots  (Re_3\bigoplus R1_2) \bigoplus (Re_1 \bigoplus R1_0).
\end{align*}
\end{ex}

The next item compares to Remark~\ref{disc120530a}.

\begin{disc}\label{disc120530a'}
If $L$ is  semi-free,
then the natural map $L\to\tau(L)_{(\leq n)}$ is a semi-free resolution 
for each $n\geq\sup(L)$.
\end{disc}

The next facts contain important existence results for semi-free resolutions.
Notice that the second paragraph applies when $A$ is a Koszul complex over $R$
or is finite dimensional over a field, by Exercise~\ref{exer130312b}.

\begin{fact}\label{fact110218d}
The DG $A$-module $M$ has a semi-free resolution 
if and only if $\HH_i(M)=0$ for $i\ll 0$, by~\cite[Theorem 2.7.4.2]{avramov:dgha}.

Assume that $A$ is noetherian, and let $j$ be an integer.
Assume that each module $\HH_i(M)$ is finitely generated over $\HH_0(A)$
and that $\HH_i(M)=0$ for $i<j$. Then $M$ has a semi-free resolution $F\xra\simeq M$ such that
$\und{F}\cong\bigoplus_{i=j}^\infty \shift^i(\und{A})^{\beta_i}$
for some integers $\beta_i$, and so $F_i=0$ for all $i<j$;
see~\cite[Proposition~1]{apassov:hddgr}.
In particular, homologically finite DG $A$-modules admit  ``degree-wise finite, bounded below'' semi-free resolutions.
\end{fact}

\begin{fact}\label{fact110218d'}
Assume that $L$ and $M$ are semi-free.
If there is a quasiisomorphism $L\xra\simeq M$, then there is also
a quasiisomorphism $M\xra\simeq L$
by~\cite[Proposition 1.3.1]{avramov:ifr}
\end{fact}

The previous fact explains why the next relations are symmetric. 
The fact that they are reflexive and transitive are straightforward to verify.

\begin{defn}\label{defn120523b'}
Two semi-free DG $A$-modules $L$ and $M$ are \emph{quasiisomorphic} if there is a 
quasiisomorphism $L\xra\simeq M$; this equivalence relation is
denoted by the symbol $\simeq$.
Two semi-free DG $A$-modules $L$ and $M$ are \emph{shift-quasiisomorphic}
if there is an integer $m$ such that $L\simeq\shift^mM$;
this equivalence relation is
denoted by the symbol $\sim$.
\end{defn}

\subsection*{Semidualizing DG Modules}

For Theorem~\ref{intthm120521a}, we use a version of  Christensen and Sather-Wagstaff's notion of semidualizing DG modules
from~\cite{christensen:dvke},
defined next.

\begin{defn}
\label{dfn:DGsdm}
The \emph{homothety morphism} $\chi^A_M\colon A\to \HomA MM$ is given by
$(\chi^A_M)_{|a|}(a):=\mult Ma$, i.e., $(\chi^A_M)_{|a|}(a)_{|m|}(m)=am$.

Assume that $A$ is noetherian. Then  $M$ is a \emph{semidualizing}
DG $A$-module if $M$ is homologically finite
and semi-free such that 
$\chi^A_M\colon A\to\Hom[A]{M}{M}$ is a quasiisomorphism.
Let $\s(A)$ denote the set of shift-quasiisomorphism classes 
of semidualizing DG $A$-modules,
that is, the set of equivalence classes of semidualizing DG $A$-modules 
under the relation $\sim$ from Definition~\ref{defn120523b'}.
\end{defn}

\begin{exer}\label{exer130331b}
Prove that the homothety morphism $\chi^A_M\colon A\to \HomA MM$ is
a well-defined morphism of DG $A$-modules.
\end{exer}

The following fact explains part of diagram~\eqref{eq120523b}.

\begin{fact}\label{ex120524c}
Let $M$ be an $R$-module with projective resolution $P$. 
Then Fact~\ref{exer120522e} shows that $M$ is a semidualizing
$R$-module if and only if $P$ is a semidualizing DG $R$-module.
It follows that, we have an injection $\s_0(R)\into\s(R)$.
\end{fact}

The next example gives some explanation for our focus on  shift-quasiisomorphism classes 
of semidualizing DG $A$-modules.

\begin{ex}\label{ex120524d}
Let  $B$ and $C$ be semi-free DG $A$-modules
such that $B\sim C$.
Then $B$ is semidualizing over $A$ if and only if $C$ is semidualizing over $A$.
The point  is the following.
The condition $B\simeq \shift^iC$ tells us that $B$ is homologically finite if
and only if $\shift^iC$ is homologically finite, that is, if
and only if $C$ is homologically finite.
Fact~\ref{fact110218d'}
provides a quasiisomorphism $B\xra[\simeq]{f} \shift^iC$.
Thus, there is a commutative diagram of morphisms of DG $A$-modules:
$$\xymatrix@C=2cm{
\Hom[A]CC\ar[rd]_{\cong}&
A\ar[l]_-{\chi^A_C}\ar[r]^-{\chi^A_B}\ar[d]_-{\chi^A_{\shift^iC}}
&\Hom[A]BB \ar[d]^-{\Hom[A] Bf}_-{\simeq}\\
&\Hom[A]{\shift^iC}{\shift^iC}\ar[r]^-{\Hom[A]f{\shift^iC}}_-{\simeq}
&\Hom[A]B{\shift^iC}. }$$
The unspecified isomorphism follows from a bookkeeping exercise.\footnote{The interested reader may wish to
show how this isomorphism is defined and to check the commutativity of the diagram. If this applies to you, make sure to mind the signs.}
The morphisms $\Hom[A]fC$ and $\Hom[A]Bf$ are quasiisomorphisms
by~\cite[Propositions 1.3.2 and 1.3.3]{avramov:ifr}
because 
$B$ and $\shift^iC$ are semi-free and $f$ is a quasiisomorphism.
It follows that $\chi^A_B$ is a quasiisomorphism
if and only if $\chi^A_{\shift^iC}$ is a quasiisomorphism
if and only if $\chi^A_C$ is a quasiisomorphism.
\end{ex}

The following facts explain other parts of diagram~\eqref{eq120523b}.

\begin{fact}\label{lem120524a}
Assume that $(R,\m)$ is local.
Fix a list of elements $\x \in\m$ and set $K=K^R(\x)$.
Base change  $\Otimes K-$ induces an injective map $\s(R)\into\s(K)$
by~\cite[A.3. Lemma]{christensen:dvke};
if $R$ is complete, then this map is  bijective by~\cite[Corollary 3.10]{nasseh:ldgm}.
\end{fact}

\begin{fact}\label{lem111117a'}
Let $\vf\colon A\xra\simeq B$ be a quasiisomorphism of noetherian DG $R$-algebras.
Base change  $\Otimes[A]B-$ induces a bijection from $\s(A)$ to $\s(B)$ by~\cite[Lemma 2.22(c)]{nasseh:lrfsdc}.
\end{fact}

\subsection*{Ext for DG Modules}
One subtlety in the proof of Fact~\ref{lem120524a} is found in the
behavior of Ext for DG modules, which we describe next.

\begin{defn}
\label{defn110223a'}
Given a semi-free resolution $F\xra\simeq M$, for each integer $i$ set 
$\Ext[A]iMN:=\HH_{-i}(\Hom[A]FN)$.\footnote{One
can also define $\Tor iMN:=\HH_i(\Otimes[A]FN)$, but we do not need
this here.}
\end{defn}

The next two items are included in our continued spirit of providing  perspective.

\begin{exer}\label{fact120524a}
Given $R$-modules $M$ and $N$, prove that the module $\Ext iMN$ 
defined in~\ref{defn110223a'} is the usual $\Ext iMN$;
see Exercise~\ref{ex120524a}.
\end{exer}

\begin{ex}\label{ex120526g}
In the notation of Example~\ref{ex120526a}, Examples~\ref{ex120526e} and~\ref{ex120526f} imply
$$\Ext[U]iRR=\HH_{-i}(\Hom[U]{G}{R})=\begin{cases}
R & \text{if $i\geq 0$ is even} \\
0 & \text{otherwise.}\end{cases}$$
Contrast this with the equality $\Ext[\und U]iRR=R$ for all $i\geq 0$.
This shows that $U$ is fundamentally different from $\und U\cong R[X]/(X^2)$,
even though $U$ is obtained using a trivial differential on $R[X]/(X^2)$
with the natural grading.
\end{ex}

The next result compares with the fact that $\Ext iMN$ is independent of the 
choice of free resolution when $M$ and $N$ are modules.

\begin{fact}\label{fact120524b}
For each index $i$, the module $\Ext[A]iMN$ is independent of the 
choice of semi-free resolution of $M$ by~\cite[Proposition 1.3.3]{avramov:ifr}.
\end{fact}

\begin{disc}\label{disc120524a}
An important fact about $\Ext 1MN$ for $R$-modules $M$ and $N$ is the following:
the elements of $\Ext 1MN$ are in bijection with the equivalence classes of 
short exact sequences (i.e., ``extensions'') of the form
$0\to N\to X\to M\to 0$. 
For DG modules over a DG $R$-algebra $A$, things are a bit more subtle.

Given DG $A$-modules $M$ and $N$, one  defines the notion of a 
short exact sequence of the form
$0\to N\to X\to M\to 0$ in the naive way: the arrows are morphisms
of DG $A$-modules such that for each $i\in\bbz$ the sequence
$0\to N_i\to X_i\to M_i\to 0$ is exact. One defines an equivalence relation
on the set of short exact sequences of this form
(i.e., ``extensions'') in the natural way:
two extensions $0\to N\xra{f} X\xra{g} M\to 0$ and $0\to N\xra{f'} X'\xra{g'} M\to 0$
are equivalent if there is a commutative diagram 
$$\xymatrix{
0\ar[r] 
& N\ar[r]^-{f} \ar[d]_-{=}
& X\ar[r]^-{g} \ar[d]_-{h}
& M\ar[r] \ar[d]_-{=}
& 0
\\
0\ar[r] 
& N\ar[r]^-{f'} 
& X'\ar[r]^-{g'} 
& M\ar[r] 
& 0
}$$
of morphisms of DG $A$-modules. 
Let $\yext[A]1MN$ denote the set of equivalence classes of such extensions. (The ``Y'' is for ``Yoneda''.)
As with $R$-modules,
one can define an abelian group structure on $\yext[A]1MN$.
However, in general one has $\yext[A]1MN\ncong\Ext[A]1MN$,
even when $A=R$, as the next example shows.
\end{disc}

\begin{ex}\label{ex111018a}
Let $R=k[\![X]\!]$, and consider the following exact sequence of DG $R$-modules,
i.e., exact sequence of $R$-complexes:
$$\xymatrix@R=5mm{
0\ar[r]
&\underline R\ar[r]
&\underline R\ar[r]
&\underline k\ar[r]
&0 \\
& 0\ar[d]
& 0\ar[d]
& 0\ar[d] \\
0\ar[r]
& R\ar[r]^-X\ar[d]_-1
& R\ar[r]\ar[d]_-1
& k\ar[r]\ar[d]_-1
&0 \\
0\ar[r]
& R\ar[r]^-X\ar[d]
& R\ar[r]\ar[d]
& k\ar[r]\ar[d]
&0 \\
&0&0&0.
}$$
This sequence does not split over $R$ (it is not even degree-wise split) so 
it gives a non-trivial class in $\yext 1{\underline k}{\underline R}$,
and we conclude that $\yext 1{\underline k}{\underline R}\neq 0$.
On the other hand, $\underline k$ is homologically trivial,
so we have $\Ext 1{\underline k}{\underline R}=0$ since
$0\xra\simeq\underline k$ is a semi-free resolution.
\end{ex}

For our proof of Theorem~\ref{intthm120521a}, the following connection between
Ext and YExt is quite important; see~\cite[Corollary 3.8 and Proposition 3.12]{nasseh:lrfsdc}. 

\begin{fact}
\label{prop111110a}
If $L$ is semi-free,
then $\yext[A]1LM\cong\Ext[A]1LM$;
if furthermore $\Ext 1LL=0$, then for each $n\geq\sup(L)$, one has
$$\yext[A]1LL=0=\yext[A]1{\tau(L)_{(\leq n)}}{\tau(L)_{(\leq n)}}.$$
\end{fact}

We conclude this section with the second part of the proof of Theorem~\ref{intthm120521a}.
The rest of the proof is contained in~\ref{proof120523c}.

\begin{para}[\textbf{Second part of the proof of Theorem~\ref{intthm120521a}}]
\label{proof120523a}
We continue with the notation established in~\ref{proof120523b}.
The properties of diagram~\eqref{eq120523b} follow from
diagram~\eqref{eq120523a}
because of Facts~\ref{lem120524a} and~\ref{lem111117a'}.
Thus, it remains to show that $\s(\Otimes[Q]{k}{A})$ is finite.
This is shown in~\ref{proof120523c}.
\end{para}

\section{A Version of Happel's Result for DG Modules}
\label{sec120522f}

This section contains the final steps of the proof of Theorem~\ref{intthm120521a}; see~\ref{proof120523c}.
The idea, from~\cite{SGA3, gabriel:frto,happel:sm,voigt:idteag}
is  bold one: use algebraic geometry to study all possible module structures on a fixed set.
A simple case of this is in Example~\ref{ex120524e}.
We begin with some notation for use throughout this section.

\begin{notation}\label{notn110516a}
Let $F$ be an algebraically closed field, and let
$$
U:=(0\to U_q\xra{\partial_q^U} U_{q-1}\xra{\partial_{q-1}^U} \cdots \xra{\partial_{1}^U} U_{0}\to 0)
$$
be a finite-dimensional  DG $F$-algebra.
Let $\dim_F(U_i)=n_i$ for  $i=0,\ldots, q$. Let
$$
W:=\bigoplus_{i=0}^s W_i
$$
be a graded $F$-vector space with $r_i:=\dim_F(W_i)$ for  $i=0,\ldots, s$.

A DG $U$-module structure on $W$ consists of two pieces of data.
First, we need a differential $\partial$.
Second, once the differential $\partial$ has been chosen, we need a scalar multiplication $\mu$.
Let $\od^U(W)$ denote the set of all ordered pairs $(\partial,\mu)$
making $W$ into a DG $U$-module.
Let $\en_F(W)_0$ denote the set of  $F$-linear endomorphisms of $W$ that are homogeneous of degree 0.
Let $\ggl(W)_0$ denote the set of  $F$-linear automorphisms of $W$ that are homogeneous of degree 0, that is, the invertible elements of $\en_F(W)_0$.
\end{notation}

We next describe a geometric structure on  $\od^U(W)$,
as in Example~\ref{ex120524e}.

\begin{anal}\label{disc110516a}
We work in the setting of Notation~\ref{notn110516a}.

A differential $\partial$ on $W$ is an element of the graded vector space
$\Hom[F]WW_{-1}=\bigoplus_{i=0}^s\Hom[F]{W_i}{W_{i-1}}$ such that $\partial\partial=0$.
The vector space $\Hom[F]{W_i}{W_{i-1}}$ has dimension $r_ir_{i-1}$,
so the map $\partial$ corresponds to an element of the affine space $\mathbb{A}^{d}_F$ where
$d:=\sum_i r_ir_{i-1}$. The vanishing condition $\partial\partial=0$ is equivalent to the entries of the matrices
representing $\partial$ satisfying certain fixed homogeneous quadratic polynomial equations over $F$. Hence, the
set of all  differentials on $W$ is a Zariski-closed subset of $\mathbb{A}^{d}_F$.

Once the differential $\partial$ has been chosen, a scalar multiplication $\mu$ is in particular a cycle
in $\Hom[F]{U\otimes_F W}{W}_0=\bigoplus_{i,j} \Hom[F]{U_i\otimes_F W_j}{W_{i+j}}$.
For all $i,j$, the vector space $\Hom[F]{U_i\otimes_F W_j}{W_{i+j}}$ has dimension $n_ir_jr_{i+j}$, so the map $\mu$
corresponds to an element of the affine space $\mathbb{A}^{d'}_F$ where
$d':= \sum_{i,j} n_ir_{j}r_{i+j}$. The condition that $\mu$ be an associative, unital cycle
is equivalent to the entries of the matrices
representing $\partial$ and $\mu$ satisfying certain fixed polynomials over $F$. Thus, the
set $\od^U(W)$  is a Zariski-closed subset of $\mathbb{A}^{d}_F\times \mathbb{A}^{d'}_F\cong \mathbb{A}^{d+d'}_F$.
\end{anal}

\begin{ex}\label{ex120524z}
We continue with the notation of Example~\ref{ex120524e}.
In this example, we have 
$\od^U(W)=\{1\}$ (representing the non-trivial scalar multiplication by 1)
and
$\od^U(W')=\{(x_1,x_0)\in F^2\mid x_1x_0=0\}$.
Re-writing $F^n$ as $\bba^n_F$, we see that
$\od^U(W)$ is a single point  in $\bba^1_F$
and $\od^U(W')$ is the union of the two coordinate axes
$V(x_1x_0)=V(x_0)\cup V(x_1)$.
\end{ex}

\begin{exer}\label{ex120524e'}
Continue with the notation of Example~\ref{ex120524e}.
Write out the coordinates and equations describing 
$\od^U(W'')$ and $\od^U(W''')$ where
\begin{align*}
W''&=\quad 0\bigoplus Fw_2\bigoplus Fw_1\bigoplus Fw_0\bigoplus 0\\
W'''&=\quad 0\bigoplus Fz_2\bigoplus (Fz_{1,1}\bigoplus Fz_{1,2})\bigoplus Fz_0\bigoplus 0.
\end{align*}
For scalar multiplication,
note that since multiplication by 1 is already determined by the $F$-vector
space structure, we only need to worry about multiplication by $e$
which maps
$W''_i\to W''_{i+1}$
and $W'''_i\to W'''_{i+1}$
for $i=0,1,2$.
\end{exer}

We next describe a geometric structure on the set $\ggl(W)_0$.

\begin{anal}\label{disc110519a}
We work in the setting of Notation~\ref{notn110516a}.

A map $\alpha\in\ggl(W)_0$  is an element of the graded vector space
$\Hom[F]WW_{0}=\bigoplus_{i=0}^s\Hom[F]{W_i}{W_{i}}$ with a multiplicative inverse. The vector space $\Hom[F]{W_i}{W_{i}}$ has dimension $r_i^2$,
so the map $\alpha$ corresponds to an element of the affine space $\mathbb{A}^{e}_F$ where
$e:=\sum_i r_i^2$. The invertibility of $\alpha$ is equivalent to the invertibility of  each ``block''
$\alpha_i\in\Hom[F]{W_i}{W_{i}}$, which is an open condition defined by the non-vanishing of the determinant polynomial.
Thus, the set $\ggl(W)_0$  is a Zariski-open subset of $\mathbb{A}^{e}_F$,
so it is smooth over $F$.

Alternately, one can view $\ggl(W)_0$ as the product $\ggl(W_0)\times\cdots\times\ggl(W_s)$.
Since each $\ggl(W_i)$ is an  algebraic group smooth over $F$, it follows that $\ggl(W)_0$ is also  
an  algebraic group that is smooth over $F$.
\end{anal}

\begin{ex}\label{ex120524y}
We continue with the notation of Example~\ref{ex120524e}.
It is straightforward to show that
\begin{align*}
\en_F(W)_0&=\Hom[F]{F\nu_0}{F\nu_0}\cong F=\bba^1_F\\
\ggl_F(W)_0&=\au_F(F\nu_0)\cong F^\times =U_x\subset\bba^1_F\\
\en_F(W')_0&=\Hom[F]{F\eta_1}{F\eta_1}\bigoplus\Hom[F]{F\eta_0}{F\eta_0}\cong F\times F=\bba^2_F\\
\ggl_F(W')_0&=\au_F(F\eta_1)\bigoplus\au_F(F\eta_0)\cong F^\times  \times F^\times
=U_{x_1x_0}\subset\bba^2_F.
\end{align*}
Here $U_x$ is the subset $\bba^1_F\ssm V(x)$,
and $U_{x_1x_0}=\bba^2_F\ssm V(x_1x_0)$.
\end{ex}

\begin{exer}\label{ex120524e''}
Continue with the notation of Example~\ref{ex120524e}.
Write out the coordinates and equations describing 
$\ggl^U(W'')_0$ and $\ggl^U(W''')_0$ where
$W''$ and $W'''$ are from Exercise~\ref{ex120524e'}
\end{exer}

Next, we describe an action of  $\ggl(W)_0$ on $\od^U(W)$.

\begin{anal}\label{action}
We work in the setting of Notation~\ref{notn110516a}.

Let $\alpha\in\ggl(W)_0$. For every $(\partial,\mu)\in\od^U(W)$, we define
$\alpha\cdot(\partial,\mu):=(\widetilde{\partial},\widetilde{\mu})$, where
$\widetilde{\partial}_i:=\alpha_{i-1}\circ \partial_i\circ \alpha_i^{-1}$ and
$\widetilde{\mu}_{i+j}:=\alpha_{i+j}\circ \mu_{i+j}\circ (U\otimes_F \alpha_j^{-1})$.
For the multiplication, this defines a new multiplication
$$u_i\cdot_{\alpha} w_j:=\alpha_{i+j}(u_i\cdot\alpha_j^{-1}(w_j))$$
where $\cdot$ is the multiplication given by $\mu$, as in Discussion~\ref{disc110516a}: 
$u_i\cdot w_j:=\mu_{i+j}(u_i\otimes w_j)$.
Note that this leaves multiplication by $1_A$ unaffected:
$$1_A\cdot_{\alpha} w_j=\alpha_{j}(1_A\cdot\alpha_j^{-1}(w_j))
=\alpha_{j}(\alpha_j^{-1}(w_j))=w_j.$$
It is routine to show that the ordered pair $(\widetilde{\partial},\widetilde{\mu})$ 
describes a
DG $U$-module structure for $W$, that is, we have
$\alpha\cdot(\partial,\mu):=(\widetilde{\partial},\widetilde{\mu})\in\od^U(W)$.
From the definition of $\alpha\cdot(\partial,\mu)$, it follows readily that this describes a
$\ggl(W)_0$-action on $\od^U(W)$.
\end{anal}

\begin{ex}\label{ex120524e'''}
Continue with the notation of Example~\ref{ex120524e}.

In this case, the only DG $U$-module structure on $W$
is the trivial one $(\partial,\mu)=(0,0)$, so we have
$\alpha\cdot(\partial,\mu)=(\partial,\mu)$ for all
$\alpha\in \ggl(W)_0$.

The action of $\ggl(W')_0$ on $\od^U(W')$ is a bit more interesting. 
Let $x_0,x_1\in F$ such that $x_0x_1=0$, as in Example~\ref{ex120524e}.
Identify $\ggl_F(W')_0$ with $F^\times  \times F^\times$, as in
Example~\ref{ex120524z}, and let $\alpha\in\ggl_F(W')_0$ be given by
the ordered pair $(y_1,y_0)\in F^\times  \times F^\times$.
The differential $\wti\partial$ is defined so that the following diagram commutes.
$$\xymatrix{
\partial:&
0\ar[r]
&F\eta_1\ar[r]^-{x_1}\ar[d]_-{y_1}
&F\eta_0\ar[r]\ar[d]_-{y_0}
&0
\\
\wti\partial:&
0\ar[r]
&F\wti\eta_1\ar[r]^-{\wti x_1}
&F\wti\eta_0\ar[r]
&0}$$
so we have $\wti\partial_1(\wti\eta_1)=y_0x_1y_1^{-1}\wti\eta_0$, i.e., $\wti x_1=y_0x_1y_1^{-1}$.

Since multiplication by 1 is already determined,
and we have $e\cdot_{\alpha}\wti\eta_1=0$ because of degree considerations, 
we only need to understand
$e\cdot_{\alpha}\wti\eta_0$. From Discussion~\ref{action}, this is given by
$$e\cdot_{\alpha}\wti\eta_0=
\alpha_1(e\cdot\alpha_0^{-1}(\wti\eta_0))
=\alpha_1(e\cdot y_0^{-1}\eta_0)
=y_0^{-1}\alpha_1(e\cdot \eta_0)
=y_0^{-1}\alpha_1(x_0 \eta_1)
=y_0^{-1}y_1x_0 \wti\eta_1.
$$
\end{ex}

\begin{exer}\label{exer120526a}
Continue with the notation of Example~\ref{ex120524e}.
Using the solutions to Exercises~\ref{ex120524e'} and~\ref{ex120524e''}
describe the actions of $\ggl(W'')_0$ and $\ggl(W''')_0$ on $\od^U(W'')$ and $\od^U(W''')$, respectively, as in the previous example.
\end{exer}

Next, we describe some properties of the action from Discussion~\ref{action}
that indicate a deeper connection between the algebra and geometry.

\begin{anal}\label{action'}
We work in the setting of Notation~\ref{notn110516a}.

Let $\alpha\in\ggl(W)_0$. For every $(\partial,\mu)\in\od^U(W)$, let
$\alpha\cdot(\partial,\mu):=(\widetilde{\partial},\widetilde{\mu})$
be as in Discussion~\ref{action}.
It is straightforward to show that  a map $\alpha$ gives a
DG $U$-module isomorphism $(W,\partial,\mu)\xra\cong(W,\wti\partial,\wti\mu)$.
Conversely, given another element $(\partial',\mu')\in\od^U(W)$,
if there is a DG $U$-module isomorphism $\beta\colon(W,\partial,\mu)\xra\cong(W,\partial',\mu')$, then
$\beta\in\ggl(W)_0$ and $(\partial',\mu')=\beta\cdot(\partial,\mu)$.
In other words, the orbits in $\od^U(W)$ under the action of $\ggl(W)_0$
are the isomorphism classes
of DG $U$-module structures on $W$.
Given an element $M=(\partial,\mu)\in\od^U(W)$,
the orbit $\ggl(W)_0\cdot M$
is locally closed in $\Mod^U(W)$;
see~\cite[Chapter II, \S 5.3]{demazure:iagag}.

Note that the maps defining the action of $\ggl(W)_0$ on $\od^U(W)$ are regular, that is, 
determined by polynomial functions. This is because the inversion map $\alpha\mapsto\alpha^{-1}$ on $\ggl(W)_0$
is regular, as is the multiplication of matrices corresponding to the compositions defining $\wti\partial$ and
$\wti\mu$.
\end{anal}

Next, we consider even more geometry by identifying tangent spaces to two of our objects of study. 

\begin{notn}\label{notn110519a}
We work in the setting of Notation~\ref{notn110516a}.
Let $F[\epsilon]:=F\epsilon \bigoplus  F$ be the algebra of dual numbers, where  $\epsilon^2=0$ and $|\epsilon|=0$.
For our convenience, we write  elements of $F[\epsilon]$ as column vectors: $a\epsilon+b=\left[\begin{smallmatrix}a \\ b\end{smallmatrix}\right]$.
We identify $U[\epsilon]:=\Otimes[F]{F[\epsilon]}{U}$ with
$U\epsilon \bigoplus  U\cong U \bigoplus  U$, 
and $W[\epsilon]:=\Otimes[F]{F[\epsilon]}{W}$ with $W\epsilon \bigoplus W\cong 
W \bigoplus W$.
Using this protocol, we have $\partial^{U[\epsilon]}_i=\left[\begin{smallmatrix}\partial^U_i & 0\\ 0 & \partial^U_i\end{smallmatrix}\right]$.

Let $\od^{U[\epsilon]}(W[\epsilon])$ denote the set of all ordered pairs $(\partial,\mu)$
making $W[\epsilon]$ into a DG $U[\epsilon]$-module.
Let $\en_{F[\epsilon]}(W[\epsilon])_0$ denote the set of  $F[\epsilon]$-linear endomorphisms of $W[\epsilon]$ that are homogeneous of degree 0.
Let $\ggl(W[\epsilon])_0$ denote the set of  $F[\epsilon]$-linear automorphisms of $W[\epsilon]$ that are homogeneous of degree 0, that is, the invertible elements of $\en_{F[\epsilon]}(W[\epsilon])_0$.

Given an element $M=(\partial,\mu)\in\od^U(W)$, the tangent space
$\tamod_M$ is the set of all ordered pairs $(\ol\partial,\ol\mu)\in\Mod^{U[\epsilon]}(W[\epsilon])$ 
that give rise to $M$ modulo
$\epsilon$.
The tangent space
$\ta^{\ggl(W)_0}_{\id_{W}}$ 
is the set of all elements of $\ggl(W[\epsilon])_0$
that give rise to $\id_W$ modulo
$\epsilon$.
\end{notn}

\begin{disc}
Alternate descriptions of the tangent spaces from Notation~\ref{notn110519a} are contained
in~\cite[Lemmas 4.8 and 4.10]{nasseh:lrfsdc}.
Because of smoothness conditions, the map 
$\ggl(W)_0\xra{\cdot M}\od^U(W)$
induces a linear transformation
$\ta^{\ggl(W)_0}_{\id_{W}}\to\tamod_M$ whose image is
$\taglm_M$; see~\cite[4.11. Proof of Theorem B]{nasseh:lrfsdc}.
\end{disc}

The next two results show some profound connections between the algebra and the geometry of the
objects under consideration.
The ideas behind these results are due to Voigt~\cite{voigt:idteag} and
Gabriel~\cite[1.2 Corollary]{gabriel:frto}.

\begin{thm}[\protect{\cite[4.11. Proof of Theorem B]{nasseh:lrfsdc}}]
\label{tangent space and Ext}
We work in the setting of Notation~\ref{notn110516a}.
Given an element $M=(\partial,\mu)\in\od^U(W)$, there is an
isomorphism of abelian groups
$$
\tamod_M/\taglm_M\cong \yext[U]{1}MM.
$$
\end{thm}

\begin{proof}[Sketch of proof]
Using Notation~\ref{notn110519a}, let $N=(\ol\partial,\ol\mu)$
be an element of $\tamod_M$. 
Since $N$ is a DG $U[\epsilon]$-module, restriction of scalars along the natural inclusion $U\to U[\epsilon]$
makes $N$ a DG $U$-module.

Define $\rho\colon M\to N$ and $\pi\colon N\to M$ by the formulas
$\rho(w):=\left[\begin{smallmatrix}w \\ 0\end{smallmatrix}\right]$ and
$\pi\left(\left[\begin{smallmatrix}w' \\ w\end{smallmatrix}\right]\right):=w$.
With~\cite[Lemmas 4.8 and 4.10]{nasseh:lrfsdc},
one shows that $\rho$ and $\pi$ are chain maps
and that $\rho$ and $\pi$ are $U$-linear.
In other words, we have an exact sequence
$$
0\to M\overset{\rho}\to N\overset{\pi}\to M\to 0
$$
of DG $U$-module morphisms.
So, we obtain a map $\tau\colon\tamod_M\to \yext[U]{1}MM$ where $\tau(N)$ is the equivalence class of the displayed
sequence in $\yext[U]{1}MM$.
One  shows that $\tau$ is a surjective  abelian group homomorphism with $\Ker(\tau)=\taglm_M$, and the result follows from the First Isomorphism 
Theorem.\footnote{It is  amazing and ridiculous that this proof boils down to the First Isomorphism 
Theorem.}

To show that $\tau$ is surjective,
fix an arbitrary element $\zeta\in\yext[U]1MM$, represented by the sequence $0\to M\xra{f} Z\xra{g} M\to 0$. 
In particular, this is an exact sequence of $F$-complexes, so it is degree-wise split. 
This implies that 
we have a commutative diagram of graded vector spaces:
$$
\xymatrix{
0\ar[r]&M\ar[r]^-f\ar[d]_-{=}& Z \ar[d]^-{\vartheta}\ar[r]^-g& M\ar[r]\ar[d]^-{=}& 0\\
0\ar[r]& M\ar[r]^-{\rho}& W[\epsilon]\ar[r]^-{\pi}& M\ar[r]& 0
}
$$
where $\rho(w)=\left[\begin{smallmatrix} w\\ 0\end{smallmatrix}\right]$, 
$\pi\left(\left[\begin{smallmatrix} w'\\ w\end{smallmatrix}\right]\right)=w$, and $\vartheta$ is an isomorphism of graded $F$-vector spaces.
The map $\vartheta$ allows us to 
endow $W[\epsilon]$ with a DG $U[\epsilon]$-module structure 
$(\ol\partial,\ol\mu)$
that gives rise to $M$ modulo $\epsilon$. So we have $N:=(\ol\partial,\ol\mu)\in\tamod_M$.
Furthermore, we have $\tau(N)=\zeta$,
so $\tau$ is surjective.

See~\cite[4.11. Proof of Theorem B]{nasseh:lrfsdc}
for more details.
\end{proof}

\begin{cor}[\protect{\cite[Corollary 4.12]{nasseh:lrfsdc}}]
\label{prop111103a}
We work in the setting of Notation~\ref{notn110516a}.
Let $C$ be a semidualizing
DG $U$-module, and let $s\geq\sup(C)$.
Set $M=\tau(C)_{(\leq s)}$ and $W=\und M$.
Then the orbit $\ggl(W)_0\cdot M$ is open in $\od^U(W)$.
\end{cor}

\begin{proof}
\newcommand{\mco}{\mathcal{O}}
Fact~\ref{prop111110a} implies that
$\yext[U]1MM=0$, so 
$\tamod_M=\taglm_M$ by 
Theorem~\ref{tangent space and Ext}. 
Since the orbit 
$\gl(W)_0\cdot M$ is smooth and locally closed,
this implies that $\gl(W)_0\cdot M$ is open in $\Mod^U(W)$.
See~\cite[Corollary 4.12]{nasseh:lrfsdc}
for more details.
\end{proof}

We are now in a position to state and prove our version
of Happel's result~\cite[proof of first proposition in section 3]{happel:sm}
that was used in the proof of Theorem~\ref{intthm120522b}.

\begin{lem}[\protect{\cite[Lemma 5.1]{nasseh:lrfsdc}}]
\label{prop111115a}
We work in the setting of Notation~\ref{notn110516a}.
The set $\s_W(U)$ of quasiisomorphism classes of
semi-free semidualizing DG $U$-modules $C$ such that
$s\geq\sup(C)$, $C_i=0$ for all $i<0$, and $\und{(\tau(C)_{(\leq s)})}\cong W$ is finite.
\end{lem}

\begin{proof}
Fix a representative $C$ for each quasiisomorphism class in $\s_W(U)$,
and write $[C]\in\s_W(U)$ and $M_C=\tau(C)_{(\leq s)}$. 

Let $[C],[C']\in\s_W(U)$.
If $\ggl(W)_0\cdot M_C=\ggl(W)_0\cdot M_{C'}$, then $[C]=[C']$:
indeed, Discussion~\ref{action} explains the second step in the next display
$$C\simeq M_C\cong M_{C'}\simeq C'$$
and the remaining steps follow from the assumptions 
$s\geq\sup(C)$ and $s\geq\sup(C')$, by Exercise~\ref{disc110302a}.

Now, each orbit $\ggl(W)_0\cdot M_C$ is open in $\od^U(W)$ by 
Corollary~\ref{prop111103a}.
Since $\od^U(W)$ is a subset of an affine space over $F$,
it is quasi-compact, so it can only have finitely many open orbits.
By the previous paragraph, this implies that there are only finitely many
distinct elements $[C]\in\s_W(U)$.
\end{proof}

We conclude this section with the third and final part of the proof of Theorem~\ref{intthm120521a}.

\begin{para}[\textbf{Final part of the proof of Theorem~\ref{intthm120521a}}]
\label{proof120523c}
We need to prove that $\s(U)$ is finite where $U=\Otimes[Q]{k}{A}$.
Set $s=\dim(R)-\depth(R)+n$.
\newcommand{\ua}{\underline{a}}
One uses various accounting principles to 
prove that every semidualizing DG $U$-module
is equivalent to a semidualizing DG $U$-module $C'$ such that
$\HH_i(C')=0$ for all $i<0$ and for all $i>s$.
Let $L\xra\simeq C'$ be a minimal semi-free  resolution of $C'$ over $U$.
The conditions $\sup(L)=\sup(C')\leq s$ imply that $L$ (and hence $C'$)
is quasiisomorphic to the truncation $\wti L:=\tau(L)_{\leq s}$.
We set $W:=\und{\wti{L}}$
and work in the setting of Notation~\ref{notn110516a}.

One then uses further accounting principles
to prove  that there is an integer $\lambda\geq 0$, depending only on $R$
and $U$, such that 
$\sum_{i=0}^sr_i\leq \lambda$.
Compare this with Lemma~\ref{lem120522a}.
(Recall that $r_i$ and other quantities are fixed in Notation~\ref{notn110516a}.)
Then, because there are only finitely many $(r_0,\ldots,r_s)\in \mathbb{N}^{s+1}$ with 
$\sum_{i=0}^sr_i\leq \lambda$,
there are only finitely many $W$ that occur from this construction,
say $W^{(1)},\ldots,W^{(b)}$.
Lemma~\ref{prop111115a}
implies that $\s(U)=\s_{W^{(1)}}(U)\cup\cdots\cup \s_{W^{(b)}}(U)\cup\{ [U]\}$
is finite.
\qed
\end{para}

\appendix
\section{Applications of Semidualizing Modules}
\label{sec120522c}

This section contains three applications of semidualizing modules, to 
indicate why Theorem~\ref{intthm120521a}
might be interesting.

\begin{assumption}
Throughout this section, $(R,\m,k)$ is local. 
\end{assumption}

\subsection*{Application I. Asymptotic Behavior of Bass Numbers}
Our first application shows that the existence of non-trivial semidualizing modules
forces the sequence of Bass numbers of a local ring to be unbounded. This partially answers a question of Huneke.

\begin{defn}\label{defn120521b}
The
\emph{$i$th Bass number} of $R$ is
$\mu^i_R:=\rank_k(\Ext ikR)$.
The \emph{Bass series} of $R$ is the formal power series
$I_R(t)=\sum_{i=1}^\infty\mu_R^it^i$.
\end{defn}

\begin{disc}\label{disc120521c}
The Bass numbers of $R$ contain 
important structural information about the
minimal injective resolution $J$ of $R$.
They also keep track of the depth and injective dimension of $R$:
\begin{align*}
\depth(R)&=\min\{i\geq 0\mid\mu^i_R\neq 0\}\\
\id_R(R)&=\sup\{i\geq 0\mid\mu^i_R\neq 0\}.
\end{align*}
In particular, $R$ is Gorenstein if and only if the
sequence $\{\mu^i_R\}$ is eventually 0.
If $R$ has a dualizing module $D$, then
the Bass numbers of $R$ are related to the Betti numbers of $D$ by the formula
$$\mu^{i+\depth(R)}_R=\beta_i^R(D):=\rank_k(\Ext iDk).$$
\end{disc}

Viewed in the context of the characterization of Gorenstein rings in 
Remark~\ref{disc120521c}, the next question is natural, even if it is a bit bold.\footnote{As best we know,
Huneke has not posed this question in print.}

\begin{question}[Huneke]
If the
sequence $\{\mu^i_R\}$ is bounded, must $R$ be Gorenstein?
Equivalently, if $R$ is not Gorenstein, must the
sequence $\{\mu^i_R\}$ be unbounded?
\end{question}

The connection between semidualizing modules and Huneke's question 
is found in the following result. It shows that
Huneke's question reduces to the case where $R$ has only trivial
semidualizing modules. 
It is worth noting that the same conclusion holds
when $R$ is not assumed to be Cohen-Macaulay and $C$ is a semidualizing \emph{DG} module that is neither free nor dualizing.
However, since we have not talked about dualizing DG modules, we only state 
the module case.

\begin{thm}[\protect{\cite[Theorem B]{sather:bnsc}}]
\label{thm130330a}
If $R$ is Cohen-Macaulay and has a semidualizing module $C$ that is neither free nor dualizing,
then the
sequence $\{\mu^i_R\}$ is unbounded.
\end{thm}

\begin{proof}[Sketch of proof]
Pass to the completion of $R$ to assume that $R$ is complete.
This does not change the sequence $\{\mu^i_R\}$ nor the assumptions on $C$.
As $R$ is complete, it has a dualizing module $D$, and one has
$\mu^i_R=\rank_k(\Ext{i-d}Dk)$ for all $i$, where $d=\depth(R)$.
Thus, it suffices to show that the sequence $\{\rank_k(\Ext{i}Dk)\}$ is unbounded.
With $C'=\Hom CD$, one has $\rank_k(\Ext{i}Ck)\geq 1$ and $\rank_k(\Ext{i}{C'}k)\geq 1$ for all $i\geq 0$.
Hence, the computation
\begin{align*}
\rank_k(\Ext{i}Dk)
&=\sum_{p=0}^i\rank_k(\Ext{i}Ck)\rank_k(\Ext{i}{C'}k)
\geq\sum_{p=0}^ip=p+1
\end{align*}
gives the desired unboundedness.
\end{proof}

\subsection*{Application II. Structure of Quasi-deformations}
Our second application shows how semidualizing modules can be used to 
improve given structures.
Specifically, one can use a particular semidualizing module to improve the
closed fibre of a given quasi-deformation.

\begin{defn}[\protect{\cite[(1.1) and (1.2)]{avramov:cid}}]
\label{defn120522a}
A \emph{quasi-deformation} of $R$ is a diagram 
of local ring homomorphisms
$R\xra\vf R'\xla\tau Q$
such that $\vf$ is flat and $\tau$ is surjective with kernel generated
by a $Q$-regular sequence.
A finitely generated $R$-module $M$ has 
\emph{finite CI-dimension} if there is a quasideformation
$R\to R'\from Q$ such that $\pd_Q(\Otimes{R'}{M})<\infty$.
\end{defn}

\begin{disc}
\label{disc120522a}
A straightforward localization and completion argument shows that
if $M$ is an $R$-module of finite CI-dimension, then there is a 
quasideformation
$R\to R'\from Q$ such that $\pd_Q(\Otimes{R'}{M})$ is finite, $Q$ is complete,
and $R'/\m R'$ is artinian, hence Cohen-Macaulay.
\end{disc}

The next result is a souped-up version of the previous remark.
In contrast to the  application of semidualizing modules in Theorem~\ref{thm130330a},
this one does not refer to any semidualizing modules in the statement.
Instead, in its proof, one uses a semidualizing module to improve the 
quasideformation given by definition to one satisfying the desired conclusions.

\begin{thm}[\protect{\cite[Theorem F]{sather:cidfc}}]
\label{thm120522a}
If $M$ is an $R$-module of finite CI-dimension, then there is a 
quasideformation
$R\to R'\from Q$ such that $\pd_Q(\Otimes{R'}{M})<\infty$
and such that $R'/\m R'$ is artinian and Gorenstein.
\end{thm}

\begin{proof}[Sketch of proof]
By Remark~\ref{disc120522a}, there is a 
quasideformation
$R\xra{\vf} R'\from Q$ such that $\pd_Q(\Otimes{R'}{M})$ is finite, $Q$ is complete,
and $R'/\m R'$ is artinian. We work to improve this quasi-deformation.

A relative version of Cohen's Structure Theorem due to Avramov, Foxby, and Herzog~\cite[(1.1) Theorem]{avramov:solh}
provides ``Cohen factorization'' of $\vf$, that is a commutative diagram 
of local ring homomorphisms
$$
\xymatrix@R=7mm{&R''\ar[rd]^{\vf'}\\
R\ar[ru]^{\dot\vf}\ar[rr]^-\vf && R'}
$$ 
such that $\dot\vf$ is flat, $R''/\m R''$ is regular, and $\vf'$ is surjective.
Since $\vf$ is flat and $R'/\m R'$ is Cohen-Macaulay, 
it follows that $R'$ is perfect over $R''$.
From this, we conclude that $\Ext[R'']i{R'}{R''}=0$ for all $i\neq c$ where $c=\depth(R'')-\depth(R')$,
and that $D^{\vf}:=\Ext[R'']c{R'}{R''}$ is a semidualizing $R'$-module.
(This is the ``relative dualizing module'' of Avramov and Foxby~\cite{avramov:rhafgd}.)
This implies that $\Ext[R']2{D^{\vf}}{D^{\vf}}=0$, so there is a semidualizing $Q$-module $B$ such that
$D^\vf\cong\Otimes[Q]{R'}B$.
(This essentially follows from a lifting theorem of Auslander, Ding, and Solberg~\cite[Proposition 1.7]{auslander:lawlom},
since $Q$ is complete.)

The desired quasi-deformation is in the bottom row of the following diagram
$$\xymatrix{
&R'\ar[d]&Q\ar[l]\ar[d] \\
R\ar[r]\ar[ru]^\vf&R'\ltimes D^\vf&Q\ltimes B\ar[l]}$$
where $Q\ltimes B$ and $R'\ltimes D^\vf$ are  ``trivial extensions'', i.e., Nagata's ``idealizations''.
\end{proof}

\subsection*{Application III. Bass Series of Local Ring Homomorphisms}
Our third application of semidualizing modules is a version of Theorem~\ref{thm120521b} from~\cite{avramov:rhafgd}
where $\vf$ is only assumed to have finite G-dimension, defined in the next few items.  

\begin{defn}
\label{defn130401a}
A finitely generated $R$-module $G$ is \emph{totally reflexive} if one has $G\cong\Hom{\Hom GR}R$
and $\Ext iGR=0=\Ext i{\Hom GR}R$ for all $i\geq 1$. A finitely generated $R$-module $M$ has \emph{finite
G-dimension} if there is an exact sequence
$0\to G_n\to\cdots\to G_0\to M\to 0$
such that each $G_i$ is totally reflexive.
\end{defn}

\begin{disc}
If $M$ is an $R$-module of finite G-dimension, then $\Ext iMR=0$ for $i\gg 0$ and
$\Ext {\depth(R)-\depth_R(M)}MR\neq 0$.
\end{disc}

\begin{defn}
\label{defn130401b}
Let $R\xra\vf S$ be a local ring homomorphism, and let $S\xra\psi\comp S$ be the natural map where $\comp S$ is the completion of $S$.
Fix a Cohen factorization $R\xra{\dot\vf}R'\xra{\vf'}\comp S$ of the ``semi-completion of $\vf$'', i.e., the composition $R\xra{\psi\vf}\comp S$ 
(see the proof of Theorem~\ref{thm120522a}).
We say that $\vf$ has \emph{finite G-dimension} if $\comp S$ has finite G-dimension over $R'$.
Moreover, the map $\vf$ is \emph{quasi-Gorenstein} if it has finite G-dimension and $\Ext i{R'}{\comp S}=0$ for all $i\neq\depth(R')-\depth(\comp S)$.
\end{defn}

The next result is the aforementioned improvement of Theorem~\ref{thm120521b}.
As with Theorem~\ref{thm120522a}, note that the statement does not involve semidualizing modules.

\begin{thm}[\protect{\cite[(7.1) Theorem]{avramov:rhafgd}}]
\label{thm130401}
Let $(R,\m)\to (S,\n)$ be a  local ring homomorphism of finite G-dimension.
Then there is a formal Laurent series $I_{\vf}(t)$ with non-negative integer
coefficients such that
$I_S(t)=I_R(t)I_{\vf}(t)$.
In particular, if $S$ is Gorenstein, then so is $R$.
\end{thm}

\begin{proof}[Sketch of proof when $\vf$ is quasi-Gorenstein]
Fix a Cohen factorization $R\xra{\dot\vf}R'\xra{\vf'}\comp S$ of the semi-completion of $\vf$, and set
$d=\depth(R')-\depth(\comp S)$.
Since $\vf$ is quasi-Gorenstein, the $\comp S$-module $D'=\Ext {d}{R'}{\comp S}$ is semidualizing.
(Again, this is Avramov and Foxby's relative dualizing module~\cite{avramov:rhafgd}.)
If $l$ denotes the residue field of $S$, then the series $I_{\vf}:=\sum_i\rank_l(\Ext[\comp S]{i+d'}{D'}{l})t^i$
satisfies the desired conditions where $d'=\depth(R)-\depth(S)$.
\end{proof}

\section{Sketches of Solutions to Exercises}
\label{sec130212a}

\begin{para}[Sketch of Solution to Exercise~\ref{exer130218a}]
\label{para130218a}
As $S$ is $R$-flat, there is an isomorphism $$\Ext[S]i{\Otimes{S}{C}}{\Otimes{S}{C}}\cong\Otimes{S}{\Ext iCC}$$
for each $i$. It follows that (1) if $\Ext iCC=0$, then $\Ext[S]i{\Otimes{S}{C}}{\Otimes{S}{C}}=0$, and (2)
if $\vf$ is faithfully flat and $\Ext[S]i{\Otimes{S}{C}}{\Otimes{S}{C}}=0$, then $\Ext iCC=0$.
Similarly, there is a commutative diagram
$$\xymatrix{
S\ar[r]^-{\Otimes{S}{\chi^R_C}}\ar[d]_{\chi^S_{\Otimes SC}}
&\Otimes{S}{\Hom CC} \ar[ld]^-{\cong}\\
\Hom[S]{\Otimes{S}{C}}{\Otimes{S}{C}}
}$$
so (1) if $\chi^R_C$ is an isomorphism, then so is $\chi^S_{\Otimes SC}$, and (2)
if $\vf$ is faithfully flat and $\chi^S_{\Otimes SC}$  is an isomorphism, then so is $\chi^R_C$.
\qed
\end{para}

\begin{para}[Sketch of Solution to Exercise~\ref{exer120522b}]
\label{para130218b}
\

\eqref{exer120522b1}
It is routine to show that $\partial^{\Hom XY}_n$ is $R$-linear and maps  $\Hom XY_n$ to $\Hom XY_{n-1}$.
To show that $\Hom XY$ is an $R$-complex, we compute:
\begin{align*}
\partial^{\Hom XY}_{n-1}(\partial^{\Hom XY}_n(\{f_p\}))\hspace{-4.5cm} \\
&=\partial^{\Hom XY}_{n-1}(\{\partial^Y_{p+n}f_p-(-1)^nf_{p-1}\partial^X_p\})\\
&=\{\partial^Y_{p+n}[\partial^Y_{p+n}f_p-(-1)^nf_{p-1}\partial^X_p]-(-1)^{n-1}[\partial^Y_{p+n-1}f_{p-1}-(-1)^nf_{p-2}\partial^X_{p-1}]\partial^X_p\} \\
&=\{\underbrace{\partial^Y_{p+n}\partial^Y_{p+n}}_{=0}f_p\underbrace{-(-1)^n\partial^Y_{p+n}f_{p-1}\partial^X_p-(-1)^{n-1}\partial^Y_{p+n-1}f_{p-1}\partial^X_p}_{=0}
+f_{p-2}\underbrace{\partial^X_{p-1}\partial^X_p}_{=0}\}\\
&=0.
\end{align*}

\eqref{exer120522b2}
For $f=\{f_p\}\in\Hom XY_0$, we have
$$\partial^{\Hom XY}_0(\{f_p\})
=\{\partial^Y_{p}f_p-f_{p-1}\partial^X_p\}.$$
Hence, $f$ is a chain map if and only if $\partial^Y_{p}f_p-f_{p-1}\partial^X_p=0$ for all $p$,
which is equivalent to the commutativity of the given diagram.

\eqref{exer120522b3}
This follows by the fact  $\partial^{\Hom XY}_{0}\circ\partial^{\Hom XY}_1=0$, from part~\eqref{exer120522b1}.

\eqref{exer120522b4}
Since $\partial^{\Hom XY}_1(\{s_p\})=\{\partial^Y_{p+1}s_p+s_{p-1}\partial^X_p\}$, this follows by definition.
\qed
\end{para}

\begin{para}[Sketch of Solution to Exercise~\ref{exer120522i}]
\label{para130218c}
Let $\tau\colon\Hom RX\to X$ be given by
$\tau_n(\{f_p\})=f_n(1)$.
We consider $R$ as a complex concentrated in degree 0, so for all $p\neq 0$ we have $R_p=0$,
and for all $n$ we have $\partial^R_n=0$.
It follows that, for all $n$ and all $f=\{f_p\}\in\Hom RX_n$ and all $p\neq 0$, we have $f_p=0$.
Thus, for each $n$ the natural maps $\Hom R{X}_n\xra{\cong}\Hom R{X_n}\xra{\cong} X_n$ are $R$-module isomorphisms,
the composition of which is $\tau_n$. To show that $\tau$ is a chain map,
we compute:
\begin{align*}
\tau_{n-1}(\partial^{\Hom RX}_n(\{f_p\})) 
&=\tau_{n-1}(\partial^{\Hom RX}_n(\ldots,0,f_0,0,\ldots))\\
&=\tau_{n-1}(\ldots,0,\partial^X_{n}f_0,-(-1)^nf_{0}\partial^R_1,0,\ldots)\\
&=\tau_{n-1}(\ldots,0,\partial^X_{n}f_0,0,0,\ldots)\\
&=\partial^X_{n}(f_0(1))\\
&=\partial^{X}_n(\tau_n(\{f_p\})). 
\end{align*}
(Note that these steps are optimal for presentation, in some sense, but they do not exactly represent the thought process
we used to find the solution. Instead, we computed and simplified $\tau_{n-1}(\partial^{\Hom RX}_n(\{f_p\}))$ and $\partial^{X}_n(\tau_n(\{f_p\}))$
separately and checked that the resulting expression was the same for both.
Similar comments apply to many solutions.)
\qed
\end{para}

\begin{para}[Sketch of Solution to Exercise~\ref{exer130314b}]
\label{para130314b}
Let $X$ be an $R$-complex, and let $M$ be an $R$-module.

\eqref{exer130314b1}
Write $X_*$ for the complex
$$X_*=\cdots\xra{(\partial^X_{n+1})_*}(X_{n})_*\xra{(\partial^X_{n})_*}(X_{n-1})_*\xra{(\partial^X_{n-1})_*}\cdots.$$
We consider $M$ as a complex concentrated in degree 0, so for all $p\neq 0$ we have $M_p=0$,
and for all $n$ we have $\partial^M_n=0$.
It follows that, for all $n$ and all $f=\{f_p\}\in\Hom MX_n$ and all $p\neq 0$, we have $f_p=0$.
Thus, for each $n$ the natural map $\tau_n\colon\Hom M{X}_n\xra{\cong}\Hom M{X_n}=(X_*)_n$ is an $R$-module isomorphism.
Thus, it remains to show that the  map $\tau\colon\Hom MX\to X_*$ given by $f=\{f_p\}\mapsto f_0$ is a chain map.
We compute:
\begin{align*}
\tau_{n-1}(\partial^{\Hom MX}_n(\{f_p\})) 
&=\tau_{n-1}(\partial^{\Hom MX}_n(\ldots,0,f_0,0,\ldots))\\
&=\tau_{n-1}(\ldots,0,\partial^X_{n}f_0,-(-1)^nf_{0}\partial^M_1,0,\ldots)\\
&=\tau_{n-1}(\ldots,0,\partial^X_{n}f_0,0,0,\ldots)\\
&=\partial^X_{n}f_0\\
&=(\partial^X_{n})_*(f_0)\\
&=\partial^{X_*}_{n}(f_0)\\
&=\partial^{X_*}_n(\tau_n(\{f_p\})). 
\end{align*}

\eqref{exer130314b2}
Write $X^*$ and $X^{\dagger}$ for the complexes
\begin{gather*}
X^*=\cdots\xra{(-1)^{n}(\partial^X_{n})^*}X_{n}^*\xra{(-1)^{n+1}(\partial^X_{n+1})^*}X_{n+1}^*\xra{(-1)^{n+2}(\partial^X_{n+2})^*}\cdots
\\
X^\dagger=\cdots\xra{(\partial^X_{n})^*}X_{n}^*\xra{(\partial^X_{n+1})^*}X_{n+1}^*\xra{(\partial^X_{n+2})^*}\cdots.
\end{gather*}
Note that the displayed pieces for $X^*$ are in degree $-n$, $1-n$, and similarly for $X^{\dagger}$.
We prove that $\Hom XM\cong X^*\cong X^\dagger$.

As in part~\eqref{exer130314b1}, for all $p\neq 0$ we have $M_p=0$,
and for all $n$ we have $\partial^M_n=0$.
It follows that, for all $n$ and all $f=\{f_p\}\in\Hom XM_n$ and all $p\neq -n$, we have $f_p=0$.
Thus, for each $n$ the  map $\tau_{n}\colon\Hom {X}M_{n}\xra{\cong}\Hom {X_{-n}}M=(X^*)_n$ 
given by $\{f_p\}\mapsto f_{-n}$ is an $R$-module isomorphism.
Thus, for the  isomorphism $\Hom MX\cong X^*$, it remains to show that  $\tau\colon\Hom XM\to X^*$  is a chain map.
We compute:
\begin{align*}
\tau_{n-1}(\partial^{\Hom XM}_n(\{f_p\})) 
&=\tau_{n-1}(\partial^{\Hom XM}_n(\ldots,0,f_{-n},0,\ldots))\\
&=\tau_{n-1}(\ldots,0,\partial^M_{0}f_{-n},-(-1)^{n}f_{-n}\partial^X_{1-n},0,\ldots)\\
&=\tau_{n-1}(\ldots,0,(-1)^{n-1}f_{-n}\partial^X_{1-n},0,0,\ldots)\\
&=(-1)^{n-1}f_{-n}\partial^X_{1-n}\\
&=(-1)^{1-n}(\partial^X_{1-n})^*(f_{-n})\\
&=\partial^{X^*}_{n}(f_{-n})\\
&=\partial^{X^*}_n(\tau_n(\{f_p\})). 
\end{align*}
For the isomorphism $X^*\cong X^\dagger$, we first observe that $X^\dagger$ is an $R$-complex.
Next, note the following: given an $R$-complex $Y$,
the following diagram describes an isomorphism of $R$-complexes.
$$\xymatrix@C=18.5mm{
Y_{4n+2}\ar[r]^-{(-1)^{-4n-1}\partial^Y_{4n+2}}\ar[d]^{(-1)^{4n+2}}
&Y_{4n+1}\ar[r]^-{(-1)^{-4n}\partial^Y_{4n+1}}\ar[d]_{(-1)^{4n+1}}
&Y_{4n}\ar[r]^-{(-1)^{-4n+1}\partial^Y_{4n}}\ar[d]_{(-1)^{4n+1}}
&Y_{4n-1}\ar[r]^-{(-1)^{-4n+2}\partial^Y_{4n-1}}\ar[d]_{(-1)^{4n}}
&Y_{4n-2}\ar[d]_{(-1)^{4n-2}}
\\
Y_{4n+2}\ar[r]_-{\partial^Y_{4n+2}}
&Y_{4n+1}\ar[r]_-{\partial^Y_{4n+1}}
&Y_{4n}\ar[r]_-{\partial^Y_{4n}}
&Y_{4n-1}\ar[r]_-{\partial^Y_{4n-1}}
&Y_{4n-2}
}$$
Now, apply this observation to $X^{\dagger}$.
\qed
\end{para}

\begin{para}[Sketch of Solution to Exercise~\ref{exer120522f}]
\label{para130218d}
\eqref{exer120522f1}
Let $f\colon X\to Y$ be a  chain map.
By definition, we have $f_{i}\partial^X_{i+1}=\partial^Y_{i+1}f_{i+1}$.
It follows readily that
$f_i(\im(\partial^X_{i+1}))\subseteq\im(\partial^Y_{i+1})$
and $f_i(\Ker(\partial^X_{i}))\subseteq\Ker(\partial^Y_{i})$.
From this, it is straightforward to show that the map
$\Ker(\partial^X_{i})/\im(\partial^X_{i+1})\to\Ker(\partial^Y_{i})/\im(\partial^Y_{i+1})$
given by $\ol x\mapsto\ol{f_i(x)}$ is a  well-defined
$R$-module homomorphism, as desired.

\eqref{exer120522f2}
Assume now that $f$ is null-homotopic.
By definition, there is an element $s=\{s_p\}\in\Hom XY_1$ such that
$\{f_p\}=\partial^{\Hom XY}_1(\{s_p\})=\{\partial^Y_{p+1}s_p+s_{p-1}\partial^X_p\}$.
It follows that for each $i$ and each $\ol x\in\HH_i(X)$, one has
\begin{align*}
\HH_i(f)(\ol x)
&=\ol{f_i(x)}
=\ol{\underbrace{\partial^Y_{i+1}(s_i(x))}_{\in\im(\partial^Y_{i+1})}+s_{i-1}(\underbrace{\partial^X_i(x)}_{=0})}
=0
\end{align*}
in $\HH_i(Y)$.
\qed
\end{para}

\begin{para}[Sketch of Solution to Exercise~\ref{exer120522h}]
\label{para130311a}
Let $f\colon X\xra\cong Y$ be an isomorphism between the $R$-complexes
$X$ and $Y$.
Then for each $i$ the map $f_i$ induces  isomorphisms $\Ker(\partial^X_i)\xra\cong\Ker(\partial^Y_i)$
and $\im(\partial^X_{i+1})\xra\cong\im(\partial^Y_{i+1})$.
It follows that $f_i$ induces an isomorphism 
$\Ker(\partial^X_i)/\im(\partial^X_{i+1})\xra\cong\Ker(\partial^Y_i)/\im(\partial^Y_{i+1})$,
as desired.
\qed
\end{para}

\begin{para}[Sketch of Solution to Exercise~\ref{exer120522g}]
\label{para130311b}
Let $M$ be an $R$-module with augmented projective resolution $P^+$.
$$P^+=\cdots\xra{\partial^P_2}P_1\xra{\partial^P_1}P_0\xra\tau M\to 0.$$
It is straightforward to check that the following diagram commutes
$$
\xymatrix{
P\ar[d]_t
&\cdots\ar[r]^-{\partial^P_2}&P_1\ar[r]^-{\partial^P_1}\ar[d]&P_0\ar[r]\ar[d]_\tau &0\ar[d]\\
M&&0\ar[r]&M\ar[r]&0.
}$$
The exactness of $P^+$ and the definition of $P$ implies that $\HH_i(P)=0=\HH_i(M)$ for $i\neq 0$.
Thus, to show that $t$ is a quasiisomorphism, it suffices to show that $\HH_0(t)\colon\HH_0(P)\to \HH_0(M)$ is an isomorphism.
Notice that this can be identified with the map $\tau'\colon\coker(\partial^P_1)\to M$ induced by $\tau$.
Since $\tau$ is surjective, it is straightforward to show that $\tau'$ is surjective.
Using the fact that $\Ker(\tau)=\im(\partial^P_1)$, one shows readily that $\tau'$
is injective, as desired.

The case of an injective resolution result is handled similarly.
\qed
\end{para}

\begin{para}[Sketch of Solution to Exercise~\ref{exer130312a}]\label{para130312a}
Fix a chain map $f\colon  X\to Y$ and an $R$-complex $Z$.
Consider a sequence $\{g_p\}\in\Hom YZ_n$.
Note that the sequence $\Hom fZ_n(\{g_p\})=\{g_pf_p\}$ is in $\Hom XZ_n$, so the map $\Hom fZ$ is well-defined
and of degree 0. Also, it is straightforward to show that $\Hom fZ$ is $R$-linear.
To verify that $\Hom fZ$ is a chain map, we  compute:
\begin{align*}
\partial^{\Hom XZ}_{n}(\Hom fZ_n(\{g_p\}))
&=\partial^{\Hom XZ}_{n}(\{g_pf_p\})\\
&=\{\partial^Z_{p+n}g_pf_p-(-1)^ng_{p-1}f_{p-1}\partial^X_p\} \\
&=\{\partial^Z_{p+n}g_pf_p-(-1)^ng_{p-1}\partial^Y_pf_p\} \\
&=\Hom fZ_n(\{\partial^Z_{p+n}g_p-(-1)^ng_{p-1}\partial^Y_p\})\\
&=\Hom fZ_n(\partial^{\Hom YZ}_{n}(\{g_p\})). 
\end{align*}
Note that the third equality follows from the fact that $f$ is a chain map. 

The computation for $\Hom Zf\colon\Hom ZX\to\Hom ZY$ is similar.
\qed
\end{para}

\begin{para}[Sketch of Solution to Exercise~\ref{exer120522c}]
\label{para130311c}
It is straightforward to show that $\mu^{X,r}$ is $R$-linear.
(Note that this uses the fact that $R$ is commutative.)
To show that $\mu^{X,r}$ is a chain map, we compute:
$$\partial^X_{i}(\mu^{X,r}_i(x))=\partial^X_{i}(rx)=r\partial^X_{i}(x)=\mu^{X,r}_{i-1}(\partial^X_{i}(x)).$$
For the induced map 
$\HH_i(\mu^{X,r})$, we have
$\HH_i(\mu^{X,r})(\ol x)=\ol{rx}=r\ol x$,
as desired.
\qed
\end{para}

\begin{para}[Sketch of Solution to Exercise~\ref{exer120522d}]
\label{para130311d}
Exercise~\ref{exer120522c} shows that, for all $r\in R$, the sequence
$\chi^X_0(r)=\mu^{X,r}$ is a chain map. That is, $\chi^X_0(r)$ is a cycle in $\Hom XX_0$.
It follows that the next diagram commutes
$$\xymatrix@C=15mm{
&0\ar[r]\ar[d]
&R\ar[r]\ar[d]_{\chi^X_0}
&0 \ar[d]\\
&\Hom XX_{1}\ar[r]^{\partial^{\Hom XX}_{1}}
&\Hom XX_{0}\ar[r]^{\partial^{\Hom XX}_{0}}
&\Hom XX_{-1} 
}$$
so $\chi^X$ is a chain map.
\qed
\end{para}

\begin{para}[Sketch of Solution to Exercise~\ref{exer120522k}]
\label{para130311e}
Let $X$, $Y$, and $Z$ be $R$-complexes.

\eqref{exer120522k1}
It is routine to show that $\partial^{\Otimes XY}_n$ is $R$-linear and maps from $(\Otimes XY)_n$ to $(\Otimes XY)_{n-1}$.
To show  that $\Otimes XY$ is an $R$-complex, we compute:
\begin{align*}
\partial^{\Otimes XY}_{n-1}(\partial^{\Otimes XY}_n(\ldots,0,x_p\otimes y_{n-p},0,\ldots))\hspace{-2.25in}\\
&=
\partial^{\Otimes XY}_{n-1}(\ldots,0,\partial^X_p(x_p)\otimes y_{n-p},(-1)^px_p\otimes \partial^Y_{n-p}(y_{n-p}),0,\ldots) \\
&=
\partial^{\Otimes XY}_{n-1}(\ldots,0,\partial^X_p(x_p)\otimes y_{n-p},0,0,\ldots) \\
&\quad+\partial^{\Otimes XY}_{n-1}(\ldots,0,0,(-1)^px_p\otimes \partial^Y_{n-p}(y_{n-p}),0,\ldots) \\
&=
(\ldots,0,\underbrace{\partial^X_{p-1}(\partial^X_p(x_p))}_{=0}\otimes y_{n-p},(-1)^{p-1}\partial^X_p(x_p)\otimes \partial^Y_{n-p}(y_{n-p}),0,\ldots) \\
&\quad+(\ldots,0,(-1)^p\partial^X_p(x_p)\otimes \partial^Y_{n-p}(y_{n-p}),(-1)^px_p\otimes \underbrace{\partial^Y_{n-p-1}(\partial^Y_{n-p}(y_{n-p}))}_{=0},0,\ldots).
\end{align*}
The only possibly non-trivial entry in this sum is
\begin{multline*}
(-1)^{p-1}\partial^X_p(x_p)\otimes \partial^Y_{n-p}(y_{n-p})+(-1)^{p}\partial^X_p(x_p)\otimes \partial^Y_{n-p}(y_{n-p})
\\
=(-1)^{p-1}[\partial^X_p(x_p)\otimes \partial^Y_{n-p}(y_{n-p})-\partial^X_p(x_p)\otimes \partial^Y_{n-p}(y_{n-p})]=0
\end{multline*}
so we have $\partial^{\Otimes XY}_{n-1}\partial^{\Otimes XY}_{n}=0$.

\eqref{exer120522k2}
The isomorphism
$X\xra\cong\Otimes RX$ 
is given by $x\mapsto (\ldots,0,1\otimes x,0,\ldots)$.
Check that this is an isomorphism as in~\ref{para130218c}, using the fact that
$R$ is concentrated in degree 0.

\eqref{exer120522k3}
As the hint suggests, the isomorphism $g\colon\Otimes XY\xra\cong\Otimes YX$
requires a sign-change: we define 
$$g_n(\ldots,0,x_p\otimes y_{n-p},0,\ldots):=(\ldots,0,(-1)^{p(n-p)}y_{n-p}\otimes x_p,0,\ldots).$$
It is straightforward to show that $g=\{g_n\}$ is $R$-linear and maps $(\Otimes XY)_n$ to $(\Otimes YX)_n$.
The following computation shows that $g$ is a chain map:
\begin{align*}
g_{n-1}(\partial^{\Otimes XY}_{n}(\ldots,0,x_p\otimes y_{n-p},0,\ldots))\hspace{-2in}\\
&=g_{n-1}(\ldots,0,\partial^X_p(x_p)\otimes y_{n-p},(-1)^px_p\otimes \partial^Y_{n-p}(y_{n-p}),0,\ldots)\\
&=g_{n-1}(\ldots,0,\partial^X_p(x_p)\otimes y_{n-p},0,\ldots)\\
&\quad+g_{n-1}(\ldots,0,(-1)^px_p\otimes \partial^Y_{n-p}(y_{n-p}),0,\ldots)\\
&=(\ldots,0,(-1)^{(p-1)(n-p)}y_{n-p}\otimes \partial^X_p(x_p),0,\ldots)\\
&\quad+(\ldots,0,(-1)^{p+p(n-p-1)}\partial^Y_{n-p}(y_{n-p})\otimes x_p,0,\ldots)\\
&=(\ldots,0,(-1)^{(p+1)(n-p)}y_{n-p}\otimes \partial^X_p(x_p),0,\ldots)\\
&\quad+(\ldots,0,(-1)^{p(n-p)}\partial^Y_{n-p}(y_{n-p})\otimes x_p,0,\ldots)\\
&=(\ldots,0,(-1)^{p(n-p)}\partial^Y_{n-p}(y_{n-p})\otimes x_p,(-1)^{(p+1)(n-p)}y_{n-p}\otimes \partial^X_{p}(x_p),0,\ldots)\\
&=(\ldots,0,(-1)^{p(n-p)}\partial^Y_{n-p}(y_{n-p})\otimes x_p,(-1)^{p(n-p)+(n-p)}y_{n-p}\otimes \partial^X_{p}(x_p),0,\ldots)\\
&=\partial^{\Otimes YX}_{n}(\ldots,0,(-1)^{p(n-p)}y_{n-p}\otimes x_p,0,\ldots)\\
&=\partial^{\Otimes YX}_{n}(g_n(\ldots,0,x_p\otimes y_{n-p},0,\ldots))
\end{align*}
Similarly, one shows that the map
$h\colon\Otimes YX\xra\cong\Otimes XY$
defined as 
$$h_n(\ldots,0,y_p\otimes x_{n-p},0,\ldots):=(\ldots,0,(-1)^{p(n-p)}x_{n-p}\otimes y_p,0,\ldots)$$
is a chain map. Moreover, one has
\begin{align*}
g_n(h_n(\ldots,0,y_p\otimes x_{n-p},0,\ldots))
&=g_n(\ldots,0,(-1)^{p(n-p)}x_{n-p}\otimes y_p,0,\ldots)\\
&=(\ldots,0,(-1)^{p(n-p)}(-1)^{(n-p)p}y_p\otimes x_{n-p},0,\ldots)\\
&=(\ldots,0,((-1)^{p(n-p)})^2y_p\otimes x_{n-p},0,\ldots) \\
&=(\ldots,0,y_p\otimes x_{n-p},0,\ldots)
\end{align*}
so $hg$ is the identity on $\Otimes YX$. 
Similarly, $gh$ is the identity on $\Otimes XY$.
It follows that $h$ is a two-sided inverse for $g$, so $g$ is an isomorphism.

\eqref{exer120522k4}
For the  isomorphism $\Otimes{X}{(\Otimes YZ)}\to\Otimes{(\Otimes XY)}{Z}$,
we change notation. The point is that elements of $\Otimes{X}{(\Otimes YZ)}$
are sequences where each entry is itself a sequence. For instance, a generator in $\Otimes{X}{(\Otimes YZ)}_n$
is of the form
$$\bigl(\ldots,(\ldots,0,0,\ldots),(\ldots,0,x_p\otimes(y_q\otimes z_{n-p-q}),0,\ldots),(\ldots,0,0,\ldots),\ldots\bigr)$$
which we simply write as $x_p\otimes(y_q\otimes z_{n-p-q})$.
One has to be careful here not to combine elements illegally:
the elements $x_a\otimes(y_b\otimes z_{n-a-b})$ and $x_p\otimes(y_q\otimes z_{n-p-q})$
are only in the same summand of $\Otimes{X}{(\Otimes YZ)}_n$ if $a=p$ and $b=q$.

Using this protocol,  
define $f\colon\Otimes{X}{(\Otimes YZ)}\to\Otimes{(\Otimes XY)}{Z}$
as
$$f_{p+q+r}(x_p\otimes(y_q\otimes z_r))
=(x_p\otimes y_q)\otimes z_r.$$
(This has no sign-change since no factors are commuted.)
As in the previous case,  showing that $f$ is an isomorphism reduces to showing that
it is a chain map:
\begin{align*}
\partial^{\Otimes{(\Otimes XY)}{Z}}_{p+q+r}(f_{p+q+r}(x_p\otimes(y_q\otimes z_r)))
\hspace{-2in}\\
&=\partial^{\Otimes{(\Otimes XY)}{Z}}_{p+q+r}((x_p\otimes y_q)\otimes z_r) \\
&=\partial^{\Otimes XY}_{p+q}(x_p\otimes y_q)\otimes z_r+(-1)^{p+q}(x_p\otimes y_q)\otimes \partial^Z_r(z_r)\\
&=(\partial^{X}_{p}(x_p)\otimes y_q)\otimes z_r+(-1)^{p}(x_p\otimes \partial^{Y}_{q}(y_q))\otimes z_r+(-1)^{p+q}(x_p\otimes y_q)\otimes \partial^Z_r(z_r)\\
&=f_{p+q+r-1}(\partial^X_p(x_p)\otimes(y_q\otimes z_r))+(-1)^px_p\otimes(\partial^Y_q(y_q)\otimes z_r)\\
&\qquad\qquad+(-1)^{p+q}x_p\otimes(y_q\otimes\partial^Z_r(z_r))\\
&=f_{p+q+r-1}(\partial^{\Otimes{X}{(\Otimes YZ)}}_{p+q+r}(x_p\otimes(y_q\otimes z_r))).
\end{align*}
The last equality in the sequence follows like the first two.
\qed
\end{para}

\begin{para}[Sketch of Solution to Exercise~\ref{exer120523a}]
\label{para130311f}
Fix a chain map $f\colon  X\to Y$ and an $R$-complex $Z$.
For each element $z_p\otimes x_q\in(\Otimes ZX)_n$, the output $(\Otimes Zf)_n(z_p\otimes x_q)=z_p\otimes f_q(x_q)$ is in $(\Otimes XZ)_n$, so the map 
$\Otimes Zf\colon\Otimes ZX\to\Otimes ZY$ is well-defined
and of degree 0. Also, it is straightforward to show that $\Otimes Zf$ is $R$-linear.
To show that $\Otimes Zf$
is a chain map, we compute: 
\begin{align*}
\partial^{\Otimes ZY}_{p+q}((\Otimes Zf)_{p+q}(z_p\otimes x_q))
&=\partial^{\Otimes ZY}_{p+q}(z_p\otimes f_q(x_q)) \\
&=\partial^Z_p(z_p)\otimes f_q(x_q)+(-1)^pz_p\otimes \partial^Y_q(f_q(x_q)) \\
&=\partial^Z_p(z_p)\otimes f_q(x_q)+(-1)^pz_p\otimes f_{q-1}(\partial^X_q(x_q)) \\
&=(\Otimes Zf)_{p+q-1}(\partial^Z_p(z_p)\otimes x_q+(-1)^pz_p\otimes \partial^X_q(x_q)) \\
&=(\Otimes Zf)_{p+q-1}(\partial^{\Otimes ZY}_{p+q}(z_p\otimes x_q)).
\end{align*}
The proof for $\Otimes fZ$ is similar.
\qed
\end{para}

\begin{para}[Sketch of Solution to Exercise~\ref{exer120522j}]
\label{para130312b}
We briefly describe our linear algebra protocols before sketching this solution. 
For each $n\in\bbn$, the module $R^n$ consists of the column vectors of size $n$ with entries from $R$.
Under this convention, the $R$-module homomorphisms $R^m\xra f R^n$ are in bijection with the $n\times m$ matrices with entries
from $R$. Moreover, the matrix representing $f$ has columns $f(e_1),\ldots,f(e_m)$ where $e_1,\ldots,e_m$ is the standard basis
for $R^m$. 

More generally, given free $R$-modules $V$ and $W$ with ordered bases $v_1,\ldots,v_m$ and $w_1,\ldots,w_n$
respectively, the
$R$-module homomorphisms $V\xra f W$ are in 
bijection with the $n\times m$ matrices with entries
from $R$. Moreover, the $j$th column of the matrix representing $f$ with these bases is 
$\left(\begin{smallmatrix}a_{1,j}\\ a_{2,j}\\\vdots\\a_{n,j}\end{smallmatrix}\right)$
where $f(v_j)=\sum_{i=1}^na_{i,j}w_i$.

Note that these protocols allow for function composition to be represented by matrix multiplication in the same order:
if $f$ and $g$ are represented by matrices $A$ and $B$, respectively, then $gf$ is represented by $BA$, using the same ordered bases.

Now for the sketch of the solution.
By definition, we have
\begin{align*}
K^R(x)=\quad&0\to Re\xra x R1\to 0 \\
K^R(y)=\quad&0\to Re\xra y R1\to 0 \\
K^R(z)=\quad&0\to Re\xra z R1\to 0.
\end{align*}
The notation indicates that, in each case, the basis vector in degree 0 is 1 and the basis in degree 1 is e.

We have $K^R(x,y)=\Otimes{K^R(x)}{K^R(y)}$, by definition.
It follows that $K^R(x,y)_n=\bigoplus_{i+j=n}K^R(x)_i\otimes K^R(y)_j$.
Since $K^R(x)_i=0=K^R(y)_i$ for all $i\neq 0,1$,
it follows readily that $K^R(x,y)_n=0$ for all $n\neq 0,1,2$.
In degree 0, we have 
$$K^R(x,y)_0=\Otimes{K^R(x)_0}{K^R(y)_0}=\Otimes{R1}{R1}\cong R$$
with basis vector $1\otimes 1$.
Similarly, we have
$$K^R(x,y)_1=(\Otimes{Re}{R1})\bigoplus(\Otimes{R1}{Re})\cong R^2$$
with ordered basis $e\otimes 1$, $1\otimes e$.
In degree 2, we have
$$K^R(x,y)_2=\Otimes{Re}{Re}\cong R$$
with basis vector $e\otimes e$.
In particular, $K^R(x,y)$ has the following shape:
$$K^R(x,y)=\quad 0\to R\xra{d_2} R^2\xra{d_1} R\to 0.$$
Using the linear algebra protocols described above, we identify an element
$re\otimes 1+s1\otimes e\in K^R(x,y)_1\cong R^2$ with the column vector
$\left(\begin{smallmatrix}r\\s\end{smallmatrix}\right)$. 
It follows that $d_2$ is a $2\times 1$ matrix, and $d_1$ is a $1\times 2$ matrix,
which we identify from the images of the corresponding basis vectors.
First, for $d_2$:
\begin{align*}
d_2(e\otimes e)
&=\partial^{K^R(x)}_1(e)\otimes e+(-1)^{|e|}e\otimes \partial^{K^R(y)}_1(e) 
=x1\otimes e-ye\otimes 1. 
\end{align*}
This corresponds to the column vector $\left(\begin{smallmatrix}-y\\x\end{smallmatrix}\right)$,
and it follows that $d_2$ is represented by the matrix $\left(\begin{smallmatrix}-y\\x\end{smallmatrix}\right)$.
For $d_1$, we have two basis vectors to consider, in order:
\begin{align*}
d_1(e\otimes 1)
&=\partial^{K^R(x)}_1(e)\otimes 1+(-1)^{|e|}e\otimes \partial^{K^R(y)}_1(1) 
=x1\otimes 1
\\
d_1(1\otimes e)
&=\partial^{K^R(x)}_1(1)\otimes e+(-1)^{|1|}1\otimes \partial^{K^R(y)}_1(e) 
=y1\otimes 1.
\end{align*}
It follows that the map $d_1$ is represented by the row matrix $(x\,\,\, y)$, so
in summary:
$$K^R(x,y)=\quad 0\to R\xra{\left(\begin{smallmatrix}-y\\x\end{smallmatrix}\right)} R^2\xra{(x\,\,\, y)} R\to 0.$$
The condition $d_1d_2=0$ follows from the fact that $xy=yx$, that is, from the commutativity of $R$.
Sometimes,  relations of this form are called ``Koszul relations''.

We repeat the process for 
$K^R(x,y,z)=\Otimes{K^R(x,y)}{K^R(z)}$.
For all $n$, we have $K^R(x,y,z)_n=\bigoplus_{i+j=n}K^R(x,y)_i\otimes K^R(z)_j$.
As $K^R(x,y)_i=0=K^R(z)_j$ for all $i\neq 0,1,2$ and all $j\neq 0,1$,
we have $K^R(x,y,z)_n=0$ for  $n\neq 0,1,2,3$.
In degree~0,
$$K^R(x,y,z)_0=\Otimes{K^R(x,y)_0}{K^R(z)_0}=\Otimes{R(1\otimes 1)}{R1}\cong R$$
with basis vector $1\otimes 1\otimes 1$.
The other modules and bases are computed similarly.
We summarize in the following table:
\begin{center}
\begin{tabular}{ccc}
$i$ & $K^R(x,y,z)_i$ & ordered basis
\\
\hline
$3$ & $R^1$ & $e\otimes e\otimes e$
\\
$2$ & $R^3$ & $e\otimes e\otimes 1$, $e\otimes 1\otimes e$, $1\otimes e\otimes e$
\\
$1$ & $R^3$ & $e\otimes 1\otimes 1$, $1\otimes e\otimes 1$, $1\otimes 1\otimes e$
\\
$0$ & $R^1$ & $1\otimes 1\otimes 1$
\end{tabular}
\end{center}
In particular, $K^R(x,y,z)$ has the following shape:
$$K^R(x,y,z)=\quad 0\to R\xra{d_3} R^3\xra{d_2} R^3\xra{d_1} R\to 0.$$
It follows that $d_3$ is a $3\times 1$ matrix, $d_2$ is a $3\times 3$ matrix, and $d_1$ is a $1\times 3$ matrix,
which we identify from the images of the corresponding basis vectors.
First, for $d_3$:
\begin{align*}
d_3(e\otimes e\otimes e)
&=\partial^{K^R(x,y)}_1(e\otimes e)\otimes e+(-1)^{|e\otimes e|}e\otimes e\otimes \partial^{K^R(z)}_1(e) \\
&=x1\otimes e\otimes e-ye\otimes 1\otimes e+ze\otimes e\otimes1. 
\end{align*}
This corresponds to the column vector $\left(\begin{smallmatrix}z\\-y\\x\end{smallmatrix}\right)$,
and it follows that $d_2$ is represented by the matrix $\left(\begin{smallmatrix}z\\-y\\x\end{smallmatrix}\right)$.
Next, for $d_2$:
\begin{align*}
d_2(e\otimes e\otimes1)
&=\partial^{K^R(x,y)}_1(e\otimes e)\otimes 1+(-1)^{|e\otimes e|}e\otimes e\otimes \partial^{K^R(z)}_1(1)\\ 
&=x1\otimes e\otimes 1-ye\otimes 1\otimes 1
\\
d_2(e\otimes 1\otimes e)
&=\partial^{K^R(x,y)}_1(e\otimes 1)\otimes e+(-1)^{|e\otimes 1|}e\otimes 1\otimes \partial^{K^R(z)}_1(e)\\ 
&=x1\otimes 1\otimes e-ze\otimes 1\otimes 1 
\\
d_2(1\otimes e\otimes e)
&=\partial^{K^R(x,y)}_1(1\otimes e)\otimes e+(-1)^{|1\otimes e|}1\otimes e\otimes \partial^{K^R(z)}_1(e)\\ 
&=y1\otimes 1\otimes e-z1\otimes e\otimes 1. 
\end{align*}
It follows that $d_2$ is represented by the following matrix. 
$$d_2=\begin{pmatrix}
-y & -z & 0 \\
x & 0 & -z \\
0 & x & y
\end{pmatrix}$$
For $d_1$, we have three basis vectors to consider, in order:
\begin{align*}
d_1(e\otimes 1\otimes 1)
&=\partial^{K^R(x,y)}_1(e\otimes 1)\otimes 1+(-1)^{|e\otimes 1|}e\otimes 1\otimes \partial^{K^R(z)}_1(1)\\ 
&=x1\otimes 1\otimes 1
\\
d_1(1\otimes e\otimes 1)
&=\partial^{K^R(x,y)}_1(1\otimes e)\otimes 1+(-1)^{|1\otimes e|}1\otimes e\otimes \partial^{K^R(z)}_1(1)\\ 
&=y1\otimes 1\otimes 1
\\
d_1(1\otimes 1\otimes e)
&=\partial^{K^R(x,y)}_1(1\otimes 1)\otimes e+(-1)^{|1\otimes 1|}1\otimes 1\otimes \partial^{K^R(z)}_1(e)\\ 
&=z1\otimes 1\otimes 1
\\
\end{align*}
It follows that the map $d_1$ is represented by the row matrix $(x\,\,\, y\,\,\, z)$.
In summary, we have the following:
$$K^R(x,y,z)=\quad 0\to R\xra{\left(\begin{smallmatrix}z\\-y\\x\end{smallmatrix}\right)} R^3
\xra{\left(\begin{smallmatrix}-y & -z & 0 \\
x & 0 & -z \\
0 & x & y
\end{smallmatrix}\right)} R^3\xra{(x\,\,\, y\,\,\, z)} R\to 0.$$
\qed
\end{para}

\begin{para}[Sketch of Solution to Exercise~\ref{exer120522l}]
\label{para130312c}
This is essentially a manifestation of Pascal's Triangle.
We proceed by induction on $n$, 
where
$\x=x_1,\ldots,x_n$.
The base case $n=1$ follows from directly from Definition~\ref{defn120522h}.
Assume now that the result holds for sequences of length $t\geq 1$.
We prove the it holds for sequences $x_1,\ldots,x_t,x_{t+1}$.
By definition, we have 
$$K^R(x_1,\ldots,x_{t+1})= \Otimes{K^R(x_1,\ldots,x_{t})}{K^R(x_{t+1})}$$
Since $K^R(x_{t+1})_i=0$ for all $i\neq 0,1$, it follows that
\begin{align*}
K^R(x_1,\ldots,x_{t+1})_m\hspace{-.8in}\\
&= \Otimes{K^R(x_1,\ldots,x_{t})}{K^R(x_{t+1})} \\
&= (\Otimes{K^R(x_1,\ldots,x_{t})_{m}}{K^R(x_{t+1})_{0}})
\bigoplus(\Otimes{K^R(x_1,\ldots,x_{t})_{m-1}}{K^R(x_{t+1})_{1}})\\
&\cong (\Otimes{R^{\binom{t}{m}}}{R})\bigoplus(\Otimes{R^{\binom{t}{m-1}}}{R})\\
&\cong R^{\binom{t}{m}+\binom{t}{m-1}}\\
&=R^{\binom{t+1}{m}}
\end{align*}
which is the desired conclusion.
\qed
\end{para}

\begin{para}[Sketch of Solution to Exercise~\ref{exer120522u}]
\label{para130312d}
This follows from the commutativity of tensor product
\begin{align*}
K^R(\x)
&=K^R(x_1)\otimes_R\cdots\otimes_RK^R(x_n)
\cong K^R(x_{\sigma(1)})\otimes_R\cdots\otimes_RK^R(x_{\sigma(n)})
=K^R(\x')
\end{align*}
which is the desired conclusion.
\qed
\end{para}

\begin{para}[Sketch of Solution to Exercise~\ref{exer130314a}]
Set $(-)^*=\Hom -R$.
We use the following linear algebra facts freely.
First, there is an isomorphism $(R^n)^*\cong R^n$.
Second, given an $R$-module homomorphism $R^m\to R^n$ represented by a matrix $A$,
the dual map $(R^n)^*\to(R^m)^*$ is represented by the transpose $A^{\text{T}}$.

\label{para130314a}
Let $x,y,z\in R$.
First, we verify that
$\Hom{K^R(x)}{R}\cong\shift K^R(x)$,
by the above linear algebra facts with Exercise~\ref{exer130314b}.
The complex $K^R(x)$ has the form
$$K^R(x)=\quad 0\to R\xra x R\to 0$$
concentrated in degrees 0 and 1.
Thus, the shifted complex $\shift K^R(x)$ is of the form
$$\shift K^R(x)=\quad 0\to R\xra{-x} R\to 0$$
concentrated in degrees 0 and $-1$.
By Exercise~\ref{exer130314b}, the dual 
$\Hom{K^R(x)}{R}$ 
is concentrated in degrees 0 and $-1$, and has the form
$$\Hom{K^R(x)}{R}=\quad 0\to R\xra{x} R\to 0.$$
The following diagram shows that these complexes are isomorphic.
$$\xymatrix{
\shift K^R(x)\ar[d]_\cong
&0\ar[r] 
&R\ar[r]^-{-x} \ar[d]_{-1}
&R\ar[r] \ar[d]^1
&0
\\
\Hom{K^R(x)}{R}
&0\ar[r] 
&R\ar[r]^-{x} 
&R\ar[r] 
&0
}$$

Next, we verify the isomorphism
$\Hom{K^R(x,y)}{R}\cong\shift^2K^R(x,y)$.
By Exercise~\ref{exer120522j}, the complex $K^R(x,y)$ is of the form
$$K^R(x,y)=\quad 0\to R\xra{\left(\begin{smallmatrix}-y\\x\end{smallmatrix}\right)} R^2\xra{(x\,\,\, y)} R\to 0$$
concentrated in degrees $0,1,2$.
Thus, the shifted complex $\shift K^R(x)$ is of the form
$$\shift K^R(x,y)=\quad 0\to R\xra{\left(\begin{smallmatrix}-y\\x\end{smallmatrix}\right)} R^2\xra{(x\,\,\, y)} R\to 0$$
concentrated in degrees $0,-1,-2$.
By Exercise~\ref{exer130314b}, the dual 
$\Hom{K^R(x,y)}{R}$ 
has the form and is concentrated in degrees 0 and $-1$
$$\Hom{K^R(x,y)}{R}=\quad 0\to R\xra{\left(\begin{smallmatrix}x\\ y\end{smallmatrix}\right)} R^2\xra{(-y\,\,\, x)} R\to 0.$$
The following diagram shows that these complexes are isomorphic.
$$\xymatrix@C=13mm{
\shift K^R(x,y)\ar[d]_\cong
&0\ar[r] 
&R\ar[r]^-{\left(\begin{smallmatrix}-y\\x\end{smallmatrix}\right)} \ar[d]_{1}
&R^2\ar[r]^-{(x\,\,\, y)} \ar[d]_{\left(\begin{smallmatrix}0 & 1 \\ -1 & 0\end{smallmatrix}\right)}
&R\ar[r] \ar[d]^{1}
&0
\\
\Hom{K^R(x,y)}{R}
&0\ar[r] 
&R\ar[r]_-{\left(\begin{smallmatrix}x\\ y\end{smallmatrix}\right)} 
&R^2\ar[r]^-{(-y\,\,\, x)} 
&R\ar[r] 
&0
}$$
Note that we found this isomorphism, as follows. Use the identity in degree 1. 
Consider the matrix $\left(\begin{smallmatrix}a&b\\c&d\end{smallmatrix}\right)$
in degree 1, and solve the linear equations needed to make the right-most square commute.
Then check that the identity  in degree 2 makes the left-most square commute.

Lastly, we verify the isomorphism
$\Hom{K^R(x,y,z)}{R}\cong\shift^3K^R(x,y,z)$.
By Exercise~\ref{exer120522j}, the complex $K^R(x,y,x)$ is of the form
$$K^R(x,y,z)=\quad 0\to R\xra{\left(\begin{smallmatrix}z\\-y\\x\end{smallmatrix}\right)} R^3
\xra{\left(\begin{smallmatrix}-y & -z & 0 \\
x & 0 & -z \\
0 & x & y
\end{smallmatrix}\right)} R^3\xra{(x\,\,\, y\,\,\, z)} R\to 0$$
concentrated in degrees $0,1,2,3$.
The shifted complex $\shift^3 K^R(x,y,z)$ is of the form
$$\shift^3 K^R(x,y,z)=\quad 0\to R\xra{\left(\begin{smallmatrix}-z\\y\\-x\end{smallmatrix}\right)} R^3
\xra{\left(\begin{smallmatrix}y & z & 0 \\
-x & 0 & z \\
0 & -x & -y
\end{smallmatrix}\right)} R^3\xra{(-x\,\,\, -y\,\,\, -z)} R\to 0$$
concentrated in degrees  $0,-1,-2,-3$.
By Exercise~\ref{exer130314b}, the dual 
$\Hom{K^R(x)}{R}$ 
has the form and is concentrated in degrees $0,-1,-2,-3$
$$\Hom{K^R(x)}{R}=\quad 0\to R\xra{\left(\begin{smallmatrix}x\\ y\\ z\end{smallmatrix}\right)} R^3
\xra{\left(\begin{smallmatrix}-y & x & 0 \\
-z & 0 & x \\
0 & -z & y
\end{smallmatrix}\right)} R^3\xra{(z\,\,\,-y\,\,\,x)} R\to 0.$$
The following diagram 
$$\xymatrix@C=4mm{
\shift^3 K^R(x,y,z)\ar[d]_\cong
&0\ar[r] 
&R\ar[rrr]^-{\left(\begin{smallmatrix}-z\\y\\-x\end{smallmatrix}\right)} \ar[d]_{1}
&&&R^3\ar[rrr]^-{\left(\begin{smallmatrix}y & z & 0 \\
-x & 0 & z \\
0 & -x & -y
\end{smallmatrix}\right)} \ar[d]_{\left(\begin{smallmatrix}0 & 0 & -1 \\
0 & 1 & 0 \\
-1 & 0 & 0
\end{smallmatrix}\right)}
&&&R^3\ar[rrr]^-{(-x\,\,\, -y\,\,\, -z)} \ar[d]_{\left(\begin{smallmatrix}0 & 0 & -1 \\
0 & 1 & 0 \\
-1 & 0 & 0
\end{smallmatrix}\right)}
&&&R\ar[r] \ar[d]^1
&0
\\
\Hom{K^R(x,y,z)}{R}
&0\ar[r] 
&R\ar[rrr]_-{\left(\begin{smallmatrix}x\\ y\\ z\end{smallmatrix}\right)} 
&&&R^3\ar[rrr]_-{\left(\begin{smallmatrix}-y & x & 0 \\
-z & 0 & x \\
0 & -z & y
\end{smallmatrix}\right)} 
&&&R^3\ar[rrr]_-{(z\,\,\,-y\,\,\,x)} 
&&&R\ar[r] 
&0
}$$
shows that these complexes are isomorphic.
\qed
\end{para}

\begin{para}[Sketch of Solution to Exercise~\ref{exer120522o}]
\label{para130312e}
Let $\x=x_1,\ldots,x_n\in R$.

By definition of $\wti K^R(\x)$, the first differential is  given by $d_1(e_i)=x_i$ for $i=1,\ldots,n$,
represented in matrix form by $R^n\xra{(x_1\,\,\, \cdots \,\,\, x_n)} R$.
In particular, we have
$$\wti K^R(x_1)=\quad0\to R\xra{x_1} R\to 0.$$
For $\wti K^R(x_1,x_2)$ this says that we only have to compute $d_2$:
$$d_2(e_1\wedge e_2)=x_1e_2-x_2e_1$$
so we have
$$\wti K^R(x_1,x_2)=\quad
0\to R\xra{\left(\begin{smallmatrix}-x_2\\x_1\end{smallmatrix}\right)}R^2\xra{(x_1\,\,\, x_2)}R\to 0.$$
In comparison with Exercise~\ref{exer120522j}, this says that
$$\wti K^R(x,y)=\quad
0\to R\xra{\left(\begin{smallmatrix}-y\\x\end{smallmatrix}\right)}R^2\xra{(x\,\,\, y)}R\to 0.$$
For $\wti K^R(x,y,z)$, we specify an ordering on the basis for $\wti K^R(x,y,z)_2$
$$e_1\wedge e_2, e_1\wedge e_3, e_2\wedge e_3$$
and we compute:
\begin{align*}
d_2(e_1\wedge e_2)&=x_1e_2-x_2e_1\\
d_2(e_1\wedge e_3)&=x_1e_3-x_3e_1\\
d_2(e_2\wedge e_3)&=x_2e_3-x_3e_2\\
d_3(e_1\wedge e_2\wedge e_3)&=x_1e_2\wedge e_3-x_2e_1\wedge e_3+x_3e_1\wedge e_2.
\end{align*}
In summary, we have the following:
$$K^R(x_1,x_2,x_3)=\quad 0\to R\xra{\left(\begin{smallmatrix}x_3\\-x_2\\x_1\end{smallmatrix}\right)} R^3
\xra{\left(\begin{smallmatrix}-x_2 & -x_3 & 0 \\
x_1 & 0 & -x_3 \\
0 & x_1 & x_2
\end{smallmatrix}\right)} R^3\xra{(x_1\,\,\, x_2\,\,\, x_3)} R\to 0.$$
Again, this compares directly with Exercise~\ref{exer120522j}.
\qed
\end{para}

\begin{para}[Sketch of Solution to Exercise~\ref{exer120522p}]
\label{para130312f}
In the following multiplication tables, given an element $x$ from the left column and an element $y$ from the top row, the corresponding element
in the table is the product $xy$.

\begin{center}
\begin{tabular}{r|cc}
$\bigwedge R^1$ & 1 & $e$ \\
\hline
1 & 1 & $e$ \\
$e$ & $e$ & 0
\end{tabular}
\qquad
\begin{tabular}{r|cccc}
$\bigwedge R^2$ & 1 & $e_1$ & $e_2$ & $e_1\wedge e_2$ \\
\hline
1 & 1 & $e_1$ & $e_2$ & $e_1\wedge e_2$ \\
$e_1$ & $e_1$ & $0$ & $e_1\wedge e_2$ & 0 \\
$e_2$ & $e_2$  & $-e_1\wedge e_2$ & 0 & 0 \\
$e_1\wedge e_2$ & $e_1\wedge e_2$ & 0 & 0 & $0$
\end{tabular}
\end{center}
\
\begin{center}
\begin{tabular}{r|cccc}
$\bigwedge R^3$ & 1 & $e_1$ & $e_2$ & $e_3$  \\
\hline
1 & 1 & $e_1$ & $e_2$ & $e_3$  \\
$e_1$ & $e_1$ & $0$ & $e_1\wedge e_2$ & $e_1\wedge e_3$  \\
$e_2$ & $e_2$  & $-e_1\wedge e_2$ & 0 & $e_2\wedge e_3$ \\
$e_3$ & $e_3$ & $-e_1\wedge e_3$ & $-e_2\wedge e_3$ & 0  \\
$e_1\wedge e_2$ & $e_1\wedge e_2$ & 0 & 0 & $e_1\wedge e_2\wedge e_3$ \\
$e_1\wedge e_3$ & $e_1\wedge e_3$ & 0 & $-e_1\wedge e_2\wedge e_3$ & 0  \\
$e_2\wedge e_3$ & $e_2\wedge e_3$ & $e_1\wedge e_2\wedge e_3$ & 0 & 0  \\
$e_1\wedge e_2\wedge e_3$ & $e_1\wedge e_2\wedge e_3$ & 0 & 0 & 0  \\
\end{tabular}
\end{center}
\
\begin{center}
\begin{tabular}{r|cccc}
$\bigwedge R^3$ & $e_1\wedge e_2$ & $e_1\wedge e_3$ & $e_2\wedge e_3$ & $e_1\wedge e_2\wedge e_3$ \\
\hline
1  & $e_1\wedge e_2$ & $e_1\wedge e_3$ & $e_2\wedge e_3$ & $e_1\wedge e_2\wedge e_3$ \\
$e_1$ & 0 & 0 & $e_1\wedge e_2\wedge e_3$ & 0 \\
$e_2$ &  0 & $-e_1\wedge e_2\wedge e_3$ & 0 & 0\\
$e_3$ & $e_1\wedge e_2\wedge e_3$ & 0 & $0$ & 0 \\
$e_1\wedge e_2$  & 0 & 0 & 0 & 0\\
$e_1\wedge e_3$  & 0 & 0 & 0 & 0 \\
$e_2\wedge e_3$  & 0 & 0 & 0 & 0 \\
$e_1\wedge e_2\wedge e_3$  & 0 & 0 & 0 & 0 \\
\end{tabular}
\end{center}
\end{para}

\begin{para}[Sketch of Solution to Exercise~\ref{exer120522q}]
\label{para130312g}
We proceed by induction on $t$.

Base case:  $t=2$.
The result is trivial if $\iota$ is the identity, so assume for the rest of this paragraph that $\iota$ is not the identity permutation.
Because of our assumptions on $\iota$, in this case it has the cycle-notation form $\iota=(j_1\,\,\, j_2)$.
If $j_1<j_2$, then by definition we have
$$e_{\iota(j_1)}\wedge e_{\iota(j_2)}
=e_{j_2}\wedge e_{j_1}
=-e_{j_1}\wedge e_{j_2}
=\operatorname{sgn}(\iota)e_{j_1}\wedge e_{j_2}.
$$
The same logic applies when $j_i>j_2$.

Induction step: Assume that $t\geq 3$ and 
that the result holds for elements of the form $e_{i_1}\wedge\cdots\wedge e_{i_s}$
such that $2\leq s<t$.
We proceed by cases.

Case 1: $\iota(j_1)=j_1$.
In this case, $\iota$ describes a permutation of $\{j_2,\ldots,j_t\}$ with the same signum as $\iota$;
thus, our induction hypothesis explains the third equality in the following display
\begin{align*}
e_{\iota(j_1)}\wedge e_{\iota(j_2)}\wedge\cdots\wedge e_{\iota(j_t)}
&=e_{\iota(j_1)}\wedge (e_{\iota(j_2)}\wedge\cdots\wedge e_{\iota(j_t)})\\
&=e_{j_1}\wedge (e_{\iota(j_2)}\wedge\cdots\wedge e_{\iota(j_t)})\\
&=e_{j_1}\wedge (\operatorname{sgn}(\iota)e_{j_2}\wedge\cdots\wedge e_{j_t}) \\
&=\operatorname{sgn}(\iota)e_{j_1}\wedge(e_{j_2}\wedge\cdots\wedge e_{j_t})\\
&=\operatorname{sgn}(\iota)e_{j_1}\wedge e_{j_2}\wedge\cdots\wedge e_{j_t}
\end{align*}
and the other equalities are by definition and the condition $\iota(j_1)=j_1$.

Case 2: $\iota$ has the cycle-notation form $\iota=(j_1\,\,\, j_2)$ and $j_2=\min\{j_2,j_3,\ldots,j_t\}$.
In this case, we are trying to show that 
\begin{equation}
\label{eq130316a}
e_{j_2}\wedge e_{j_1}\wedge e_{j_3}\wedge\cdots\wedge e_{j_t}
=-e_{j_1}\wedge e_{j_2}\wedge e_{j_3}\wedge\cdots\wedge e_{j_t}.
\end{equation}
Let $\alpha$ be the permutation of $\{j_3,\ldots,j_t\}$ such that $\alpha(j_3)<\cdots<\alpha(j_t)$.
Our assumption on $j_2$ implies that $j_2<\alpha(j_3)<\cdots<\alpha(j_t)$.

Case 2a: $j_1<j_2<\alpha(j_3)<\cdots<\alpha(j_t)$.
Our induction hypothesis explains the second equality in the following sequence
\begin{align*}
e_{j_2}\wedge e_{j_1}\wedge e_{j_3}\wedge\cdots\wedge e_{j_t}
&=e_{j_2}\wedge (e_{j_1}\wedge(e_{j_3}\wedge\cdots\wedge e_{j_t})) \\
&=e_{j_2}\wedge (e_{j_1}\wedge(\operatorname{sgn}(\alpha)e_{\alpha(j_3)}\wedge\cdots\wedge e_{\alpha(j_t)})) \\
&=\operatorname{sgn}(\alpha)e_{j_2}\wedge (e_{j_1}\wedge(e_{\alpha(j_3)}\wedge\cdots\wedge e_{\alpha(j_t)})) \\
&=\operatorname{sgn}(\alpha)e_{j_2}\wedge (e_{j_1}\wedge e_{\alpha(j_3)}\wedge\cdots\wedge e_{\alpha(j_t)}) \\
&=-\operatorname{sgn}(\alpha)e_{j_1}\wedge e_{j_2}\wedge e_{\alpha(j_3)}\wedge\cdots\wedge e_{\alpha(j_t)}
\end{align*}
and the remaining equalities are by definition, where the fourth and fifth ones use our Case 2a assumption.
Similar logic explains the next sequence:
\begin{align*}
-e_{j_1}\wedge e_{j_2}\wedge e_{j_3}\wedge\cdots\wedge e_{j_t}
&=-e_{j_1}\wedge (e_{j_2}\wedge (e_{j_3}\wedge\cdots\wedge e_{j_t}))\\
&=-e_{j_1}\wedge (e_{j_2}\wedge (\operatorname{sgn}(\alpha)e_{\alpha(j_3)}\wedge\cdots\wedge e_{\alpha(j_t)}))\\
&=-\operatorname{sgn}(\alpha)e_{j_1}\wedge (e_{j_2}\wedge (e_{\alpha(j_3)}\wedge\cdots\wedge e_{\alpha(j_t)}))\\
&=-\operatorname{sgn}(\alpha)e_{j_1}\wedge e_{j_2}\wedge e_{\alpha(j_3)}\wedge\cdots\wedge e_{\alpha(j_t)}.
\end{align*}
This explains equation~\eqref{eq130316a} in Case 2a.

Case 2b: $j_2<j_1<\alpha(j_3)<\cdots<\alpha(j_t)$.
Our induction hypothesis explains the second equality in the following sequence
\begin{align*}
e_{j_2}\wedge e_{j_1}\wedge e_{j_3}\wedge\cdots\wedge e_{j_t}
&=e_{j_2}\wedge (e_{j_1}\wedge(e_{j_3}\wedge\cdots\wedge e_{j_t})) \\
&=e_{j_2}\wedge (e_{j_1}\wedge(\operatorname{sgn}(\alpha)e_{\alpha(j_3)}\wedge\cdots\wedge e_{\alpha(j_t)})) \\
&=\operatorname{sgn}(\alpha)e_{j_2}\wedge (e_{j_1}\wedge(e_{\alpha(j_3)}\wedge\cdots\wedge e_{\alpha(j_t)})) \\
&=\operatorname{sgn}(\alpha)e_{j_2}\wedge e_{j_1}\wedge e_{\alpha(j_3)}\wedge\cdots\wedge e_{\alpha(j_t)}
\end{align*}
and the remaining equalities are by definition, where the fourth  one uses our Case 2b assumption.
Similar logic explains the next sequence:
\begin{align*}
-e_{j_1}\wedge e_{j_2}\wedge e_{j_3}\wedge\cdots\wedge e_{j_t}
&=-e_{j_1}\wedge (e_{j_2}\wedge (e_{j_3}\wedge\cdots\wedge e_{j_t}))\\
&=-e_{j_1}\wedge (e_{j_2}\wedge (\operatorname{sgn}(\alpha)e_{\alpha(j_3)}\wedge\cdots\wedge e_{\alpha(j_t)}))\\
&=-\operatorname{sgn}(\alpha)e_{j_1}\wedge (e_{j_2}\wedge (e_{\alpha(j_3)}\wedge\cdots\wedge e_{\alpha(j_t)}))\\
&=-\operatorname{sgn}(\alpha)e_{j_1}\wedge (e_{j_2}\wedge e_{\alpha(j_3)}\wedge\cdots\wedge e_{\alpha(j_t)})\\
&=\operatorname{sgn}(\alpha)e_{j_2}\wedge e_{j_1}\wedge e_{\alpha(j_3)}\wedge\cdots\wedge e_{\alpha(j_t)}.
\end{align*}
This explains equation~\eqref{eq130316a} in Case 2b.

Case 2c: $j_2<\alpha(j_3)<\cdots<\alpha(j_p)<j_1<\alpha(j_{p+1})<\cdots<\alpha(j_t)$.
Similar logic as in the previous cases explain the following sequences
\begin{align*}
e_{j_2}\wedge e_{j_1}\wedge e_{j_3}\wedge\cdots\wedge e_{j_t}\hspace{-20mm} \\
&=e_{j_2}\wedge (e_{j_1}\wedge(e_{j_3}\wedge\cdots\wedge e_{j_t})) \\
&=e_{j_2}\wedge (e_{j_1}\wedge(\operatorname{sgn}(\alpha)e_{\alpha(j_3)}\wedge\cdots\wedge e_{\alpha(j_t)})) \\
&=\operatorname{sgn}(\alpha)e_{j_2}\wedge (e_{j_1}\wedge(e_{\alpha(j_3)}\wedge\cdots\wedge e_{\alpha(j_t)})) \\
&=\operatorname{sgn}(\alpha)e_{j_2}\wedge ((-1)^{p-2}e_{\alpha(j_3)}\wedge\cdots\wedge e_{\alpha(j_p)}\wedge e_{j_1}\wedge e_{\alpha(j_{p+1})}\wedge \cdots e_{\alpha(j_t)}) \\  
&=(-1)^{p-2}\operatorname{sgn}(\alpha)e_{j_2}\wedge e_{\alpha(j_3)}\wedge\cdots\wedge e_{\alpha(j_p)}\wedge e_{j_1}\wedge e_{\alpha(j_{p+1})}\wedge \cdots e_{\alpha(j_t)}\\
-e_{j_1}\wedge e_{j_2}\wedge e_{j_3}\wedge\cdots\wedge e_{j_t}\hspace{-21mm} \\
&=-e_{j_1}\wedge (e_{j_2}\wedge (e_{j_3}\wedge\cdots\wedge e_{j_t}))\\
&=-e_{j_1}\wedge (e_{j_2}\wedge (\operatorname{sgn}(\alpha)e_{\alpha(j_3)}\wedge\cdots\wedge e_{\alpha(j_t)}))\\
&=-\operatorname{sgn}(\alpha)e_{j_1}\wedge (e_{j_2}\wedge (e_{\alpha(j_3)}\wedge\cdots\wedge e_{\alpha(j_t)}))\\
&=-\operatorname{sgn}(\alpha)e_{j_1}\wedge (e_{j_2}\wedge e_{\alpha(j_3)}\wedge\cdots\wedge e_{\alpha(j_t)})\\
&=-(-1)^{p-1}\operatorname{sgn}(\alpha)e_{j_2}\wedge e_{\alpha(j_3)}\wedge\cdots\wedge e_{\alpha(j_p)}\wedge e_{j_1}\wedge e_{\alpha(j_{p+1})}\wedge \cdots e_{\alpha(j_t)} \\  
&=(-1)^{p-2}\operatorname{sgn}(\alpha)e_{j_2}\wedge e_{\alpha(j_3)}\wedge\cdots\wedge e_{\alpha(j_p)}\wedge e_{j_1}\wedge e_{\alpha(j_{p+1})}\wedge \cdots e_{\alpha(j_t)}
\end{align*}
This explains equation~\eqref{eq130316a} in Case 2c.

Case 2d: $j_2<\alpha(j_3)<\cdots<\alpha(j_t)<j_1$.
This case is handled as in Case 2c.

Case 3: $\iota$ has the cycle-notation form $\iota=(j_1\,\,\, j_z)$ where $j_z=\min\{j_2,j_3,\ldots,j_t\}$.
In this case, we are trying to prove the following:
$$
e_{j_z}\wedge e_{j_2}\wedge\cdots\wedge e_{j_{z-1}}\wedge e_{j_1}\wedge e_{j_{z+1}}\wedge \cdots\wedge e_{j_t}
=-e_{j_1}\wedge e_{j_2}\wedge \cdots\wedge e_{j_{z-1}}\wedge e_{j_z}\wedge e_{j_{z+1}}\wedge \cdots\wedge e_{j_t}.
$$
This follows from the next sequence, where the first and third equalities are by our induction hypothesis
\begin{align*}
e_{j_z}\wedge e_{j_2}\wedge\cdots\wedge e_{j_{z-1}}\wedge e_{j_1}\wedge e_{j_{z+1}}\wedge \cdots\wedge e_{j_t}\hspace{-30mm}\\
&=(-1)^{z-2}e_{j_z}\wedge e_{j_1}\wedge e_{j_2}\wedge\cdots\wedge e_{j_{z-1}}\wedge e_{j_{z+1}}\wedge \cdots\wedge e_{j_t} \\
&=-(-1)^{z-2}e_{j_1}\wedge e_{j_z}\wedge e_{j_2}\wedge\cdots\wedge e_{j_{z-1}}\wedge e_{j_{z+1}}\wedge \cdots\wedge e_{j_t} \\
&=-e_{j_1}\wedge e_{j_2}\wedge\cdots\wedge e_{j_{z-1}}\wedge e_{j_z}\wedge e_{j_{z+1}}\wedge \cdots\wedge e_{j_t} 
\end{align*}
and the second equality is from Case 2.

Case 4: $\iota$ has the cycle-notation form $\iota=(j_1\,\,\, j_q)$ for some $q\geq 2$.
Set $j_z=\min\{j_2,j_3,\ldots,j_t\}$.

Case 4a: $j_1<j_z$. In this case, we have $j_1=\min\{j_2,\ldots,j_{q-1},j_1,j_{q+1},\ldots,j_t\}$,
so Case 3 implies that
$$
e_{j_q}\wedge e_{j_2}\wedge\cdots\wedge e_{j_{q-1}}\wedge e_{j_1}\wedge e_{j_{q+1}}\wedge \cdots\wedge e_{j_t}
=-e_{j_1}\wedge e_{j_2}\wedge \cdots\wedge e_{j_{q-1}}\wedge e_{j_q}\wedge e_{j_{q+1}}\wedge \cdots\wedge e_{j_t}.
$$
which is the desired equality in this case.

Case 4b: $j_z<j_1$.
In this case, we have
$$\min\{j_2,j_3,\ldots,j_t\}=j_z=\min\{j_2,\ldots,j_{q-1},j_1,j_{q+1},\ldots,j_t\}$$
so Case 3 explains the first and third equalities in the next sequence
\begin{align*}
e_{j_q}\wedge e_{j_2}\wedge\cdots\wedge e_{j_{z-1}}\wedge e_{j_z}\wedge e_{j_{z+1}}\wedge \cdots\wedge e_{j_{q-1}}\wedge e_{j_1}\wedge e_{j_{q+1}}\wedge \cdots\wedge e_{j_t}
\hspace{-95mm} \\
&=-e_{j_z}\wedge e_{j_2}\wedge\cdots\wedge e_{j_{z-1}}\wedge e_{j_q}\wedge e_{j_{z+1}}\wedge \cdots\wedge e_{j_{q-1}}\wedge e_{j_1}\wedge e_{j_{q+1}}\wedge \cdots\wedge e_{j_t}\\
&=e_{j_z}\wedge e_{j_2}\wedge\cdots\wedge e_{j_{z-1}}\wedge e_{j_1}\wedge e_{j_{z+1}}\wedge \cdots\wedge e_{j_{q-1}}\wedge e_{j_q}\wedge e_{j_{q+1}}\wedge \cdots\wedge e_{j_t}\\
&=-e_{j_1}\wedge e_{j_2}\wedge\cdots\wedge e_{j_{z-1}}\wedge e_{j_z}\wedge e_{j_{z+1}}\wedge \cdots\wedge e_{j_{q-1}}\wedge e_{j_q}\wedge e_{j_{q+1}}\wedge \cdots\wedge e_{j_t}
\end{align*}
and Case 1 explains the second equality.\footnote{This visual argument assumes that $z<q$. The case $q<z$ is handled similarly, and the case
$z=q$ follows from Case 3.}
This is the desired conclusion here.

Case 5: the general case.
Cases 1--4 show that the desired result holds for any transposition that fixes $C:=\{1,\ldots,n\}\ssm\{j_1,\ldots,j_t\}$.
In general, since $\iota$ fixes $C$, it is a product $\iota=\tau_1\cdots\tau_m$ of transpositions that fix $C$.
(For instance, this can be seen by decomposing $\iota$, considered as an element of $S_t$.)
Since the result holds for each $\tau_i$,  induction on $m$ shows that
the desired result holds for $\iota$.
The main point is the following: if the result holds for permutations $\delta$ and $\sigma$ that fix $C$, then
\begin{align*}
e_{\delta\sigma(j_1)}\wedge\cdots\wedge e_{\delta\sigma(j_t)}
&=\operatorname{sgn}(\delta)e_{\sigma(j_1)}\wedge\cdots\wedge e_{\sigma(j_t)}\\
&=\operatorname{sgn}(\delta)\operatorname{sgn}(\sigma)e_{j_1}\wedge\cdots\wedge e_{j_t}\\
&=\operatorname{sgn}(\delta\sigma)e_{j_1}\wedge\cdots\wedge e_{j_t}
\end{align*}
so the result holds for $\delta\sigma$.
\qed
\end{para}

\begin{para}[Sketch of Solution to Exercise~\ref{exer120522q'}]
\label{para130312g'}
Eventually, we will argue by induction on $s$. To remove some technical issues from the induction argument, we deal with some degenerate cases first.
If $i_p=i_q$ for some $p< q$, then 
by definition, we have
\begin{align*}
(e_{i_1}\wedge \cdots\wedge e_{i_p}\wedge \cdots \wedge e_{i_q}\wedge  \cdots \wedge e_{i_s})\wedge (e_{j_1}\wedge \cdots\wedge e_{j_t})\hspace{-38mm}\\
&=0\wedge (e_{j_1}\wedge \cdots\wedge e_{j_t}) \\
&=0 \\
&=e_{i_1}\wedge \cdots\wedge e_{i_p}\wedge \cdots \wedge e_{i_q}\wedge  \cdots \wedge e_{i_s}\wedge e_{j_1}\wedge \cdots\wedge e_{j_t}.
\end{align*}
If $j_p=j_q$ for some $p< q$, then the same logic applies.

For the rest of the proof, we assume that $i_p\neq i_q$ and $j_p\neq j_q$ for all $p<q$.
We argue by induction on $s$.

Base case: $s=1$. 
There are two cases to consider.
If $i=j_q$ for some $q$, then Definitions~\ref{defn120522l} and~\ref{disc130316a} imply that
\begin{align*}
e_{i_1}\wedge (e_{j_1}\wedge \cdots\wedge e_{j_t})
&=0
=e_{i_1}\wedge e_{i_s}\wedge e_{j_1}\wedge \cdots\wedge e_{j_t}.
\end{align*}
If $i\neq j_q$ for all $q$, then
Definition~\ref{disc130316a} gives the desired equality directly
\begin{align*}
e_{i_1}\wedge (e_{j_1}\wedge \cdots\wedge e_{j_t})
=e_{i_1}\wedge e_{i_s}\wedge e_{j_1}\wedge \cdots\wedge e_{j_t}.
\end{align*}

Induction step: assume that $s\geq 2$ and that
\begin{align*}
(e_{k_1}\wedge e_{k_2}\wedge \cdots\wedge e_{k_{s-1}})\wedge (e_{j_1}\wedge \cdots\wedge e_{j_t})
&=e_{k_1}\wedge e_{k_2}\wedge \cdots\wedge e_{k_{s-1}}\wedge e_{j_1}\wedge \cdots\wedge e_{j_t}
\end{align*}
for all sequences $k_1,\ldots,k_{s-1}$ of distinct elements in $\{1,\ldots,n\}$.
Let $\alpha$ be the permutation of $\{i_1,\ldots,i_s\}$ such that $\alpha(i_1)<\cdots<\alpha(i_s)$, and
let $\beta$ be the permutation of $\{j_1,\ldots,j_t\}$ such that $\beta(j_1)<\cdots<\beta(j_t)$.
Exercise~\ref{exer120522q} explains the first step in the next sequence,
and the second equality is from Definition~\ref{defn120522l}:
\begin{align*}
(e_{i_1}\wedge \cdots\wedge e_{i_s})\wedge (e_{j_1}\wedge \cdots\wedge e_{j_t})\hspace{-30mm}\\
&=\signum\alpha\signum\beta(e_{\alpha(i_1)}\wedge  \cdots\wedge e_{\alpha(i_s)})\wedge (e_{\beta(j_1)}\wedge \cdots\wedge e_{\beta(j_t)})\\
&=\signum\alpha\signum\beta e_{\alpha(i_1)}\wedge [(e_{\alpha(i_2)}\wedge \cdots\wedge e_{\alpha(i_s)})\wedge (e_{\beta(j_1)}\wedge \cdots\wedge e_{\beta(j_t)})]\\
&=\signum\alpha\signum\beta e_{\alpha(i_1)}\wedge (e_{\alpha(i_2)}\wedge \cdots\wedge e_{\alpha(i_s)}\wedge e_{\beta(j_1)}\wedge \cdots\wedge e_{\beta(j_t)}).
\end{align*}
The third equality follows from our induction hypothesis.
If $i_p=j_q$ for some $p$ and $q$, then $\alpha(i_{p'})=\beta(j_{q'})$ for some $p'$ and $q'$,
so the first paragraph of this proof (or the base case, depending on the situation) implies that
\begin{align*}
(e_{i_1}\wedge \cdots\wedge e_{i_s})\wedge (e_{j_1}\wedge \cdots\wedge e_{j_t})\hspace{-30mm}\\
&=\signum\alpha\signum\beta e_{\alpha(i_1)}\wedge (e_{\alpha(i_2)}\wedge \cdots\wedge e_{\alpha(i_s)}\wedge e_{\beta(j_1)}\wedge \cdots\wedge e_{\beta(j_t)})\\
&=0 \\
&=e_{i_1}\wedge \cdots\wedge e_{i_s}\wedge e_{j_1}\wedge \cdots\wedge e_{j_t}.
\end{align*}
Thus, we assume that $i_p\neq j_q$ for all $p$ and $q$. In particular, the permutations $\alpha$ and $\beta$ combine to give
a permutation $\gamma$ of $\{i_1,\ldots,i_s,j_1,\ldots,j_t\}$ given by the rules
$\gamma(i_p)=\alpha(i_p)$ and $\gamma(j_q)=\beta(j_q)$ for all $p$ and $q$.
Furthermore, one has $\signum\gamma=\signum\alpha\signum\beta$.
(To see this, write $\alpha$  as a product $\alpha=\tau_1\cdots\tau_u$ of transpositions on $\{i_1,\ldots,i_s\}$,
write $\beta$  as a product $\beta=\pi_1\cdots\pi_u$ of transpositions on $\{j_1,\ldots,j_t\}$,
and observe that $\gamma=\tau_1\cdots\tau_u\pi_1\cdots\pi_u$ is a product of transpositions on $\{i_1,\ldots,i_s,j_1,\ldots,j_t\}$.)
This explains the second equality in the next sequence:
\begin{align*}
(e_{i_1}\wedge \cdots\wedge e_{i_s})\wedge (e_{j_1}\wedge \cdots\wedge e_{j_t})\hspace{-30mm}\\
&=\signum\alpha\signum\beta e_{\alpha(i_1)}\wedge (e_{\alpha(i_2)}\wedge \cdots\wedge e_{\alpha(i_s)}\wedge e_{\beta(j_1)}\wedge \cdots\wedge e_{\beta(j_t)})\\
&=\signum\gamma e_{\gamma(i_1)}\wedge (e_{\gamma(i_2)}\wedge \cdots\wedge e_{\gamma(i_s)}\wedge e_{\gamma(j_1)}\wedge \cdots\gamma e_{\beta(j_t)})\\
&=\signum\gamma e_{\gamma(i_1)}\wedge e_{\gamma(i_2)}\wedge \cdots\wedge e_{\gamma(i_s)}\wedge e_{\gamma(j_1)}\wedge \cdots\gamma e_{\beta(j_t)}\\
&=e_{i_1}\wedge \cdots\wedge e_{i_s}\wedge e_{j_1}\wedge \cdots\wedge e_{j_t}.
\end{align*}
The third equality is by our base case. The fourth equality is from Exercise~\ref{exer120522q}.
\qed
\end{para}

\begin{para}[Sketch of Solution to Exercise~\ref{exer120522r}]
\label{para130312h}
\eqref{exer120522r1}
Multiplication in $\bigwedge R^n$ is distributive and unital, by definition.
We check associativity first on basis elements, using Exercise~\ref{exer120522q'}:
\begin{align*}
[(e_{i_1}\wedge \cdots\wedge e_{i_s})\wedge (e_{j_1}\wedge \cdots\wedge e_{j_t})]\wedge (e_{k_1}\wedge \cdots\wedge e_{k_u})\hspace{-40mm}\\
&=(e_{i_1}\wedge \cdots\wedge e_{i_s}\wedge e_{j_1}\wedge \cdots\wedge e_{j_t})\wedge (e_{k_1}\wedge \cdots\wedge e_{k_u})\\
&=e_{i_1}\wedge \cdots\wedge e_{i_s}\wedge e_{j_1}\wedge \cdots\wedge e_{j_t}\wedge e_{k_1}\wedge \cdots\wedge e_{k_u}\\
&=e_{i_1}\wedge \cdots\wedge e_{i_s}\wedge (e_{j_1}\wedge \cdots\wedge e_{j_t}\wedge e_{k_1}\wedge \cdots\wedge e_{k_u})\\
&=e_{i_1}\wedge \cdots\wedge e_{i_s}\wedge [(e_{j_1}\wedge \cdots\wedge e_{j_t})\wedge (e_{k_1}\wedge \cdots\wedge e_{k_u})].
\end{align*}
The general associativity holds by distributivity: if $\alpha_1,\ldots,\alpha_a,\beta_1,\ldots,\beta_b,\gamma_1,\ldots,\gamma_c$ are basis elements in
$\bigwedge R^n$, then we have
\begin{align*}
\textstyle\left[\left(\sum_iz_i\alpha_i\right)\wedge\left(\sum_jy_j\beta_j\right)\right]\wedge\left(\sum_kz_k\gamma_k\right)
&=\textstyle\sum_{i,j,k}z_iy_jz_k[(\alpha_i\wedge\beta_j)\wedge\gamma_k] \\
&=\textstyle\sum_{i,j,k}z_iy_jz_k[\alpha_i\wedge(\beta_j\wedge\gamma_k)] \\
&=\textstyle\left(\sum_iz_i\alpha_i\right)\wedge\left[\left(\sum_jy_j\beta_j\right)\wedge\left(\sum_kz_k\gamma_k\right)\right].
\end{align*}

\eqref{exer120522r2}
As in part~\eqref{exer120522r1}, it suffices to consider basis vectors
$\alpha=e_{i_1}\wedge \cdots\wedge e_{i_s}\in \bigwedge^s R^n$ and $\beta=e_{j_1}\wedge \cdots\wedge e_{j_t}\in \bigwedge^t R^n$,
and prove that  $\alpha\wedge\beta=(-1)^{st}\beta\wedge\alpha$.
Note that we are assuming that $i_1<\cdots<i_s$ and $j_1<\cdots<j_t$.
If $i_p=j_q$ for some $p$ and $q$, then Exercise~\ref{exer120522q'} implies that 
$$\alpha\wedge\beta=0=(-1)^{st}\beta\wedge\alpha$$
as desired. Thus, we assume that $i_p\neq j_q$ for all $p$ and $q$, that is, 
there are no repetitions in the list $i_1,\ldots,i_s,j_1,\ldots,j_t$.

We proceed by induction on $s$.

Base case: $s=1$. In this case, the first and third equalities in the next sequence are from Exercise~\ref{exer120522q'}, using the condition $s=1$:
\begin{align*}
\alpha\wedge\beta
&=e_{i_1}\wedge  e_{j_1}\wedge \cdots\wedge e_{j_t}\\
&= (-1)^te_{j_1}\wedge \cdots\wedge e_{j_t}\wedge e_{i_1}\\
&=(-1)^{st}\beta\wedge\alpha.
\end{align*}
The second equality is from Exercise~\ref{exer120522q}.

Induction step: assume that $s\geq 2$ and that 
\begin{align*}
(e_{i_2}\wedge \cdots\wedge e_{i_s}) \wedge (e_{j_1}\wedge \cdots\wedge e_{j_t})
&=(-1)^{(s-1)t} (e_{j_1}\wedge \cdots\wedge e_{j_t})\wedge(e_{i_2}\wedge \cdots\wedge e_{i_s}). 
\end{align*}
The first, third, and fifth equalities in the next sequence are by associativity:
\begin{align*}
\alpha\wedge\beta
&=e_{i_1}\wedge [(e_{i_2}\wedge \cdots\wedge e_{i_s}) \wedge (e_{j_1}\wedge \cdots\wedge e_{j_t})] \\
&=e_{i_1}\wedge [(-1)^{(s-1)t} (e_{j_1}\wedge \cdots\wedge e_{j_t})\wedge(e_{i_2}\wedge \cdots\wedge e_{i_s})] \\
&=(-1)^{(s-1)t} [e_{i_1}\wedge (e_{j_1}\wedge \cdots\wedge e_{j_t})]\wedge (e_{i_2}\wedge \cdots\wedge e_{i_s}) \\
&=(-1)^{(s-1)t} (-1)^t[(e_{j_1}\wedge \cdots\wedge e_{j_t})\wedge e_{i_1}]\wedge (e_{i_2}\wedge \cdots\wedge e_{i_s}) \\
&=(-1)^{st}(e_{j_1}\wedge \cdots\wedge e_{j_t})\wedge [e_{i_1}\wedge (e_{i_2}\wedge \cdots\wedge e_{i_s})] \\
&=(-1)^{st}\beta\wedge\alpha.
\end{align*}
The second equality is by our induction hypothesis,  the fourth equality is from our base case,
and the fifth equality is by Exercise~\ref{exer120522q'}.

\eqref{exer120522r3}
Let $\alpha\in \bigwedge^s R^n$ such that $s$ is odd.
If $\alpha$ is a basis vector, then $\alpha\wedge\alpha=0$ by Exercise~\ref{exer120522q'}.
In general, we have $\alpha=\sum_iz_iv_i$ where the $v_i$ are basis vectors in $\bigwedge^s R^n$, and we compute:
\begin{align*}
\alpha\wedge\alpha
&\textstyle=\left(\sum_iz_iv_i\right)\wedge\left(\sum_iz_iv_i\right)\\
&\textstyle=\sum_{i,j}z_iz_jv_i\wedge v_j\\
&\textstyle=\sum_{i=j}z_iz_jv_i\wedge v_j+\sum_{i\neq j}z_iz_jv_i\wedge v_j\\
&\textstyle=\sum_{i}z_i^2\underbrace{v_i\wedge v_i}_{=0}+\sum_{i< j}z_iz_j(v_i\wedge v_j+v_j\wedge v_i)\\
&\textstyle=\sum_{i< j}z_iz_j(v_i\wedge v_j+(-1)^{s^2}v_i\wedge v_j)\\
&\textstyle=\sum_{i< j}z_iz_j(v_i\wedge v_j-v_i\wedge v_j)\\
&=0.
\end{align*}
The vanishing $v_i\wedge v_i=0$ follows from the fact that $v_i$ is a basis vector,
so the fifth equality follows from part~\eqref{exer120522r2}.
The sixth equality follows from the fact that $s$ is odd, and the other equalities are routine.

\eqref{exer120522r4}
The map $R\to\bigwedge R^n$ is a ring homomorphism 
since  multiplication
$\bigwedge^0R^n\times\bigwedge^0R^n\to\bigwedge^{0}R^n$
is defined as the usual scalar multiplication $R\times\bigwedge^0R^n\to\bigwedge^{0}R^n$.
To see that the image is in the center, let $r\in R=\bigwedge^0R^n$ and let 
$z_i\in\bigwedge^iR^n$ for $i=0,\ldots,n$. 
Using part~\eqref{exer120522r2}, we have
\begin{align*}
\textstyle r\wedge \left(\sum_iz_i\right)
&=\textstyle \sum_i(r\wedge z_i)
=\textstyle \sum_i (z_i\wedge r)
=\textstyle  \left(\sum_iz_i\right)\wedge r
\end{align*}
so $r$ in the center of $\bigwedge R^n$.
\qed
\end{para}

\begin{para}[Sketch of Solution to Exercise~\ref{exer130318a}]
\label{para130318a}
We argue by cases.

Case 1: $j_p=j_q$ for some $p< q$.
In this case, we have 
$$\partial^{\wti K^R(\x)}_i(e_{j_1}\wedge\cdots\wedge e_{j_t})=\partial^{\wti K^R(\x)}_i(0)=0$$
so we need to prove that 
$\sum_{s=1}^t(-1)^{s+1}x_{j_s}e_{j_1}\wedge\cdots\wedge \widehat{e_{j_s}}\wedge\cdots\wedge e_{j_t}=0$.
In this sum, if $s\notin\{p,q\}$, then we have $e_{j_1}\wedge\cdots\wedge \widehat{e_{j_s}}\wedge\cdots\wedge e_{j_t}=0$
since the repetition $j_p=j_q$ occurs in this element. This explains the first equality in the next sequence:
\begin{align*}
\textstyle\sum_{s=1}^t(-1)^{s+1}x_{j_s}e_{j_1}\wedge\cdots\wedge \widehat{e_{j_s}}\wedge\cdots\wedge e_{j_t}\hspace{-60mm}\\
&=(-1)^{p+1}x_{j_p}e_{j_1}\wedge\cdots\wedge \widehat{e_{j_p}}\wedge\cdots\wedge e_{j_{t}}
+(-1)^{q+1}x_{j_p}e_{j_1}\wedge\cdots\wedge \widehat{e_{j_{q}}}\wedge\cdots\wedge e_{j_t}\\
&=(-1)^{p+1}x_{j_p}e_{j_1}\wedge\cdots\wedge e_{j_{p-1}}\wedge  e_{j_{p+1}}\wedge\cdots\wedge e_{j_{q-1}}\wedge e_{j_q}\wedge e_{j_{q+1}}\wedge\cdots\wedge e_{j_t}\\
&\quad+(-1)^{q+1}x_{j_p}e_{j_1}\wedge\cdots\wedge e_{j_{p-1}}\wedge e_{j_p}\wedge 
\underbrace{e_{j_{p+1}}\wedge\cdots\wedge e_{j_{q-1}}}_{\text{$q-p-1$ factors}}\wedge e_{j_{q+1}}\wedge\cdots\wedge e_{j_t}\\
&=(-1)^{p+1}x_{j_p}e_{j_1}\wedge\cdots\wedge e_{j_{p-1}}\wedge  e_{j_{p+1}}\wedge\cdots\wedge e_{j_{q-1}}\wedge e_{j_q}\wedge e_{j_{q+1}}\wedge\cdots\wedge e_{j_t}\\
&\,\,\,+(-1)^{q+1+q-p-1}x_{j_p}e_{j_1}\wedge\cdots\wedge e_{j_{p-1}}\wedge 
e_{j_{p+1}}\wedge\cdots\wedge e_{j_{q-1}}\wedge e_{j_p}\wedge e_{j_{q+1}}\wedge\cdots\wedge e_{j_t}\\
&=(-1)^{p+1}x_{j_p}e_{j_1}\wedge\cdots\wedge e_{j_{p-1}}\wedge  e_{j_{p+1}}\wedge\cdots\wedge e_{j_{q-1}}\wedge e_{j_q}\wedge e_{j_{q+1}}\wedge\cdots\wedge e_{j_t}\\
&\quad+(-1)^{p}x_{j_p}e_{j_1}\wedge\cdots\wedge e_{j_{p-1}}\wedge 
e_{j_{p+1}}\wedge\cdots\wedge e_{j_{q-1}}\wedge e_{j_p}\wedge e_{j_{q+1}}\wedge\cdots\wedge e_{j_t}\\
&=0.
\end{align*}
The second equality is just a rewriting of the terms,
and the third equality is from Exercise~\ref{exer120522q}.
The fourth and fifth equalities are routine simplification.

Case 2: $j_p\neq j_q$ for all $p\neq q$.

Claim. Given a permutation $\iota$ of $j_1,\ldots,j_t$,
setting
$$\widehat\iota(j_1,\ldots,j_t):=\textstyle\sum_{s=1}^t(-1)^{s+1}x_{\iota(j_s)}e_{\iota(j_1)}\wedge\cdots\wedge \widehat{e_{\iota(j_s)}}\wedge\cdots\wedge e_{\iota(j_t)}$$
we have
\begin{align}
\widehat\iota(j_1,\ldots,j_t)&=\signum\iota\widehat\id(j_1,\ldots,j_t) 
\label{eq130318a}
\end{align}
where $\id$ is the identity permutation.
To prove this, write $\iota$ as a product $\iota=\tau_1\cdots\tau_u$ of adjacent transpositions, that is, $\tau_i=(j_{p_i}\,\,\, j_{p_i+1})$
for some $p_i<t$.
We argue by induction on $u$.

Base case: $u=1$. In this case, $\iota=\tau_1$ is a transposition $\iota=(j_p\,\,\, j_{p+1})$ for some $p<t$.
In this case, we are proving the equality $\widehat\iota(j_1,\ldots,j_t)=-\widehat\id(j_1,\ldots,j_t)$.
We write out the terms of $\widehat\iota(j_1,\ldots,j_t)$, starting with the case $s<p$:
\begin{align*}
(-1)^{s+1}x_{\iota(j_s)}e_{\iota(j_1)}\!\wedge\cdots\wedge \widehat{e_{\iota(j_s)}}\wedge\cdots\wedge e_{\iota(j_{p-1})}\!\wedge e_{\iota(j_{p})}\!\wedge e_{\iota(j_{p+1})}\!\wedge e_{\iota(j_{p+2})}\!\wedge\cdots\wedge e_{\iota(j_t)}
\hspace{-113mm}\\
&=(-1)^{s+1}x_{j_s}e_{j_1}\wedge\cdots\wedge \widehat{e_{j_s}}\wedge\cdots\wedge e_{j_{p-1}}\wedge e_{j_{p+1}}\wedge e_{j_{p}}\wedge e_{j_{p+2}}\wedge\cdots\wedge e_{j_t}\\
&=-(-1)^{s+1}x_{j_s}e_{j_1}\wedge\cdots\wedge \widehat{e_{j_s}}\wedge\cdots\wedge e_{j_{p-1}}\wedge e_{j_{p}}\wedge e_{j_{p+1}}\wedge e_{j_{p+2}}\wedge\cdots\wedge e_{j_t}.
\end{align*}
This is  the $s$th term of $-\widehat\id(j_1,\ldots,j_t)$ in this case.
Similar reasoning yields the same conclusion when $s>p+1$.
In the case $s=p$, the $p$th term of $\widehat\iota(j_1,\ldots,j_t)$ is
\begin{align*}
(-1)^{p+1}x_{\iota(j_p)}e_{\iota(j_1)}\wedge\cdots\wedge e_{\iota(j_{p-1})}\wedge \widehat{e_{\iota(j_p)}}\wedge e_{\iota(j_{p+1})}\wedge e_{\iota(j_{p+2})}\wedge\cdots\wedge e_{\iota(j_t)}\hspace{-85mm}\\
&=(-1)^{p+1}x_{j_{p+1}}e_{j_1}\wedge\cdots\wedge e_{j_{p-1}}\wedge \widehat{e_{j_{p+1}}}\wedge e_{j_{p}}\wedge  e_{j_{p+2}}\wedge\cdots\wedge e_{j_t}\\
&=(-1)^{p+1}x_{j_{p+1}}e_{j_1}\wedge\cdots\wedge e_{j_{p-1}}\wedge e_{j_{p}}\wedge \widehat{e_{j_{p+1}}}\wedge  e_{j_{p+2}}\wedge\cdots\wedge e_{j_t}\\
&=-(-1)^{p+2}x_{j_{p+1}}e_{j_1}\wedge\cdots\wedge e_{j_{p-1}}\wedge e_{j_{p}}\wedge \widehat{e_{j_{p+1}}}\wedge  e_{j_{p+2}}\wedge\cdots\wedge e_{j_t}.
\end{align*}
This is exactly the $p+1$st term of $-\widehat\id(j_1,\ldots,j_t)$.
Similar reasoning shows that the $p+1$st term of $\widehat\iota(j_1,\ldots,j_t)$ is exactly the $p$th term of $-\widehat\id(j_1,\ldots,j_t)$.
It follows that equation~\eqref{eq130318a} holds in the base case.

Induction step: Assume that $u\geq 2$ and that the result holds for the permutation $\iota':=\tau_2\cdots\tau_u$.
Note that this implies that $\iota=\tau_1\iota'$, so we have
$\signum{\iota}=-\signum{\iota'}$.
This explains the first and last equalities in the next sequence:
\begin{align*}
\widehat\iota(j_1,\ldots,j_t)
&=\widehat{\tau_1\iota'}(j_1,\ldots,j_t)\\
&=\widehat{\tau_1}(\iota'(j_1),\ldots,\iota'(j_t))\\
&=-\widehat{\id}(\iota'(j_1),\ldots,\iota'(j_t))\\
&=-\widehat{\iota'}(j_1,\ldots,j_t)\\
&=-\signum{\iota'}\widehat\id(j_1,\ldots,j_t)\\
&=\signum{\iota}\widehat\id(j_1,\ldots,j_t).
\end{align*}
The second and fourth equalities follow easily from the definition of $\widehat{*}$.
The third equality is from our base case, and the fifth equality is from our induction hypothesis. 
This concludes the proof of the claim.

Now we complete the proof of the exercise.
Let $\iota$ be a permutation of $j_1,\ldots,j_t$ such that
$\iota(j_1)<\cdots<\iota(j_t)$.
Exercise~\ref{exer120522q} explains the first equality in the next sequence, and the second, third, and fifth equalities are by definition:
\begin{align*}
\partial^{\wti K^R(\x)}_i(e_{j_1}\wedge\cdots\wedge e_{j_t}) 
&=\signum\iota\partial^{\wti K^R(\x)}_i(e_{\iota(j_1)}\wedge\cdots\wedge e_{\iota(j_t)})\\
&=\signum\iota\textstyle\sum_{s=1}^t(-1)^{s+1}x_{\iota(j_s)}e_{\iota(j_1)}\wedge\cdots\wedge 
\widehat{e_{\iota(j_s)}}\wedge\cdots\wedge e_{\iota(j_t)}\\
&= \signum\iota\widehat{\iota}(j_1,\ldots,j_t)\\
&= \widehat{\id}(j_1,\ldots,j_t)\\
&=\textstyle\sum_{s=1}^t(-1)^{s+1}x_{j_s}e_{j_1}\wedge\cdots\wedge \widehat{e_{j_s}}\wedge\cdots\wedge e_{j_t}
\end{align*}
The fourth equality is by the claim we proved above.
\qed
\end{para}

\begin{para}[Sketch of Solution to Exercise~\ref{exer120522s}]
\label{para130312i}
It suffices to verify the formula for
basis vectors $\alpha=e_{j_1}\wedge\cdots\wedge e_{j_s}\in \bigwedge^s R^n$ and $\beta=e_{j_{s+1}}\wedge\cdots\wedge e_{j_{s+t}}\in \bigwedge^t R^n$.
\begin{align*}
\partial^{\wti K^R(\x)}_{s+t}(\alpha\wedge\beta)\hspace{-19mm} \\
&=\partial^{\wti K^R(\x)}_{s+t}(e_{j_1}\wedge\cdots\wedge e_{j_{s+t}})\\
&=\textstyle\sum_{u=1}^{s+t}(-1)^{u+1}x_{i_u}e_{j_1}\wedge\cdots\wedge \widehat{e_{j_u}}\wedge\cdots\wedge e_{j_{s+t}}\\
&=\textstyle\sum_{u=1}^{s}(-1)^{u+1}x_{i_u}e_{j_1}\wedge\cdots\wedge \widehat{e_{j_u}}\wedge\cdots\wedge e_{j_{s}}
\wedge e_{j_{s+1}}\wedge\cdots\wedge e_{j_{s+t}}\\
&\quad+\textstyle\sum_{u=s+1}^{s+t}(-1)^{u+1}x_{i_u} e_{j_1}\wedge\cdots\wedge e_{j_s}\wedge
e_{j_{s+1}}\wedge\cdots\wedge \widehat{e_{j_u}}\wedge\cdots\wedge e_{j_{s+t}}\\
&=\textstyle\sum_{u=1}^{s}(-1)^{u+1}x_{i_u}e_{j_1}\wedge\cdots\wedge \widehat{e_{j_u}}\wedge\cdots\wedge e_{j_{s}}
\wedge e_{j_{s+1}}\wedge\cdots\wedge e_{j_{s+t}}\\
&\quad+\textstyle\sum_{u=1}^{t}(-1)^{u-s+1}x_{i_{s+u}} e_{j_1}\wedge\cdots\wedge e_{j_s}\wedge
e_{j_{s+1}}\wedge\cdots\wedge \widehat{e_{j_{s+u}}}\wedge\cdots\wedge e_{j_{s+t}}\\
&=\textstyle\left[\sum_{u=1}^{s}(-1)^{u+1}x_{i_u}e_{j_1}\wedge\cdots\wedge \widehat{e_{j_u}}\wedge\cdots\wedge e_{j_{s}}\right]
\wedge (e_{j_{s+1}}\wedge\cdots\wedge e_{j_{s+t}})\\
&\quad+\textstyle(-1)^s(e_{j_1}\!\wedge\cdots\wedge e_{j_s})\wedge\left[
\sum_{u=1}^{t}(-1)^{u+1}x_{i_{s+u}}e_{j_{s+1}}\!\wedge\cdots\wedge \widehat{e_{j_{s+u}}}\wedge\cdots\wedge e_{j_{s+t}}\right]\\
&= \partial^{\wti K^R(\x)}_{s}(e_{j_1}\wedge\cdots\wedge e_{j_s})\wedge (e_{j_{s+1}}\wedge\cdots\wedge e_{j_{s+t}})\\
&\quad+(-1)^s(e_{j_1}\wedge\cdots\wedge e_{j_s})\wedge\partial^{\wti K^R(\x)}_{t}(e_{j_{s+1}}\wedge\cdots\wedge e_{j_{s+t}})\\
&=\partial^{\wti K^R(\x)}_{s}(\alpha)\wedge\beta+(-1)^s\alpha\wedge\partial^{\wti K^R(\x)}_{t}(\beta).
\end{align*}
The first, second, and last equalities are by definition.
The third equality is obtained by splitting the sum, and the fourth equality is by reindexing.
The fifth equality is by distributivity,
and the sixth equality is by Exercise~\ref{exer130318a}.
\qed
\end{para}

\begin{para}[Sketch of Solution to Exercise~\ref{fact110216a}]
\label{para130312j}
The fact that multiplication in $A$ is distributive implies that the map
$\mu^A\colon \Otimes AA\to A$ given by $\mu^A(a\otimes b)=ab$
is well-defined and $R$-linear. To see that it is a chain map, we compute:
\begin{align*}
\partial^A_{|a|+|b|}(\mu^A_{|a|+|b|}(a\otimes b))
&=\partial^A_{|a|+|b|}(ab) \\
&= \partial^A_{|a|}(a)b+(-1)^{|a|}a\partial^A_{|b|}(b)\\
&=\mu^A_{|a|+|b|-1}(\partial^A_{|a|}(a)\otimes b+(-1)^{|a|}a\otimes \partial^A_{|b|}(b))\\
&=\mu^A_{|a|+|b|-1}(\partial^{\Otimes AA}_{|a|+|b|}(a\otimes b)).
\end{align*}

Since the multiplication on $A$ maps $A_0\times A_0\to A_{0+0}=A_0$, we conclude that $A_0$ is closed under multiplication.
Also, the fact that $A_0$ is an $R$-module implies that it is closed under addition and subtraction.
Thus, the fact that $A_0$ is a commutative ring can be shown by restricting the axioms of $A$ to $A_0$.
To show that $A_0$ is an $R$-algebra, we define a map $R\to A_0$ by 
$r\mapsto r1_A$. Since $A_0$ is a ring and an $R$-module, it is straightforward to show that this is a ring homomorphism,
so $A_0$ is an $R$-algebra.
\qed
\end{para}

\begin{para}[Sketch of Solution to Exercise~\ref{ex120523a'''}]
\label{para130312k}
\eqref{ex120523a'''1}
It is straightforward to show  that the map $R\to A$ given by $r\mapsto r\cdot 1_A$ respects multiplication.
The fact that it is a morphism of DG $R$-algebras
follows from the  commutativity of the next diagram
$$\xymatrix{
R=
&\cdots\ar[r]
&0\ar[r]\ar[d]
&0\ar[r]\ar[d]
&R\ar[r]\ar[d]
&0\ar[d]
\\
A=
&\cdots\ar[r]
&A_2\ar[r]
&A_1\ar[r]
&A_0\ar[r]
&0}$$
which is easily checked.

\eqref{ex120523a'''2}
Argue as in part~\eqref{ex120523a'''1}.

\eqref{ex120523a'''3}
With Exercise~\ref{exer120522g} and Lemma~\ref{lem120529a},
the essential point is the commutativity of the following diagram
$$\xymatrix@C=13mm{
K=
&\cdots\ar[r]
&R^{\binom n2}\ar[r]\ar[d]
&R^n\ar[r]^-{(x_1\,\,\, \ldots\,\,\, x_n)}\ar[d]
&R\ar[r]\ar[d]
&0\ar[d]
\\
R/(\x)=
&\cdots\ar[r]
&0\ar[r]
&0\ar[r]
&R/(\x)\ar[r]
&0}$$
which is easily checked.
\qed
\end{para}

\begin{para}[Sketch of Solution to Exercise~\ref{ex120523a''}]
\label{para130312l}
\eqref{ex120523a''1}
The condition $A_{-1}=0$ implies that 
$\HH_0(A)\cong A_0/\im(\partial^A_1)$, which is a homomorphic image of
$A_0$.
To show  that $\HH_0(A)$ is an $A_0$-algebra, it suffices to show that $\im(\partial^A_1)$ is an ideal of $A_0$.
To this end, the fact that $\partial^A_1$ is $R$-linear implies that $\im(\partial^A_1)$ is non-empty and closed under subtraction.
To show that it is closed under multiplication by elements of $A_0$, let $a_0\in A_0$ and $\partial^A_1(a_1)\in\im(\partial^A_1)$,
and use the Leibniz Rule
\begin{align*}
\partial^A_1(a_0a_1)
&=\partial^A_0(a_0)a_1+(-1)^0a_0\partial^A_1(a_1)=a_0\partial^A_1(a_1)
\end{align*}
to see that $a_0\partial^A_1(a_1)\in\im(\partial^A_1)$.

\eqref{ex120523a''2}
To see that $A_i$ is an $A_0$-module, first observe that multiplication in $A$ maps $A_0\times A_i$ to $A_i$.
Thus, $A_i$ is closed under scalar multiplication by $A_0$. Since $A_i$ is an $R$-module, it is
non-empty and closed under addition and subtraction. The remaining $A_0$-module axioms follow from
the DG algebra axioms on $A$.

To show that $\HH_i(A)$ is an $\HH_0(A)$-module, the essential point is to show that the scalar multiplication $\ol{a_0}\ \ol{a_i}:=\ol{a_0a_i}$
is well-defined. (Then the axioms follow directly from the fact that $A_i$ is an $A_0$-module.)
The well-definedness boils down to showing that the products
$\im(\partial^A_{1})\Ker(\partial^A_{i})$
and
$\Ker(\partial^A_{0})\im(\partial^A_{i+1})$ are contained in $\im(\partial^A_{i+1})$.
For the first of these, let $\partial^A_{1}(a_1)\in \im(\partial^A_{1})$ and $z_{i}\in \Ker(\partial^A_{i})$:
\begin{align*}
\partial^A_{i+1}(a_1z_i)
&=\partial^A_1(a_1)z_i+(-1)^1a_1\underbrace{\partial^A_i(z_i)}_{=0}=\partial^A_1(a_1)z_i.
\end{align*}
It follows that $\partial^A_1(a_1)z_i=\partial^A_{i+1}(a_1z_i)\in\im(\partial^A_{i+1})$, as desired.
The  containment $\Ker(\partial^A_{0})\im(\partial^A_{i+1})\subseteq\im(\partial^A_{i+1})$ is verified similarly.
\qed
\end{para}

\begin{para}[Sketch of Solution to Exercise~\ref{exer130312b}]
\label{para130312m}
Assume that $A$ is a DG $R$-algebra such that each $A_i$ is finitely generated.
Since $R$ is noetherian, each $A_i$ is noetherian over $R$.
Hence, each submodule $\Ker(\partial^A_i)\subseteq A_i$ is 
finitely generated over $R$, so each quotient $\HH_i(A)=\Ker(\partial^A_i)/\im(\partial^A_{i+1})$
is finitely generated over $R$. In particular, $\HH_0(A)$ is finitely generated over $R$,
so it is noetherian by the Hilbert Basis Theorem. Since $\HH_i(A)$ is finitely generated over $R$
and it is a module over the $R$-algebra $\HH_0(A)$, it is straightforward to show that 
$\HH_i(A)$ is finitely generated over  $\HH_0(A)$.
\qed
\end{para}

\begin{para}[Sketch of Solution to Exercise~\ref{ex110218a'}]
\label{para130313a}
\eqref{ex110218a'1}
By definition, every DG $R$-module is, in particular, an $R$-complex.
Conversely, let $X$ be an $R$-complex. We verify the DG $R$-module axioms.
The associative, distributive, and unital axioms are automatic since $X$ is an $R$-complex.

For the graded axiom, let $r\in R_i$ and $x\in X_j$.
If $i\neq 0$, then $r\in R_i=0$ implies that $rx=0\in X_{i+j}$.
If $i=0$, then $r\in R_i=R$ implies that $rx\in X_j=X_{0+j}$. 

For the Leibniz Rule, let $r\in R_i$ and $x\in X_j$.
If $i\neq 0$, then $r\in R_i=0$ implies  
$$\partial^X_{i+j}(rx)=\partial^X_{i+j}(0)=0=\partial^X_{i+j}(0)x+(-1)^i0\partial^X_{i+j}(x)=\partial^X_{i}(r)x+(-1)^ir\partial^X_{j}(x)$$
as desired.
If $i=0$, then $r\in R_i=R$, so the $R$-linearity of $\partial^X_j=\partial^X_{0+j}$ implies that 
$$\partial^X_{i+j}(rx)=r\partial^X_{0+j}(x)=0x+(-1)^0r\partial^X_{j}(x)=\partial^X_{i}(r)x+(-1)^ir\partial^X_{j}(x)$$
as desired.

\eqref{ex110218a'2}
Directly compare the axioms in Definitions~\ref{DGK} and~\ref{defn130313a}.

\eqref{ex110218a'3}
Let $f\colon A\to B$ a morphism  of DG $R$-algebras, and let $Y$ be a DG $B$-module. 
Define the DG $A$-module structure on $Y$ by the formula $a_iy_j:=f_i(a_i)y_j$.

We verify associativity: 
$$(a_ia_j)y_k=f_{i+j}(a_ia_j)y_k=[f_i(a_i)f_j(a_j)]y_k=f_i(a_i)[f_j(a_j)y_k]=a_i(a_jy_k).$$
Distributivity, gradedness, and the Leibniz Rule follow similarly, as does unitality (using the condition $f_0(1_A)=1_B$).
\qed
\end{para}

\begin{para}[Sketch of Solution to Exercise~\ref{ex120526b}]
\label{para130313b}
Fix an $R$-complex $X$ and a DG $R$-algebra $A$.
First, we show that the scalar multiplication $a(b\otimes x):=(ab)\otimes x$ is well-defined.
Let $a_i\in A_i$ be a fixed element.
Since the map $A_j\to A_{i+j}$ given by $a_j\mapsto a_ia_j$ is well-defined and $R$-linear,
so is the induced map $\Otimes{A_j}{X_k}\to\Otimes{A_{i+j}}{X_k}$ which is given on generators by
the formula $a_j\otimes x_k\mapsto(a_ia_j)\otimes x_k$.
Assembling these maps together for all $j,k$ provides a well-defined $R$-linear map
$$(\Otimes{A}X)_n=\bigoplus_{j+k=n}\Otimes{A_j}{X_k}\to\bigoplus_{j+k=n}\Otimes{A_{i+j}}{X_k}=(\Otimes{A}X)_{i+n}$$
given on generators by
the formula $a_j\otimes x_k\mapsto(a_ia_j)\otimes x_k$.
In other words, the given multiplication is well-defined and satisfies the graded axiom.
Verification of associativity and distributivity is routine.
We verify the Leibniz Rule on generators:
\begin{align*}
\partial^{\Otimes{A}X}_{i+j+k}(a_i(a_j\otimes x_k))\hspace{-10mm}\\
&=\partial^{\Otimes{A}X}_{i+j+k}((a_ia_j)\otimes x_k)\\
&=\partial^{A}_{i+j}(a_ia_j)\otimes x_k+(-1)^{i+j}(a_ia_j)\otimes \partial^{X}_{k}(x_k)\\
&=[\partial^{A}_{i}(a_i)a_j]\otimes x_k+(-1)^{i}[a_i\partial^{A}_{j}(a_j)]\otimes x_k+(-1)^{i+j}(a_ia_j)\otimes \partial^{X}_{k}(x_k)\\
&=\partial^A_i(a_i)(a_j\otimes x_k)+(-1)^{i}a_i[\partial^{A}_{j}(a_j)\otimes x_k+(-1)^{j}a_j\otimes \partial^{X}_{k}(x_k)]\\
&=\partial^A_i(a_i)(a_j\otimes x_k)+(-1)^ia_i\partial^{\Otimes{A}X}_{j+k}(a_j\otimes x_k).\hspace{28mm}\qed
\end{align*}
\end{para}

\begin{para}[Sketch of Solution to Exercise~\ref{exer130331a}]
\label{para130331a}
Since scalar multiplication of $A$ on $M$ is well-defined, so is the scalar multiplication of $A$ on $\shift^iM$.
To prove that $\shift^iM$ is a DG $A$-module, we check the axioms, in order, beginning with associativity:
\begin{align*}
a\ast(b\ast m)
&=(-1)^{i|a|}a((-1)^{i|b|}bm)=(-1)^{i(|a|+|b|)}a(bm)\\
&=(-1)^{i|ab|}(ab)m=(ab)*m.
\end{align*}
Distributivity is verified similarly.
For the unital axiom, recall that $|1_A|=0$, so we have
$1_A\ast m=(-1)^{i\cdot 0}1_Am=m$.

The graded axiom requires a bit of bookkeeping.
Let $a\in A_p$ and $m\in(\shift^iM)_q=M_{q-i}$.
Then we have $a\ast m=(-1)^{ip}a m\in M_{p+q-i}=(\shift^iM)_{p+q}$, as desired.

for the Leibniz Rule, we let $a\in A_p$ and $m\in(\shift^iM)_q=M_{q-i}$, and we
compute:
\begin{align*}
\hspace{12mm}&&\partial^{\shift^iM}_{p+q}(a\ast m)
&=(-1)^i\partial^M_{p+q-i}((-1)^{ip}am)\\
&&&=(-1)^{ip+i}[\partial^A_p(a)m+(-1)^pa\partial^M_{q-i}(m)]\\
&&&=(-1)^{ip-i}[\partial^A_p(a)m+(-1)^{p+i}a\partial^{\shift^iM}_{q}(m)]\\
&&&=(-1)^{i(p-1)}\partial^A_p(a) m+(-1)^{p+ip}a\partial^{\shift^iM}_{q}(m)\\
&&&=\partial^A_p(a)\ast m+(-1)^pa\ast\partial^{\shift^iM}_{q}(m).&&&&\hspace{12mm}\qed
\end{align*}
\end{para}

\begin{para}[Sketch of Solution to Exercise~\ref{fact110216a'}]
\label{para130313c}
Argue as for Exercise~\ref{fact110216a}; cf.~\ref{para130312j}.~\qed
\end{para}

\begin{para}[Sketch of Solution to Exercise~\ref{exer130313a}]
\label{para130313e}
Fix a DG $R$-algebra $A$ and a chain map of $R$-complexes 
$g\colon X\to Y$. 

\eqref{exer130313a1}
To prove that
the chain map $\Otimes{A}g\colon \Otimes{A}X\to \Otimes{A}Y$ 
is a morphism of
DG $A$-modules, we check the desired formula on generators:
\begin{align*}
(\Otimes{A}g)_{i+j+k}(a_i(a_j\otimes x_k))
&=(\Otimes{A}g)_{i+j+k}((a_ia_j)\otimes x_k))\\
&=(a_ia_j)\otimes g_k(x_k)\\
&=a_i(a_j\otimes g_k(x_k))\\
&=a_i(\Otimes{A}g)_{j+k}(a_j\otimes x_k).
\end{align*}

\eqref{exer130313a2}
Let $k$ be a field, and consider the power series ring $R=k[\![X]\!]$.
The residue field $k$ is an $R$-algebra, hence it is a DG $R$-algebra.
The Koszul complex $K=K^R(X)$ is  a free resolution of $k$ over $R$,
so the natural map $K\to k$ is a quasiisomorphism.
However, the induced map $\Otimes kK\to\Otimes kk$
is not a quasiisomorphism since $\HH_1(\Otimes kK)\cong k\neq 0=\HH_i(\Otimes kk)$.

\eqref{exer130313a3}
Fix a morphism $f\colon A\to B$ of DG $R$-algebras. This explains the first equality in the next display
$$f(ab)=f(a)f(b)=af(b).$$
The second equality is from the definition of the DG $A$-module structure on $B$.
\qed
\end{para}

\begin{para}[Sketch of Solution to Exercise~\ref{disc110302a}]
\label{para130330k}
We use the following items.

\begin{defn}
Let $X$ be an $R$-complex. An \emph{$R$-subcomplex} of $X$ is an $R$-complex $Y$ such that each $Y_i$ is a submodule of $X_i$
and for all $y\in Y_i$ we have $\partial^Y_i(y)=\partial^X_i(y)$.
A \emph{DG submodule} of the DG $A$-module $M$ is an $R$-subcomplex $N$ such that
the scalar multiplication of $A$ on $N$ uses the same rule as the scalar multiplication on $M$.
\end{defn}

\begin{lem}\label{lem130401a}
Let $X$ be an $R$-complex, and let $Y$ be an $R$-subcomplex of $X$. 
\begin{enumerate}[\rm(a)]
\item\label{lem130401a1}
The quotient complex
$X/Y$ with differential induced by $\partial^X$ is an $R$-complex. 
\item\label{lem130401a2}
The natural sequence
$0\to Y\xra\epsilon X\xra\pi X/Y\to 0$ 
of chain maps is exact.
\item\label{lem130401a3}
The natural map $X\to X/Y$ is a quasiisomorphism if and only if $Y$ is exact.
\end{enumerate}
\end{lem}

\begin{proof}
\eqref{lem130401a1}
To show that the differential $\partial^{X/Y}$ given as
$\partial^{X/Y}_i(\ol x):=\ol{\partial^X_i(x)}$ is well-defined,
it suffices to show that $\partial^X_i(Y_i)\subseteq Y_{i-1}$.
But this condition is automatic by definition: for all $y\in Y_i$ we have $\partial^X_i(y)=\partial^Y_i(y)\in Y_{i-1}$. 
The remaining $R$-complex axioms are straightforward consequences of the axioms for $X$.

\eqref{lem130401a2}
It suffices to show that the following diagram commutes
$$\xymatrix{
0\ar[r]
&Y_i\ar[r]^-{\epsilon_i}\ar[d]_{\partial^Y_i}
&X_i\ar[r]^-{\pi_i}\ar[d]_{\partial^X_i}
&X_i/Y_i\ar[r]\ar[d]_{\partial^{X/Y}_i}
&0 \\
0\ar[r]
&Y_{i-1}\ar[r]^-{\epsilon_{i-1}}
&X_{i-1}\ar[r]^-{\pi_{i-1}}
&X_{i-1}/Y_{i-1}\ar[r]
&0}$$
where $\epsilon_i$ is the inclusion and $\pi_i$ is the natural surjection.
The commutativity of this diagram is routine, using the assumption on $\partial^Y$ and the
definition of $\partial^{X/Y}$.

\eqref{lem130401a3}
Use the long exact sequence coming from part~\eqref{lem130401a2}.
\end{proof}

\begin{lem}\label{lem130401b}
Let $N$ be a DG submodule of the DG $A$-module $M$.
\begin{enumerate}[\rm(a)]
\item\label{lem130401b1}
The quotient complex
$M/N$ with scalar multiplication induced by the scalar multiplication on $M$ is a DG $A$-module.
\item\label{lem130401b2}
The natural maps
$N\xra\epsilon M\xra\pi M/N$ are morphisms of DG $A$-modules. 
\end{enumerate}
\end{lem}

\begin{proof}
\eqref{lem130401b1}
Lemma~\ref{lem130401a}\eqref{lem130401a1}
implies that $M/N$ is an $R$-complex. 
Next, we show that the scalar multiplication $a\ol m:=\ol{am}$ is well-defined.
For this, it suffices to show that $an\in N$ for all $a\in A$ and all $n\in N$.
But this condition is automatic by definition, since $N$ is closed under scalar multiplication. 
Now that the scalar multiplication on $M/N$ is well-defined, the DG $A$-module axioms on $M/N$ 
follow readily from the axioms on $M$.

\eqref{lem130401b2}
As $N$ is a DG submodule of $M$, the inclusion $\epsilon$ respects scalar multiplication;
and $\pi$ respects scalar multiplication as follows:
$\pi(am)=\ol{am}=a\ol m=a\pi(m)$.
\end{proof}

Now, we continue with the proof of the exercise.
Consider  the complex
$$N=\cdots\xra{\partial^M_{n+2}}M_{n+1}\xra{\partial^M_{n+1}}\im(\partial^M_{n+1})\to 0.$$

\eqref{disc110302a1}
Using Lemmas~\ref{lem130401a} and~\ref{lem130401b}, it suffices to show that $N$
is a DG submodule of $M$. By inspection, the differential on $N$  maps $N_i\to N_{i-1}$ for all $i$.
Since the differential and scalar multiplication on $N$
are induced from those on $M$, it suffices to show that $N$ is 
closed under scalar multiplication. 
(The other axioms are inherited from $M$.)
For this, let $a\in A_p$ and $x\in N_q$. If $p<0$, then $a=0$ and we have $ax=0\in N_{p+q}$.
Similar reasoning applies if $q<n$.
Assume now that $p\geq 0$ and $q\geq n$. If $p\geq 1$ or $q>n$, then $ax\in M_{p+q}=N_{p+q}$, by definition.
So, we are reduced to the case where $p=0$ and $q=n$.
In this case, there is an element $m\in M_{n+1}$ such that $x=\partial^M_{n+1}(m)$.
Thus, the Leibniz Rule on $M$ implies that
$$ax=0+ax=\partial^A_0(a)m+a\partial^M_{n+1}(m)=
\partial^M_{n+1}(am)\in\im(\partial^M_{n+1})=N_n$$
as desired.

\eqref{disc110302a2}
Using Lemmas~\ref{lem130401a} and~\ref{lem130401b}, it suffices to show that $N$
is exact if and only if $n\geq\sup M$.
For this, note that
$$\HH_i(N)\cong
\begin{cases}
\HH_i(M)&\text{for $i>n$} \\
0&\text{for $i\leq n$.}\end{cases}$$ 
It follows that the complex $N$ is exact if and only if $\HH_i(M)=0$ for all $i>n$, that is, if and only if $n\geq \sup(M)$.
\qed
\end{para}

\begin{para}[Sketch of Solution to Exercise~\ref{exer:grade2perfect}]
\label{para130405a}
Note that $I=(x^2,xy,y^2)$ has grade two as an ideal of $R$ since $x^2,y^2$ is an $R$-regular sequence.  Also, $I=I_2(A)$, where
\[
A=\left[\begin{matrix}x&0\\y&x\\0&y\end{matrix}\right]
\]
Thus, $I$ is perfect by Theorem \ref{thm:codim2}.
\qed
\end{para}

\begin{para}[Sketch of Solution to Exercise~\ref{exer130405a}]
\label{para130405b}
Let $A=(a_{i,j})$, which is $(n+1)\times n$.
Let $F$ denote the $R$-complex in question, and denote the differential of $F$ by $\dell$.
First, we check the easy products.
The fact that $F_3=0$ explains the first two equalities in the next sequence; the others are by definition and cancellation.
\begin{align*}
\dell_2(f_j)f_k+(-1)^2f_j\dell_2(f_k)
&=0\\
&=\dell_4(f_jf_k)\\
\dell_1(e_i)e_i+(-1)^1e_i\dell_1(e_i)
&=(-1)^{i-1}a\det(A_i)e_i-e_i[(-1)^{i-1}a\det(A_i)]\\
&=0\\
&=\dell_2(e_ie_i)
\end{align*}
The remaining equalities require some work.

Let $1\leq i<j\leq n+1$, and consider the $n\times n$ matrices $A_j$ and $A_i$. 
Expanding $\det(A_j)$ along the $i$th row, we have
\begin{equation}\label{eq130627a}
\det(A_j)=\sum_{k=1}^n(-1)^{i+k}a_{i,k}\det(A^k_{i,j}).
\end{equation}
This uses the equality $(A_j)^k_i=A^k_{i,j}$ which is a consequence of the assumption $i<j$.
On the other hand, we have $(A_i)^k_{j-1}=A^k_{i,j}$ since the $(j-1)$st row of $A_i$ is equal to the $j$th row of $A$; 
so when we expand $\det(A_i)$ along its $(j-1)$st row, we have
\begin{align}
\det(A_i)
&=\sum_{k=1}^n(-1)^{j-1+k}a_{j,k}\det((A_i)^k_{j-1})
=\sum_{k=1}^n(-1)^{j-1+k}a_{j,k}\det(A^k_{i,j}).\label{eq130627b}
\end{align}
Next, for $\ell\neq i,j$ we let $A(\ell)$ denote the matrix obtained by  replacing the $i$th row of $A$ with the $\ell$th row.
It follows that we have $(A(\ell)_j)^k_i=A(\ell)^k_{i,j}=A^k_{i,j}$.
Notice that the matrix $A(\ell)_j$ has two equal rows, so we have $\det(A(\ell)_j)=0$.
Expanding $\det(A(\ell)_j)$ along the $i$th row, we obtain the next equalities:
\begin{align}
0
&=\sum_{k=1}^n(-1)^{i+k}a_{\ell,k}\det((A(\ell)_j)^k_{i})
=\sum_{k=1}^n(-1)^{i+k}a_{\ell,k}\det(A^k_{i,j}).\label{eq130627c}
\end{align}
Now we verify the Leibniz rule for the product $e_ie_j$, still assuming $1\leq i<j\leq n+1$.
\begin{align*}
\dell_2(e_ie_j)&=a\sum_{k=1}^n(-1)^{i+j+k}\det(A^k_{i,j})\dell_2(f_k)\\
&=a\sum_{k=1}^n(-1)^{i+j+k}\det(A^k_{i,j})\sum_{\ell=1}^{n+1}a_{\ell, k}e_{\ell}\\
&=a\sum_{\ell=1}^{n+1}\left[\sum_{k=1}^n(-1)^{i+j+k}a_{\ell, k}\det(A^k_{i,j})\right]e_{\ell}\\
&=a\left[\sum_{k=1}^n(-1)^{i+j+k}a_{i, k}\det(A^k_{i,j})\right]e_{i}+a\left[\sum_{k=1}^n(-1)^{i+j+k}a_{j, k}\det(A^k_{i,j})\right]e_{j}\\
&\quad+a\sum_{\ell\neq i,j}\left[\sum_{k=1}^n(-1)^{i+j+k}a_{\ell, k}\det(A^k_{i,j})\right]e_{\ell}\\
&=(-1)^ja\left[\sum_{k=1}^n(-1)^{i+k}a_{i, k}\det(A^k_{i,j})\right]e_{i}\\
&\quad+(-1)^{i+1}a\left[\sum_{k=1}^n(-1)^{j-1+k}a_{j, k}\det(A^k_{i,j})\right]\!e_{j}\\
&\quad+a\sum_{\ell\neq i,j}\left[\sum_{k=1}^n(-1)^{i+j+k}a_{\ell, k}\det(A^k_{i,j})\right]e_{\ell}\\
&=(-1)^ja\det(A_j)e_{i}+(-1)^{i+1}a\det(A_i)e_{j}+(-1)^ja\sum_{\ell\neq i,j}0e_{\ell}\\
&=-\dell_1(e_j)e_i+\dell_1(e_i)e_j+0\\
&=\dell_1(e_i)e_j+(-1)^{|e_i|}e_i\dell_1(e_j)
\end{align*}
The sixth equality follows from~\eqref{eq130627a}--\eqref{eq130627c},
and the others are by definition and simplification.

Next, we show how the Leibniz rule for $e_je_i$ follows from that of $e_ie_j$:
\begin{align*}
\dell_2(e_je_i)
&=\dell_2(-e_ie_j)\\
&=-\dell_2(e_ie_j)\\
&=-[\dell_1(e_i)e_j-e_i\dell_1(e_j)]\\
&=e_i\dell_1(e_j)-\dell_1(e_i)e_j\\
&=\dell_1(e_j)e_i-e_j\dell_1(e_i).
\end{align*}

Next, we verify the Leibniz rule for products of the form $e_if_j$ for any $i$ and $j$.
To begin,  note that we have
$$(A_i)^j_k=\begin{cases}A^j_{k,i}&\text{if $k<i$} \\ A^j_{i,k+1}&\text{if $k\geq i$} \end{cases}$$
and the $k,j$-entry of $A_i$ is
$$a_{k,j}':=\begin{cases}a_{k,j}&\text{if $k<i$} \\ a_{k+1,j}&\text{if $k\geq i$.} \end{cases}$$
Using this, we expand $\det(A_i)$ along the $j$th column:
\begin{align}
\det(A_i)
&=\sum_{k=1}^n(-1)^{j+k}\det((A_i)^j_k)a_{k,j}' \notag \\
&=\sum_{k=1}^{i-1}(-1)^{j+k}\det(A_{k,i}^j)a_{k,j}\notag 
+\sum_{k=i}^n(-1)^{j+k}\det(A^j_{i,k+1})a_{k+1,j} \\
&=\sum_{k=1}^{i-1}(-1)^{j+k}\det(A_{k,i}^j)a_{k,j} \notag
+\sum_{k=i+1}^{n+1}(-1)^{j+k-1}\det(A^j_{i,k})a_{k,j} \\
&=\sum_{k=1}^{i-1}(-1)^{j+k}\det(A_{k,i}^j)a_{k,j} 
-\sum_{k=i+1}^{n+1}(-1)^{j+k}\det(A^j_{i,k})a_{k,j}. \label{eq130627d}
\end{align}
Similarly, for any $\ell\neq j$, let $A[\ell]$ denote the matrix obtained by replacing the $\ell$th column of $A$ with its $j$th column.
Thus, we have $\det(A[\ell]_i)=0$. Expanding $\det(A[\ell]_i)$ along the $\ell$th column as above implies that
\begin{equation}\label{eq130627e}
\sum_{k=1}^{i-1}(-1)^{\ell+k}\det(A_{k,i}^\ell)a_{k,j} 
-\sum_{k=i+1}^{n+1}(-1)^{\ell+k}\det(A^\ell_{i,k})a_{k,j}
=0.
\end{equation}
By definition, we have $e_if_j=0$. Thus,  the Leibniz rule in this case follows from the next computation
where the final equality is a consequence of~\eqref{eq130627d}--\eqref{eq130627e}:
\begin{align*}
\dell_1(e_i)f_j-e_i\dell_2(f_j)
&=(-1)^{i+1}a\det(A_i)f_j-e_i\sum_{k=1}^{n+1}a_{k, j}e_{k} \\
&=(-1)^{i+1}a\det(A_i)f_j-\sum_{k=1}^{n+1}a_{k, j}e_ie_{k} \\
&=(-1)^{i+1}a\det(A_i)f_j+\sum_{k=1}^{i-1}a_{k, j}e_{k}e_i-\sum_{k=i+1}^{n+1}a_{k, j}e_ie_{k} \\
&=(-1)^{i+1}a\det(A_i)f_j+\sum_{k=1}^{i-1}a_{k, j}a\sum_{\ell=1}^n(-1)^{k+i+\ell}\det(A^\ell_{k,i})f_\ell\\
&\hspace{30mm}-\sum_{k=i+1}^{n+1}a_{k, j}a\sum_{\ell=1}^n(-1)^{i+k+\ell}\det(A^\ell_{i,k})f_\ell \\
&=(-1)^{i+1}a\det(A_i)f_j+a(-1)^i\sum_{\ell=1}^n\left[\sum_{k=1}^{i-1}a_{k, j}(-1)^{k+\ell}\det(A^\ell_{k,i})\right.\\
&\hspace{45mm}-\left.\sum_{k=i+1}^{n+1}a_{k, j}(-1)^{k+\ell}\det(A^\ell_{i,k})\right]f_\ell \\
&=0.
\end{align*}
The final case now follows from the previous one:
\begin{align*}
\dell_2(f_j)e_i+f_j\dell_1(e_i)
&=-e_i\dell_2(f_j)+\dell_1(e_i)f_j
=\dell_1(e_i)f_j-e_i\dell_2(f_j)
=0.
\end{align*}
This completes the proof.
\qed
\end{para}

\begin{para}[Sketch of Solution to Exercise~\ref{exer130405b}]
\label{para130405c}
The deleted minimal $R$-free resolution of $R/(x^2,xy,y^2)$ is  the following:
\[
0\to Rf_1\oplus Rf_2\xrightarrow{\left[\begin{smallmatrix}x&0\\y&x\\0&y\end{smallmatrix}\right]} Re_1\oplus Re_2\oplus Re_3\xrightarrow{\left[\begin{smallmatrix}-y^2&xy&-x^2\end{smallmatrix}\right]}R1\to 0.
\]
According to Theorem \ref{thm:codim2dg}, the above complex has a DG $R$-algebra structure with 
\begin{align*}
& e_1e_2=-e_2e_1=yf_1\\
& e_1e_3=-e_3e_1=-xf_1+yf_2\\
& e_2e_3=-e_3e_2=-xf_2
\end{align*}
and $e_i^2=0$ for all $1\leq i\leq 3$.
\qed
\end{para}

\begin{para}[Sketch of Solution to Exercise~\ref{exer130405c}]
\label{para130405d}
First, note that since $A$ is a $3\times 3$ matrix, $\Pf(A^{ijk}_{ijk})=1$ for all choices of $i,j,k$.  Following the conditions specified by Theorem \ref{thm:codim3dg}, one has the relations $e_i^2=0$ for all $1\leq i\leq 3$ and
\begin{align*}
& e_1e_2=-e_2e_1=f_1-f_2+f_3\\
& e_1e_3=-e_3e_1=-f_1-f_2-f_3\\
& e_2e_3=-e_3e_2=f_1-f_2+f_3
\end{align*}
and $e_if_j=f_je_i=\delta_{ij}g$ for all $1\leq i,j\leq 3$.
\qed
\end{para}

\begin{para}[Sketch of Solution to Exercise~\ref{fact110223a}]
\label{para130330a}
\

\eqref{fact110223a1}
To start, we let $f=\{f_j\}\in\Hom[A]MN_q$ and prove that
$\partial^{\Hom MN}_q(f)$ is in $\Hom[A]MN_{q-1}$, that is, that $\partial^{\Hom MN}_q(f)$ is $A$-linear.
For this, let $a\in A_p$ and $m\in M_t$, and compute:
\begin{align*}
\partial^{\Hom MN}_q(f)_{p+t}(am)
&=\partial^N_{p+t+q}(f_{p+t}(am))-(-1)^qf_{p+t-1}(\partial^M_{p+t}(am)) \\
&=(-1)^{pq}\partial^N_{p+t+q}(af_{t}(m))\\
&\quad-(-1)^qf_{p+t-1}\left(\partial^A_{p}(a)m+(-1)^{p}a\partial^M_{t}(m)\right) \\
&=(-1)^{pq}\left(\partial^A_{p}(a)f_{t}(m)+(-1)^{p}a\partial^N_{t+q}(f_{t}(m))\right)\\
&\quad-(-1)^qf_{p+t-1}\left(\partial^A_{p}(a)m\right)-(-1)^{q+p}f_{p+t-1}\left(a\partial^M_{t}(m)\right) \\
&=(-1)^{pq}\partial^A_{p}(a)f_{t}(m)+(-1)^{pq+p}a\partial^N_{t+q}(f_{t}(m))\\
&\quad-(-1)^{q+(p-1)q}\partial^A_{p}(a)f_{t}\left(m\right)-(-1)^{q+p+pq}af_{t-1}\left(\partial^M_{t}(m)\right) \\
&= (-1)^{pq+p}a\partial^N_{t+q}(f_{t}(m))-(-1)^{q+p+pq}af_{t-1}\left(\partial^M_{t}(m)\right)\\
&=(-1)^{p(q-1)}a\left(\partial^N_{t+q}(f_{t}(m))-(-1)^qf_{t-1}(\partial^M_{t}(m))\right) \\
&=(-1)^{p(q-1)}a\partial^{\Hom MN}_q(f)_{p+t}(m)
\end{align*}
The first and last equalities are from the definition of $\partial^{\Hom MN}_q(f)$.
The second equality is from the $A$-linearity of $f$ and the Leibniz Rule on $M$.
The third equality is from   the Leibniz Rule on $N$ and the $A$-linearity of $f$.
The fourth equality is from the $A$-linearity of $f$.
The fifth equality is by cancellation since $q+(p-1)q=pq$.
And the sixth equality is distributivity.

This shows that $\partial^{\Hom[A]MN}$ is well-defined.
Since $\Hom MN$ is an $R$-complex, it follows readily that $\Hom[A]MN$ is also an $R$-complex.

With the same $a$ and $f$ as above, we next show that the sequence $af:=\{(af)_j\}$
defined by the formula
$(af)_t(m):=a (f_t(m))$
is in $\Hom[A]MN_{p+q}$.
First, this rule maps $M_t\to N_{t+p+q}$ since $m\in M_t$ implies that
$f_t(m)\in N_{t+q}$, which implies that $a (f_t(m))\in N_{t+q+p}$.
Next, we show that $af\in\Hom MN_{p+q}$:
\begin{align*}
(af)_t(rm)&=a(f_t(rm))=a(rf_t(m))=(ar)f_t(m)\\
&=(ra)f_t(m)=r(af_t(m))=r(af)_t(m).
\end{align*}
Next, we show that $af$ is $A$-linear. For this, let $b\in A_s$:
\begin{align*}
(af)_{s+t}(bm)
&=a(f_{s+t}(bm))
=(-1)^{qs}a(bf_{s+t}(m))\\
&=(-1)^{qs+ps}b(af_{s+t}(m))
=(-1)^{(q+p)s}b((af)_{s+t}(m)).
\end{align*}
The first and fourth equalities are by definition of $af$.
The second equality follows because $f$ is $A$-linear.
The third equality is from the graded commutativity and associativity of $A$.
Since $|b|=s$ and $|af|=q+p$, this shows that $af$ is $A$-linear.

Next, we verify the DG $A$-module axioms for $\Hom[A]MN$.
The graded axiom has already been verified.
For associativity, continue with the notation from above.
We need to show that $a(bf)=(ab)f$, so we compute:
\begin{align*}
(a(bf))_t(m)
&=a((bf)_t(m))
=a(b(f_t(m)))
=(ab)f_t(m)
=((ab)f)_t(m).
\end{align*}
The third equality is by associativity, and the other equalities are by definition.
Distributivity and unitality are verified similarly.
Thus, it remains to verify the Leibniz Rule.
For this, we need to show that
\begin{equation*}
\partial^{\Hom[A]MN}_{p+q}(af)=\partial^A_p(a)f+(-1)^pa\partial^{\Hom[A]MN}_q(f).
\end{equation*}
For this, we evaluate at $m$:
\begin{align*}
\partial^{\Hom[A]MN}_{p+q}(af)_t(m)\hspace{-10mm}\\
&=\partial^N_{p+q}((af)_t(m))-(-1)^{p+q}(af)_{t-1}(\partial^M_t(m))\\
&=\partial^N_{t+p+q}(a(f_t(m)))-(-1)^{p+q}a(f_{t-1}(\partial^M_t(m)))\\
&=\partial^A_{p}(a)f_t(m)+(-1)^{p}a\partial^N_{t+q}(f_t(m)))-(-1)^{p+q}a(f_{t-1}(\partial^M_t(m)))\\
&=\partial^A_{p}(a)f_t(m)+(-1)^{p}a[\partial^N_{t+q}(f_t(m)))-(-1)^{q}(f_{t-1}(\partial^M_t(m)))]\\
&=\partial^A_{p}(a)f_t(m)+(-1)^{p}a[\partial^{\Hom[A]MN}_q(f)_t(m))]\\
&=(\partial^A_{p}(a)f)_t(m)+(-1)^{p}(a\partial^{\Hom[A]MN}_q(f))_t(m))]\\
&=\left(\partial^A_{p}(a)f+(-1)^{p}a\partial^{\Hom[A]MN}_q(f)\right)_t(m)
\end{align*}
The third equality is by the Leibniz Rule on $N$.
The fourth step is by distributivity.
The remaining equalities are by definition.

\eqref{fact110223a2}
Let $a\in A_p$. From the graded axiom for $M$, we know that the operation $m\mapsto am$
maps $M_t\to M_{p+t}$. The fact that this is $R$-linear follows from associativity:
$$a(rm)=(ar)m=(ra)m=r(am).$$
To show that it is $A$-linear, let $b\in A_s$, and compute:
\begin{align*}
\mu^{M,a}_{s+t}(bm)=a(bm)=(-1)^{ps}b(am)=(-1)^{ps}b\mu^{M,a}_{t}(m).
\end{align*}

\eqref{fact110223a3}
Argue as in the proof of Exercise~\ref{exer120522b} in~\ref{para130218b}.
\qed
\end{para}

\begin{para}[Sketch of Solution to Exercise~\ref{exer130330a}]
\label{para130330b}
We prove that $\Hom[A] Nf$ is a morphism of DG $A$-modules; the argument for $\Hom[A] fN$ is similar.
To this end, first note that since $f$ is a chain map of $R$-complexes, Exercise~\ref{exer130312a} shows that
the induced map $\Hom Nf\colon\Hom NL\to\Hom NM$ is a chain map of $R$-complexes.
Also, note that $\Hom Nf$ and $\Hom[A]Nf$ are given by the same composition-with-$f$ rule.

We need to show that $\Hom[A] fN$ is well-defined.
For this, let $g=\{g_j\}\in\Hom[A] NL_q$. 
We need to show that $\Hom[A] fN(g)\in\Hom[A]NM_q$,
that is, that $\Hom[A] fN(g)$ is $A$-linear.
For this, let  $a\in A_p$ and $n\in N_t$.
We need to show that $\Hom[A] fN(g)_{p+t}(an)=a\Hom[A] fN(g)_{t}(n)$.
We compute:
\begin{align*}
\Hom[A] fN(g)_{p+t}(an)
&=f_{p+q+t}(g_{p+t}(an))\\
&=(-1)^{pq}f_{p+q+t}(ag_{t}(n))\\
&=(-1)^{pq}a[f_{q+t}(g_{t}(n))]\\
&=(-1)^{pq}a[\Hom[A] fN(g)_{t}(n)]\\
&=(-1)^{pq}[a\Hom[A] fN(g)]_{t}(n).
\end{align*}
The second equality is by the $A$-linearity of $g$.
The third equality is by the $A$-linearity of $f$.
The remaining steps are by definition.
since $|\Hom[A] fN(g)|=|g|=q$ and $|a|=p$, this is the desired equality.

Next, we need to show that $\Hom[A] Nf$ respects multiplication on $A$.
For this, we use the same letters as in the previous paragraph.
We need to show that
$$\Hom[A]Nf(ag)=a[\Hom[A]Nf(g)]
$$
so we compute:
\begin{align*}
\Hom[A]Nf(ag)_t(n)
&=f_{t+p+q}((ag)_{t}(n))\\
&=f_{t+p+q}(a(g_{t}(n)))\\
&=a(f_{t+q}(g_{t}(n)))\\
&=a(\Hom[A]Nf(g)_{t}(n))\\
&=a[\Hom[A]Nf(g)]_{t}(n).
\end{align*}
The third equality is by the $A$-linearity of $f$,
and the others are by definition.
\qed
\end{para}

\begin{para}[Sketch of Solution to Exercise~\ref{fact110218b}]
\label{para130330c}
We begin by showing that  $\Otimes[A]MN$ is an $R$-complex, using
Lemma~\ref{lem130401a}\eqref{lem130401a1}.
Thus, we need to show that $U$ is a subcomplex of $\Otimes MN$. 
For this, it suffices to show that the differential $\partial^{\Otimes MN}$ maps each generator of $U$  into $U$.
To this end, let $a\in A_p$, $m\in M_q$ and $n\in N_s$, and compute:
\begin{align*}
\partial^{\Otimes MN}_{p+q+s}(\Otimes[]{(am)}{n}-(-1)^{pq}\Otimes[]{m}{(an)})\hspace{-20mm}\\
&=\partial^{M}_{p+q}(am)\otimes n+(-1)^{p+q}(am)\otimes\partial^N_s(n)\\
&\quad-(-1)^{pq}\partial^M_q(m)\otimes(an)-(-1)^{pq+q}m\otimes\partial^N_{p+s}(an)\\
&=[\partial^{A}_{p}(a)m]\otimes n+(-1)^{p}[a\partial^{M}_{q}(m)]\otimes n\\
&\quad+(-1)^{p+q}(am)\otimes\partial^N_s(n)-(-1)^{pq}\partial^M_q(m)\otimes(an)\\
&\quad-(-1)^{pq+q}m\otimes[\partial^A_p(a)n]-(-1)^{pq+q+p}m\otimes[a\partial^N_s(n)]\\
&=[\partial^{A}_{p}(a)m]\otimes n-(-1)^{pq+q}m\otimes[\partial^A_p(a)n]\\
&\quad+(-1)^{p}[a\partial^{M}_{q}(m)]\otimes n-(-1)^{pq}\partial^M_q(m)\otimes(an)\\
&\quad+(-1)^{p+q}(am)\otimes\partial^N_s(n)-(-1)^{pq+q+p}m\otimes[a\partial^N_s(n)]\\
&=[\partial^{A}_{p}(a)m]\otimes n-(-1)^{(p-1)q}m\otimes[\partial^A_p(a)n]\\
&\quad+(-1)^{p}\left[[a\partial^{M}_{q}(m)]\otimes n-(-1)^{p(q-1)}\partial^M_q(m)\otimes(an)\right]\\
&\quad+(-1)^{p+q}\left[(am)\otimes\partial^N_s(n)-(-1)^{pq}m\otimes[a\partial^N_s(n)]\right].
\end{align*}
Since each of these terms is a multiple of a generator of $U$, we conclude that 
$\partial^{\Otimes MN}_{p+q+s}(\Otimes[]{(am)}{n}-(-1)^{pq}\Otimes[]{m}{(an)})\in U$, as desired.

As in the solution to Exercise~\ref{ex120526b}
contained in~\ref{para130313b}, the $R$-complex $\Otimes MN$ has a well-defined DG $A$-module structure
defined on generators by the formula
$b(m\otimes n):=(bm)\otimes n$.
To show that the same formula is well-defined on $\Otimes[A]MN$, we need to show that
multiplication by $b\in A_t$ maps each generator of $U$ into $U$:
\begin{align*}
b(\Otimes[]{(am)}{n}-(-1)^{pq}\Otimes[]{m}{(an)})
&=(b(am))\otimes n-(-1)^{pq}(bm)\otimes(an)\\
&=(-1)^{pt}(a(bm))\otimes n-(-1)^{pq}(bm)\otimes(an)\\
&=(-1)^{pt}\left[(a(bm))\otimes n-(-1)^{p(q+t)}(bm)\otimes(an)\right].
\end{align*}
Since $|bm|=q+t$ and $|a|=p$, this is in $U$, as desired.

Now that we know that the differential and the scalar multiplication of $A$ on $\Otimes[A]MN$ are
well-defined, the other DG $A$-module axioms are inherited from $\Otimes MN$.
Finally, the formula
$\Otimes[]{(am)}{n}=(-1)^{|a| |m|}\Otimes[]{m}{(an)}$
in $\Otimes[A]MN$ follows from the condition $\Otimes[]{(am)}{n}-(-1)^{|a| |m|}\Otimes[]{m}{(an)}\in U$.
\qed
\end{para}

\begin{para}[Sketch of Solution to Exercise~\ref{disc111223a}]
\label{para130330d}
Let $f\colon A\to B$ be a morphism of DG $R$-algebras.
The DG $R$-algebra $B$ is a DG $A$-module via the scalar multiplication $a_ib:=f_i(a_i)b$,
by Exercise~\ref{ex110218a'}\eqref{ex110218a'3}.
So, we know from Exercise~\ref{fact110218b} that $\Otimes[A]BM$ has a well-defined  DG $A$-module structure.
It remains to show that it has a well-defined  DG $A$-module structure by the action $b(\Otimes[] {b'}m):=\Otimes[]{(bb')}{m}$.
Notice that, once this is shown, the compatibility with the DG $A$-module structure is automatic:
$$a_i(b\otimes m)=(a_ib)\otimes m=(f_i(a_i)b)\otimes m=f_i(a_i)(b\otimes m).$$
Let $U$ be the $R$-submodule of $\Otimes BM$ generated by the elements of the form
$\Otimes[]{(ab)}{m}-(-1)^{|a| |b|}\Otimes[]{b}{(am)}$.

We show that the DG $B$-module structure on $\Otimes[A]BM$ is well-defined. 
For this, note that $\Otimes BM$ has a well-defined DG $B$-module structure via the composition $R\to A\to B$, by Exercise~\ref{ex120526b}.
Thus, it suffices to let $c\in B_p$ and show that multiplication by $c$ maps generators of $U$ into $U$:
\begin{align*}
c(\Otimes[]{(ab)}{m}-(-1)^{|a| |b|}\Otimes[]{b}{(am)})\hspace{-10mm}\\
&=(c(ab))\otimes m-(-1)^{|a| |b|}(cb)\otimes(am)\\
&=(-1)^{|c||a|}(a(cb))\otimes m-(-1)^{|a| |b|}(cb)\otimes(am)\\
&=(-1)^{|c||a|}\left[(a(cb))\otimes m-(-1)^{|a|(|b|+|c|)}(cb)\otimes(am)\right]\\
&=(-1)^{|c||a|}\left[(a(cb))\otimes m-(-1)^{|a|(|cb|)}(cb)\otimes(am)\right]\in U.
\end{align*}
For the DG $B$-module axioms, the only one with any substance is the Leibniz Rule:
\begin{align*}
\partial^{\Otimes[A]BM}_{|c|+|b|+|m|}(c(b\otimes m))\hspace{-10mm}\\
&=\partial^{\Otimes[A]BM}_{|c|+|b|+|m|}((cb)\otimes m) \\
&=\partial^{B}_{|c|+|b|}(cb)\otimes m +(-1)^{|c|+|b|}(cb)\otimes \partial^{M}_{|m|}m) \\
&=(\partial^{B}_{|c|}(c)b)\otimes m+(-1)^{|c|}(c\partial^{B}_{|b|}(b))\otimes m +(-1)^{|c|+|b|}(cb)\otimes \partial^{M}_{|m|}m) \\
&=\partial^{B}_{|c|}(c)(b\otimes m)+(-1)^{|c|}c\left[(\partial^{B}_{|b|}(b))\otimes m +(-1)^{|b|}(b)\otimes \partial^{M}_{|m|}m)\right] \\
&=\partial^{B}_{|c|}(c)(b\otimes m)+(-1)^{|c|}c\partial^{\Otimes[A]BM}_{|c|+|b|+|m|}(b\otimes m).
\end{align*}
The third equality follows from the Leibniz Rule for $B$, and the fourth equality is by distributivity.
The remaining equalities are by definition.
\qed
\end{para}

\begin{para}[Sketch of Solution to Exercise~\ref{fact120526a}]
\label{para130330e}
\

Hom cancellation. 
For each $f=\{f_p\}\in\Hom[A]AL_i$ we have $f_0\colon A_0\to L_i$,
hence $f_0(1_A)\in L_i$. Define $\alpha_i\colon\Hom[A]AL_i\to L_i$ by the
formula $\alpha_i(f):=f_0(1_A)$. We  show that $\alpha\colon\Hom[A]AL\to L$ is 
a morphism of DG $A$-modules and that it is bijective, by Remark~\ref{ex110218a}.
To show that it is a chain map over $R$, we compute:
\begin{align*}
\alpha_{i-1}(\partial^{\Hom[A]AL}_i(f))
&=\alpha_{i-1}(\{\partial^Y_{p+i}f_p-(-1)^if_{p-1}\partial^X_p\})\\
&=\partial^L_{i}(f_0(1_A))-(-1)^if_{-1}\partial^A_0(1_A)\\
&=\partial^L_{i}(f_0(1_A))\\
&=\partial^L_{i}(\alpha_i(f)).
\end{align*}
To show that $\alpha$ is $A$-linear, let $a\in A_j$ and
compute:
\begin{align*}
\alpha_{i+j}(af)
&=(af)_0(1_A)
=af_0(1_A)
=a\alpha_i(f).
\end{align*}
To see that $\alpha$ is injective, suppose that $0=\alpha_i(f)=f_0(1_A)$.
It follows that for all $a\in A_j$ we have
$$f_j(a)=f_j(a1_A)=af_0(1_A)=a\cdot 0=0.
$$
We conclude that $\alpha$ is injective. 
To show that $\alpha$ is surjective, let $x\in L_i$. 
As in the proof of Exercise~\ref{fact110223a}\eqref{fact110223a2}
in~\ref{para130330a}, the map $\nu^x\colon A\to L$ given by $a\mapsto ax$
is a homomorphism of degree $i$. Moreover, we have
$$\alpha_i(\nu^x)=\nu^x_0(1_A)=1_Ax=x$$
so $\alpha$ is surjective, as desired.

Notice that the special case $L=A$ explains the isomorphism 
$\Hom[A]AA\cong A$.

Tensor cancellation.
Define $\eta\colon \Otimes AL\to L$ by the formula $\eta_{i+j}(a_i\otimes x_j):=a_ix_j$.
This is a well-defined chain map by Exercise~\ref{fact110216a'}.
Let $U$ be the $R$-submodule of $\Otimes AL$ generated by the elements of the form
$\Otimes[]{(a_ib_j)}{x_k}-(-1)^{ij}\Otimes[]{b_j}{(a_ix_k)}$.
For each such generator, we have
\begin{align*}
\eta_{i+j+k}(\Otimes[]{(a_ib_j)}{x_k}-(-1)^{ij}\Otimes[]{b_j}{(a_ix_k)})
&=(a_ib_j)x_k-(-1)^{ij}b_j(a_ix_k)=0
\end{align*}
since $b_ja_i=(-1)^{ij}a_ib_j$.
It follows that $\eta$ induces a well-defined map $\nu\colon\Otimes[A]AL\to L$ given by
$\nu_{i+j}(a_i\otimes x_j):=a_ix_j$. Since $\eta$ is a chain map, it follows readily that $\nu$ is
also a chain map. Moreover, it is $A$-linear because
\begin{align*}
\nu_{i+j+k}(a_i(b_j\otimes x_k))
\!=\!\nu_{i+j+k}((a_ib_j)\otimes x_k)
=(a_ib_j)x_k
=a_i(b_jx_k)
=a_i\nu_{j+k}(b_j\otimes x_k).
\end{align*}
To show that $\nu$ is an isomorphism, we construct a two-sided inverse.
Let $\beta\colon L\to \Otimes[A]AL$ be given by $\beta_i(x_i)=1\otimes x_i$.
As in previous exercises, this is a well-defined morphism of DG $A$-modules.
To see that it is a two-sided inverse for $\nu$, we compute:
\begin{align*}
\beta_{i+j}(\nu_{i+j}(a_i\otimes x_j))
&=1\otimes(a_ix_j)=a_i\otimes x_j.
\end{align*}
This shows that $\beta\nu$ is the identity on $\Otimes[A]AL$.
The fact that $\nu\beta$ is the identity on $L$ is even easier.

Again, the special case $L=A$ explains the isomorphism 
$\Otimes[A]AA\cong A$.

Tensor commutativity.
By Exercise~\ref{exer120522k}\eqref{exer120522k3}, the map
$\Otimes LM\xra\gamma\Otimes ML$
given by $x_i\otimes y_j\mapsto(-1)^{ij}y_j\otimes x_i$ is a well-defined isomorphism of $R$-complexes.
Let $V$ be the submodule of $\Otimes LM$ generated over $R$ by the elements of the form
$\Otimes[]{(a_ix_j)}{y_k}-(-1)^{ij}\Otimes[]{x_j}{(a_iy_k)}$.
Let $W$ be the $R$-submodule of $\Otimes ML$ generated by the elements of the form
$\Otimes[]{(a_iy_j)}{x_k}-(-1)^{ij}\Otimes[]{y_j}{(a_ix_k)}$.

For each element $(a_ix_j)\otimes y_k-(-1)^{ij}x_j\otimes (a_iy_k)\in U$, we have
\begin{align*}
\gamma_{i+j+k}((a_ix_j)\otimes y_k-(-1)^{ij}x_j\otimes (a_iy_k))\hspace{-20mm}\\
&=(-1)^{(i+j)k}y_k\otimes(a_ix_j)-(-1)^{ij+(i+k)j}(a_iy_k)\otimes x_j\\
&=(-1)^{(i+j)k}y_k\otimes(a_ix_j)-(-1)^{jk}(a_iy_k)\otimes x_j\\
&=-(-1)^{jk}[(a_iy_k)\otimes x_j-(-1)^{ik}y_k\otimes(a_ix_j)]\in W.
\end{align*}
It follows that $\gamma$ factors through the natural epimorphisms
$\Otimes LM\to\Otimes[A] LM$ and $\Otimes ML\to\Otimes[A] ML$,
that is, the map $\ol\gamma\colon\Otimes[A] LM\to\Otimes[A] ML$
given by $x_i\otimes y_j\mapsto(-1)^{ij}y_j\otimes x_i$ is well-defined.
To show that $\ol\gamma$ is $A$-linear, we compute:
\begin{align*}
\ol\gamma_{i+j+k}(a_i(x_j\otimes y_k))
&=\ol\gamma_{i+j+k}((a_ix_j)\otimes y_k)\\
&=(-1)^{(i+j)k}y_k\otimes (a_ix_j)\\
&=(-1)^{(i+j)k+ik}(a_iy_k)\otimes x_j\\
&=(-1)^{jk}a_i(y_k\otimes x_j)\\
&=a_i\ol\gamma_{j+k}(x_j\otimes y_k).
\end{align*}
Similarly, the map $\ol\delta\colon\Otimes[A] ML\to\Otimes[A] LM$
given by $y_j\otimes x_i\mapsto(-1)^{ij}x_i\otimes y_j$ is well-defined
and $A$-linear. It is straightforward to show that the compositions $\ol\gamma\ol\delta$
and $\ol\delta\ol \gamma$ are the respective identities, so that $\ol\gamma$ is the desired
isomorphism.
\qed
\end{para}

\begin{para}[Sketch of Solution to Exercise~\ref{exer130330b}]
\label{para130330f}
For the map $\Otimes[A] Nf$,
let $U$ be the $R$-submodule of $\Otimes NL$ generated by the elements of the form
$\Otimes[]{(a_ix_j)}{y_k}-(-1)^{ij}\Otimes[]{x_j}{(a_iy_k)}$,
and let $V$ be the $R$-submodule of $\Otimes NM$ generated by the elements of the form
$\Otimes[]{(a_ix_j)}{y_k}-(-1)^{ij}\Otimes[]{x_j}{(a_iy_k)}$.
To show that $\Otimes[A]Nf$ is well-defined, it suffices to show that
the chain map $\Otimes Nf\colon\Otimes NL\to\Otimes NM$ sends 
every generator of $U$ into $V$:
\begin{align*}
(\Otimes Nf)_{i+j+k}(\Otimes[]{(a_ix_j)}{y_k}-(-1)^{ij}\Otimes[]{x_j}{(a_iy_k)})\hspace{-20mm}\\
&=(a_ix_j)\otimes f_k(y_k)-(-1)^{ij}x_j\otimes f_{i+k}(a_iy_k)\\
&=(a_ix_j)\otimes f_k(y_k)-(-1)^{ij}x_j\otimes (a_if_k(y_k))\in V.
\end{align*}
To show that $\Otimes[A] Nf$ is $A$-linear, we compute similarly:
\begin{align*}
(\Otimes[A] Nf)_{i+j+k}(a_i(x_j\otimes y_k))
&=(\Otimes[A] Nf)_{i+j+k}((a_ix_j)\otimes y_k))\\
&=(a_ix_j)\otimes f_k(y_k)\\
&=a_i(x_j\otimes f_k(y_k))\\
&=a_i(\Otimes[A] Nf)_{j+k}(x_j\otimes y_k)
\end{align*}
The map $\Otimes fN$ is treated similarly.
\qed
\end{para}

\begin{para}[Sketch of Solution to Exercise~\ref{ex120524a}]
\label{para130330g}
We are working over $R$ as a DG $R$-algebra, which has $R^\natural=R$.
Since $R$ is local, we know that a direct sum $\bigoplus_iM_i$ of $R$-modules is free
if and only if each $M_i$ is free. (In general, the $M_i$ are projective; since $R$ is local,
we know that projective implies free.)
If $L$ is a semi-free DG $R$-module, then it is bounded below by definition, 
and the module $\bigoplus_iL_i$ is free over $R$, so each $L_i$ is free, as desired.
The converse is handled similarly.

If $F$ is a free resolution of $M$, then the previous paragraph implies that $F$ is semi-free.
Exercise~\ref{exer120522g} implies that there is a
quasiisomorphism $F\xra\simeq M$ over $R$, so this is a semi-free resolution by definition.
\qed
\end{para}

\begin{para}[Sketch of Solution to Exercise~\ref{ex120524b}]
\label{para130330h}
It is straightforward to show that $M$ is exact (as an $R$-complex) if and only if
the natural map $0\to M$ is a quasiisomorphism, since the induced map on homology
is the natural map $0\to\HH_i(M)$. Exercise~\ref{ex120524a} implies that
$0$ is semi-free, so the map $0\to M$ is a quasiisomorphism if and only if it
is a semi-free resolution.

Since $A$ is bounded below, so are $\shift^nA$ and $\bigoplus_{n\geq n_0}\shift^nA^{\beta_n}$.
To show that $\shift^nA$ is semi-free, we need to show that $1_A\in(\shift^nA)_n$ is a semibasis.
The only subtlety here is in the signs. 
If $n$ is odd, then we have
\begin{align*}
\textstyle(\sum_ia_i)\ast 1_A
&\textstyle=\sum_i(-1)^ia_i \\
\textstyle(\sum_i(-1)^ia_i)\ast 1_A
&\textstyle=\sum_ia_i.
\end{align*}
The first of these shows that $1_A$ is linearly independent: if $(\sum_ia_i)\ast 1_A=0$,
then $\sum_i(-1)^ia_i=0$ so $a_i=0$ for all $i$, which implies that $\sum_ia_i=0$.
The second of these shows that $1_A$ spans $\und A$ over $\und A$:
for all $\sum_ia_i\in\und A$, we have $\sum_ia_i=(\sum_i(-1)^ia_i)\ast 1_A\in\und A\cdot1_A$.
If $n$ is even, then the relevant formula is
$(\sum_ia_i)\ast 1_A=\sum_ia_i$.

To show that $\shift^nA^{\beta_n}$ is semi-free, use the previous paragraph to show that the sequence
of standard basis vectors
$(1_A,0,\ldots,0),(0,1_A,\ldots),\ldots,(0,0,\ldots,1_A)$ form a semibasis.
To show that $\bigoplus_{n\geq n_0}\shift^nA^{\beta_n}$ is semi-free, let $E_n$ be a semibasis for each 
$\shift^nA^{\beta_n}$ and show that $\cup_nE_n$ is a semibasis for $\bigoplus_{n\geq n_0}\shift^nA^{\beta_n}$.
\qed
\end{para}

\begin{para}[Sketch of Solution to Exercise~\ref{exer130330c}]
\label{para130330i}
\

\eqref{exer130330c1}
Exercise~\ref{disc111223a} shows that $\Otimes{K}F$ is a  DG $K$-module,
so we only need to show that it is semi-free.
For each $i\in\bbz$, let $E_i$ be a basis of the free $R$-module $F_i$,
and set $E_i'=\{1_K\otimes e\in\Otimes KF\mid e\in E_i\}$.
We claim that $E':=\cup_iE_i'$ is a semibasis for $\Otimes KF$.
To show that $E'$ spans $\und{(\Otimes KF)}$, it suffices to show that for
each $x\in K_i$ and each $y\in F_j$ the
generator $x\otimes y\in(\Otimes KF)_{i+j}$ is in the $K$-span of $E'$.
For this, write $y=\sum_{e\in E_j}r_ee$, and compute:
$$\textstyle
x\otimes y=x\otimes (\sum_{e\in E_j}r_ee)=\sum_{e\in E_j}r_ex(1_K\otimes e)\in K\cdot E'. 
$$
To show that $E'$ is linearly independent takes a bit of bookkeeping.
Suppose that 
\begin{equation}\label{eq130404a}
0=\sum_{i=1}^mx_i(1_K\otimes e_i)
\end{equation}
in $\Otimes[A]KF$ for some $x_i\in \und K$ and distinct elements $e_1,\ldots,e_m\in E$.
Since $\und{(\Otimes KF)}$ is a graded $\und K$-module, we may assume without loss of generality
that each $x_i$ is homogeneous and that the degree $|x_i\otimes e_i|=|x_i|+|e_i|=n$ is the same for all $i$. 
Moreover, $\und{(\Otimes KF)}$ is a bi-graded $\und K$-module (with gradings coming from $K$ and $F$)
so we may assume without loss of generality that the degree $|x_i|=p$ is the same for all $i$
and that the degree $|e_i|=q$ is the same for all $i$.
Thus, equation~\eqref{eq130404a} becomes
$0=\sum_{i=1}^mx_i\otimes e_i$. This sum occurs in the following submodule of $\Otimes{K_p}{F_q}$:
$$\bigoplus_{i=1}^m\Otimes{K_p}{Re_i}\cong \bigoplus_{i=1}^mK_p.$$
Under this isomorphism, the element $0=\sum_{i=1}^mx_i\otimes e_i$ corresponds to the 
vector
$0=(x_1,\ldots,x_m)$
which implies that $x_i=0$ for $i=1,\ldots,m$. It follows that $E'$ is linearly independent, as desired.

\eqref{exer130330c2}
Let $F\xra\simeq M$ be a semi-free resolution of a DG $R$-module $M$.
Part~\eqref{exer130330c1} implies that $\Otimes{K}F$ is semi-free over $K$.
Since $K$ is a bounded below complex of projective $R$-modules,
Fact~\ref{fact120523a}
implies that the induced map $\Otimes{K}F\xra\simeq \Otimes{K}M$ is a quasiisomoprhism, so it is a
semi-free resolution by definition.
\qed
\end{para}

\begin{para}[Sketch of Solution to Exercise~\ref{exer130331b}]
\label{para130331b}
To show that  $\chi^A_M\colon A\to \HomA MM$ has the desired properties,
we first note that Exercise~\ref{fact110223a}\eqref{fact110223a2} shows that it maps $A_i$ to $\HomA MM_i$ for all $i$. 
Next, we check that $\chi^A_M$ is a chain map:
\begin{align*}
\partial^{\Hom MM}_i((\chi^A_M)_i(a_i))
&=\{\partial^M_{i+p}\mult M{a_i}_p-(-1)^i\mult M{a_i}_{p-1}\partial^M_p\}\\
(\chi^A_M)_{i-1}(\partial^A_i(a_i))
&=\{\mult M{\partial^A_i(a_i)}_p\}.
\end{align*}
To see that these are equal, we evaluate at $m_p\in M_p$:
\begin{align*}
\partial^M_{i+p}(\mult M{a_i}_p(m_p))-(-1)^i\mult M{a_i}_{p-1}(\partial^M_p(m_p))\hspace{-20mm}
\\
&=\partial^M_{i+p}(a_im_p)-(-1)^ia_i\partial^M_p(m_p)
\\
&=\partial^A_{i}(a_i)m_p+(-1)^{i}a_i\partial^M_{p}(m_p)-(-1)^ia_i\partial^M_p(m_p)
\\
&=\partial^A_{i}(a_i)m_p\\
&=\mult M{\partial^A_i(a_i)}_p(m_p)
\end{align*}
To complete the proof, we check that $\chi^A_M$ is a $A$-linear.
For this, we need to show that
$(\chi^A_M)_{i+j}(a_ib_j)=a_i(\chi^A_M)_{j}(b_j)$. To show this, we evaluate at $m_p\in M_p$:
\begin{align*}
(\chi^A_M)_{i+j}(a_ib_j)_p(m_p)
&=a_ib_jm_p
=a_i(\chi^A_M)_{j}(b_j)_p(m_p)
\end{align*}
as desired.
\qed
\end{para}

\begin{para}[Sketch of Solution to Exercise~\ref{fact120524a}]
\label{para130330j}
Let $M$ and $N$ be $R$-modules.
Exercise~\ref{ex120524a} implies that
each free resolution $F$ of  $M$
gives rise  to a semi-free resolution $F\xra\simeq M$.
Thus,  the module $\Ext iMN$ 
defined in~\ref{defn110223a'} is $\HH_{-i}(\Hom FM)$, which is the usual $\Ext iMN$.
\qed
\end{para}

\begin{para}[Sketch of Solution to Exercise~\ref{ex120524e'}]
\label{para130330m}
We begin with the graded vector space $W''=0\bigoplus Fw_2\bigoplus Fw_1\bigoplus Fw_0\bigoplus 0$.
The differential $\partial''$ consists of two matrices of size $1\times 1$:
\begin{equation*}
0\to Fw_2\xra{x_2} Fw_1\xra{x_1} Fw_0\to 0.
\end{equation*}
The condition $\partial''_{i-1}\partial''_{i}=0$ is only non-trivial for $i=2$, in which case
it boils down to the following:
$$0=\partial''_1(\partial''_2(w_2))=\partial''_1(x_2w_1)=x_1x_2w_0.$$
We conclude that $(W'',\partial'')$ is an $R$-complex if and only if
\begin{equation}
\label{eq130405a}
x_1x_2=0.
\end{equation}
The scalar multiplication of $U$ on $W''$ is completely described
by specifying $ew_0$ and $ew_1$, and this requires two more elements $y_0,y_1\in F$
so that we have $ew_0=y_0w_1$ and $ew_1=y_1w_2$.
The associative law (which was not a concern for $W$ and $W'$) says that we must have
$$0=0w_0=e^2w_0=e(ew_0)=e(y_0w_1)=y_0y_1w_2$$
so we conclude that 
\begin{equation}
\label{eq130405b}
y_0y_1=0.
\end{equation}
Note that once this is satisfied, the general associative law follows.
This leaves the Leibniz Rule for the products $ew_0$, $ew_1$, and $ew_2$.
We begin with $ew_0$:
\begin{align*}
\partial''_1(ew_0)
&=\partial''^U_1(e)w_0+(-1)^{|e|}e\partial''_0(w_0) \\
\partial''_1(y_0w_1)
&=0w_0+(-1)^{|e|}e0 \\
x_1y_0w_0
&=0
\end{align*}
which implies that
\begin{equation}
\label{eq130405c}
x_1y_0=0.
\end{equation}
A similar computation for  $ew_2$
shows that
\begin{align}
\label{eq130405d}
x_2y_1=0.
\end{align}
A last computation for $ew_1$ 
yields $x_2y_1+x_1y_0=0$, which is redundant because of
equations~\eqref{eq130405c}--\eqref{eq130405d}.
Thus, $\od^U(W'')$ consists of all ordered quadruplets $(x_1,x_2,y_0,y_1)\in\bba^4_F$ satisfying 
the equations~\eqref{eq130405a}--\eqref{eq130405d}.
It is possibly worth noting that the ideal
defined by equations~\eqref{eq130405a}--\eqref{eq130405d}
has a simple primary decomposition:
$$(x_1x_2,y_0y_1,x_1y_0,x_2y_1)=(x_1y_1)\cap(x_2,y_0).$$

Next, we repeat this process for
$W'''= 0\bigoplus Fz_2\bigoplus (Fz_{1,1}\bigoplus Fz_{1,2})\bigoplus Fz_0\bigoplus 0$.
The differential $\partial'''$ in this case has the following form:
$$0\to Fz_2\xra{\protect{\left(\begin{smallmatrix}a_{2,1}\\a_{2,2}\end{smallmatrix}\right)}} Fz_{1,1}\bigoplus Fz_{1,2}
\xra{(a_{1,1}\, \, \, a_{1,2})} Fz_0\to 0$$
meaning that $\partial'''_2(z_2)=a_{2,1}z_{1,1}+a_{2,2}z_{1,2}$
and $\partial'''_1(z_{1,i})=a_{1,i}z_0$ for $i=1,2$.
Scalar multiplication also requires more letters:
\begin{align*}
ez_0&=b_{0,1}z_{1,1}+b_{0,2}z_{1,2}
\\
ez_{1,1}
&=b_{1,1}z_2
\\
ez_{1,2}
&=b_{1,2}z_2.
\end{align*}
The condition $\partial'''_{i-1}\partial'''_{i}=0$
is equivalent to the following  equation:
\begin{align}
\label{eq130405e}
a_{1,1}a_{2,1}+a_{1,2}a_{2,2}=0.
\end{align}
The associative law is equivalent to the next equation:
\begin{align}
\label{eq130405f}
b_{0,1}b_{1,1}+b_{0,2}b_{1,2}=0.
\end{align}
For the Leibniz Rule, we need to consider the products $ez_0$, $ez_{1,j}$ and $ez_2$,
so this axiom is equivalent to the following equations:
\begin{align}
\label{eq130405g}
a_{1,1}b_{0,1}+a_{1,2}b_{0,2}
&=0\\
\label{eq130405h}
a_{2,i}b_{1,j}+a_{1,j}b_{0,i}
&=0
&\text{for all $i=1,2$ and $j=1,2$}
\\
\label{eq130405i}
a_{2,1}b_{1,1}+a_{2,2}b_{1,2}
&=0.
\end{align}
So, $\od^U(W'')$ consists of all  $(a_{1,1}, a_{1,2}, a_{2,1}, a_{2,2}, b_{0,1}, b_{0,2}, b_{1,1}, b_{1,2})\in\bba^8_F$ satisfying 
the equations~\eqref{eq130405e}--\eqref{eq130405i}.
\qed
\end{para}

\begin{para}[Sketch of Solution to Exercise~\ref{ex120524e''}]
\label{para130330n}
We recall that
\begin{align*}
W''&=\quad 0\bigoplus Fw_2\bigoplus Fw_1\bigoplus Fw_0\bigoplus 0\\
W'''&=\quad 0\bigoplus Fz_2\bigoplus (Fz_{1,1}\bigoplus Fz_{1,2})\bigoplus Fz_0\bigoplus 0.
\end{align*}
It follows that
\begin{align*}
\hspace{8mm}&&&&\en_F(W'')_0&=\bigoplus_{i=0}^2\Hom[F]{Fw_i}{Fw_i}\cong F^3=\bba^3_F\\
&&&&\ggl_F(W'')_0&=\bigoplus_{i=0}^2\au_F(Fw_i) \cong (F^\times)^3
=U_{u_2u_1u_0}\subset\bba^3_F\\
&&&&\en_F(W''')_0&\cong\Hom[F]{F}{F}\bigoplus\Hom[F]{F^{2}}{F^{2}}\bigoplus\Hom[F]{F}{F}\\
&&&&&\cong F\times F^4\times F=\bba^6_F\\
&&&&\ggl_F(W''')_0&=\au_F(F)\bigoplus\au_F(F^2)\bigoplus\au_F(F)\\
&&&&&\cong F^\times\times \operatorname{gl}_2(F) \times F^\times
=U_{c_2(c_{11}c_{22}-c_{12}c_{21})c_0}\subset\bba^6_F.
&&&&\qed
\end{align*}
\end{para}

\begin{para}[Sketch of Solution to Exercise~\ref{exer120526a}]
\label{para130330o}
We continue with the notation of Example~\ref{ex120524e}
and the solutions to Exercises~\ref{ex120524e'} and~\ref{ex120524e''}.
Under the isomorphism $\ggl_F(W'')_0\cong U_{u_2u_1u_0}\subseteq\bba^3_F$,
an ordered triple $(u_0,u_1,u_2)\in U_{u_2u_1u_0}$ corresponds to the isomorphism
$$\xymatrix{
\partial:\ar[d]_-\alpha^-\cong&
0\ar[r]
&Fw_2\ar[r]^-{x_2}\ar[d]_-{u_2}
&Fw_1\ar[r]^-{x_1}\ar[d]_-{u_1}
&Fw_0\ar[r]\ar[d]_-{u_0}
&0
\\
\wti\partial:&
0\ar[r]
&F\wti w_2\ar[r]^-{\wti x_2}
&F\wti w_1\ar[r]^-{\wti x_1}
&F\wti w_0\ar[r]
&0.}$$
Let $e\cdot_{\alpha}\wti w_j=\wti y_j\wti w_{j+1}$ for $j=0,1$. Then direct computations as in
Example~\ref{ex120524e'''} show that
$\wti x_i=u_{i-1}x_iu_i^{-1}$ for $i=1,2$ and 
$\wti y_j=u_{j+1}y_ju_j^{-1}$ for
$j=0,1$.

Under the isomorphism $\ggl_F(W''')_0\cong U_{c_2(c_{11}c_{22}-c_{12}c_{21})c_0}\subset\bba^6_F$,
an ordered sextuple $(c_2,c_{11},c_{22},c_{12},c_{21},c_0)\in U_{c_2(c_{11}c_{22}-c_{12}c_{21})c_0}$ corresponds to the isomorphism
$$\xymatrix@C=12mm{
\partial:\ar[d]_-\alpha^-\cong&
0\ar[r]
&Fz_2\ar[r]^-{\protect{\left(\begin{smallmatrix}a_{2,1}\\a_{2,2}\end{smallmatrix}\right)}}\ar[d]_-{c_2}
&Fz_{1,1}\oplus Fz_{1,2}\ar[r]^-{(a_{1,1}\, \, \, a_{1,2})}\ar[d]_-{\left(\begin{smallmatrix}c_{11}&c_{12}\\c_{21}&c_{22}\end{smallmatrix}\right)}
&Fz_0\ar[r]\ar[d]_-{c_0}
&0
\\
\wti\partial:&
0\ar[r]
&F\wti z_2\ar[r]_-{\protect{\left(\begin{smallmatrix}\wti a_{2,1}\\ \wti a_{2,2}\end{smallmatrix}\right)}}
&F\wti z_{1,1}\oplus F\wti z_{1,2}\ar[r]_-{(\wti a_{1,1}\, \, \, \wti a_{1,2})}
&F\wti z_0\ar[r]
&0.}$$
Set $\Delta=\det\left(\begin{smallmatrix}c_{11}&c_{12}\\c_{21}&c_{22}\end{smallmatrix}\right)=c_{11}c_{22}-c_{12}c_{21}$.
Thus, we have the following:
\begin{align*}
\left(\begin{matrix}\wti a_{2,1}\\ \wti a_{2,2}\end{matrix}\right)
&=\left(\begin{matrix}c_{11}&c_{12}\\c_{21}&c_{22}\end{matrix}\right)\left(\begin{matrix}\wti a_{2,1}\\ \wti a_{2,2}\end{matrix}\right)(c_2^{-1})
=\left(\begin{matrix}(c_{11}a_{2,1}+c_{12}a_{2,2})c_2^{-1}\\(c_{21}a_{2,1}+c_{22}a_{2,2})c_2^{-1}\end{matrix}\right)
\\
\left(\begin{matrix}\wti a_{1,1}& \wti a_{1,2}\end{matrix}\right)
&=(c_0)\left(\begin{matrix}a_{2,1}\\a_{2,2}\end{matrix}\right)\left(\begin{matrix}c_{11}&c_{12}\\c_{21}&c_{22}\end{matrix}\right)^{-1}
\\
&=\left(\begin{matrix}\Delta^{-1}c_0(a_{1,1}c_{22}-a_{1,2}c_{21})& \Delta^{-1}c_0(-a_{1,1}c_{12}+a_{1,2}c_{11})\end{matrix}\right).
\end{align*}
In other words, we have
\begin{align*}
\wti a_{2,i} 
&=(c_{i1}a_{2,1}+c_{i2}a_{2,2})c_2^{-1}
\\
\wti a_{1,1}
&=\Delta^{-1}c_0(a_{1,1}c_{22}-a_{1,2}c_{21})
\\
\wti a_{1,2}
&=\Delta^{-1}c_0(-a_{1,1}c_{12}+a_{1,2}c_{11})
\end{align*}
for $i=1,2$.
For the scalar multiplication, we have
\begin{align*}
e\cdot_{\alpha}\wti z_2
&=0 
\end{align*}
and using the rule $e\cdot_{\alpha}\wti z_*=\alpha_2(e\alpha_1^{-1}(z_*))$, we find that
\begin{align*}
e\cdot_{\alpha}\wti z_{1,1}
&=c_2\Delta^{-1}(c_{22}b_{1,1}-c_{21}b_{1,2})\wti z_2
\\
e\cdot_{\alpha}\wti z_{1,2}
&=c_2\Delta^{-1}(-c_{12}b_{1,1}+c_{11}b_{1,2})\wti z_2
\\
e\cdot_{\alpha}\wti z_{0}
&=c_0^{-1}(c_{11}b_{0,1}+c_{12}b_{0,2})\wti z_{1,1}+c_0^{-1}(c_{21}b_{0,1}+c_{22}b_{0,2})\wti z_{1,2}.
\end{align*}
In other words, we have
\begin{align*}
\wti b_{1,1}
&=c_2\Delta^{-1}(c_{22}b_{1,1}-c_{21}b_{1,2})
\\
\wti b_{1,2}
&=c_2\Delta^{-1}(-c_{12}b_{1,1}+c_{11}b_{1,2})\wti z_2
\\
\wti b_{0,i}
&=c_0^{-1}(c_{i1}b_{0,1}+c_{i2}b_{0,2})
\end{align*}
for $i=1,2$.
\qed
\end{para}

\providecommand{\bysame}{\leavevmode\hbox to3em{\hrulefill}\thinspace}
\providecommand{\MR}{\relax\ifhmode\unskip\space\fi MR }
\providecommand{\MRhref}[2]{%
  \href{http://www.ams.org/mathscinet-getitem?mr=#1}{#2}
}
\providecommand{\href}[2]{#2}

\end{document}